\documentclass[12pt,reqno]{amsart}
\usepackage{amssymb,amsmath,epsfig,mathrsfs, enumerate, xparse, mathtools}
\usepackage[pagewise]{lineno}
\usepackage{graphicx}\usepackage[normalem]{ulem}
\usepackage{fancyhdr}
\usepackage[margin=2.5cm]{geometry}
\usepackage{hyperref}
\hypersetup{
 colorlinks   = true,
 urlcolor     = blue,
 linkcolor    = blue,
 citecolor   = red ,
 bookmarksopen=true
}
 \usepackage[usenames]{color}
 \usepackage{lastpage}
 \usepackage{url}
 \newcommand{\mb}{\mathbb}
 \newcommand{\mc}{\mathcal}
 
 \newcommand{\f}{\frac}
 \newcommand{\ld}{\lambda}

 \newcommand{\rt}{\sqrt}
 \newcommand{\fa}{\forall}

 \newcommand{\dd}{\partial}
 \newcommand{\tb}{\textbf}
 
 \newcommand{\nn}{\nonumber}
 \newcommand{\ep}{\epsilon}
 \newcommand{\qd}{\quad}

 \newcommand{\Ld}{\Lambda}
 
 \newcommand{\iy}{\infty}

 \newcommand{\tn}{\textnormal}
 \newcommand{\py}{  \partial_{x_{n+1}}^a} 
 \usepackage[notref,notcite]{}
 \usepackage[pagewise]{lineno}
 \newtheorem{theorem}{Theorem}[section]
 \newtheorem{lemma}[theorem]{Lemma}
 \newtheorem{proposition}[theorem]{Proposition}
 
 \newtheorem{definition}[theorem]{Definition}

 \numberwithin{equation}{section}

 \usepackage{hyperref}

 \newcommand{\PD}{\partial}

 \newcommand{\Rn}{\mathbb{R}^n}

 \newcommand{\bp}{\begin{prob}}
 	\newcommand{\bpr}{\begin{proof}}
 		\newcommand{\epr}{\end{proof}}

 	\renewcommand{\d}{\mathrm{d}}

 	\newcommand{\Rb}{\mathbb{R}}

 	\newcommand{\Hc}{\mathcal{H}}

 	\renewcommand{\O}{\Omega}
	\newcommand{\ch}{\mb{H}^s_{\overline{Q}}}
	\newcommand{\chd}{(\mb{H}^s_{\overline{Q}})^* }	\usepackage{amsmath}
 	\theoremstyle{definition}

 	\title[A fractional parabolic inverse problem]{The Calder\'on problem for  space-time fractional parabolic operators with variable coefficients}  
 	
 	\author[Banerjee, Krishnan and Senapati] {Agnid Banerjee$^{*}$ and Soumen Senapati$^{*}$}
 	\address{$^{*}$	TIFR Centre for Applicable Mathematics, Bangalore 560065, India. 
		\newline\indent E-mail:{\tt \ agnid@tifrbng.res.in, soumen@tifrbng.res.in }}
\thanks{AB is supported  by Department of Atomic Energy,  Government of India, under
Project No.  12-R \& D-TFR-5.01-0520.}

\begin{document}
		\begin{abstract}
		We study an inverse problem for  variable coefficient fractional parabolic operators of the form  $(\PD_t -\mbox{div}(A(x) \nabla_x))^s + q(x,t)$ for $s\in(0,1)$ and show the unique recovery of $q$ from exterior measured data. Similar to the fractional elliptic case, we use Runge type approximation  argument which is obtained via  a global weak unique continuation property. The proof of such a unique continuation result   involves a new Carleman estimate for the associated variable coefficient extension operator. In the latter part of the work, we  prove analogous unique determination results for fractional parabolic operators with  drift.
	\end{abstract}

  \maketitle

  \tableofcontents

 	\section{Introduction and statement of the main results}
 	
 	Let $\O$ be a domain in $\Rb^n$ and let $T>0$. Let $A(x)$ be a positive definite $n\times n$ matrix on $\O$ with Lipschitz coefficients. We denote by  $\Hc=\PD_t -\mbox{div}(A(x) \nabla_x)$ the parabolic operator in $\Rb^{n+1}$, and for $s\in(0,1)$, by $\Hc^{s}$ the fractional parabolic operator. In this article, we study two inverse problems associated to this fractional parabolic operator, which we now proceed to describe precisely.  
 	
 	Let us denote the cylindrical domain $\O\times (-T,T)$ by $Q$ and the exterior domain $\O_{e}\times (-T,T)$ by $Q_{e}$ where $\O_{e}=\Rb^{n}\setminus \overline{\O}$.
 	
 	Let the potential term $q\in L^{\infty}(Q)$. We consider the initial-exterior problem
 	\begin{align}\label{ini-ext prob}
 		\begin{cases}
 			\left( \mc{H}^s + q(x,t) \right)u = 0, & \tn{ in } Q\\
 			u(x,t) = f(x,t), & \tn{ in } Q_e \\ 
 			u(x,t) = 0, & \tn{ for } t\leq -T.
 		\end{cases}
 	\end{align}
 	%
 	We will assume that
 	
 	\begin{equation}\label{asus}
 		\text{ $0$ is not a Dirichlet eigenvalue for \eqref{ini-ext prob}.}
 	\end{equation}
 	We  define the nonlocal Dirichlet to Neumann (DN)  map as follows
 	\begin{align}\label{DN map}
 		\Ld_q : \left.u\right|_{Q_e} \to \left.\mc{H}^s u\right|_{Q_e}
 	\end{align}

Our first result is that  one can recover the potential term $q$ in $Q$ uniquely given the nonlocal DN map. 
 	
 	Next we consider a fractional parabolic problem involving a first order term as well. For $ q \in L^\iy(Q) $ and $ b \in L^\iy \left((-T,T);W^{1-s,\iy}(\Omega)\right) $, we consider the initial-exterior problem 
 	\begin{align}\label{ini-ext prob for b,q}
 		\begin{cases}
 			\left( \mc{H}^s + \langle b(x,t), \nabla_{x} \rangle + q(x,t) \right)u = 0, & \tn{ in } Q\\
 			u(x,t) = f(x,t), & \tn{ in } Q_e\\
 			u(x,t) = 0, & \tn{ for } t \leq-T.
 		\end{cases}
 	\end{align}
 	  As before, we assume that
 	\begin{equation}\label{kju2}
 		\text{$0$ is not a Dirichlet eigenvalue for \eqref{ini-ext prob for b,q}}.
 		\end{equation}
 	and define the nonlocal parabolic DN map 
 	\begin{align}\label{DN map for b,q}
 		\Ld_{b,q}: u\vert_{Q_e} \to \mc{H}^s u\vert_{Q_e} .
 	\end{align}

Our second result is that one can uniquely recover the coefficients $b$ and $q$ from the data $\Ld_{b,q}$.



    

  We now give a brief survey of local and non-local versions of the Calder\'on inverse problem in the elliptic and parabolic settings. Calder\'on initiated the study in this direction in his fundamental article \cite{Calderon_Paper}, where he asked the question whether one can determine the conductivity of a medium from boundary Dirichlet to Neumann data, and gave some partial answers. 
  This work served as the initial impetus for several deep and insightful works in the context of elliptic inverse problems; see \cite{SYL, Nachman_annals, AP, Ch, KEN}. The problem of unique determination of the conductivity from boundary Dirichlet to Neumann map is typically transformed to an inverse problem for the Schr\"odinger equation, that is an equation of the type $(-\Delta +q)$, from the corresponding Dirichlet to Neumann map. The method of complex geometric optics  (CGO) solutions has served as a crucial ingredient in the proofs of these inverse problems.  This has proven versatile to be effective in the solution of several inverse problems involving PDEs. It is not our intention to give a broad survey of exisiting results in inverse problems for PDEs and for this reason we limit ourselves to those problems whose fractional analogues we study in this paper.  Analogous to the case of the Schr\"odinger equation, an inverse problem for the magnetic Schr\"odinger equation, $\sum\limits_{i=1}^{n}\left( \frac{1}{i} \frac{\partial}{\partial x_j} + W_j\right)^2 + q(x)$,  was studied in \cite{Sun_Paper} under a smallness assumption on the first order term $W$, and removing this assumption later in \cite{NSU}. However, in this situation, the inverse problem exhibits a phenomenon of gauge invariance, that is, there is an obstruction to uniquely recovering the first order term from boundary Dirichlet to Neumann data. Inverse problems for parabolic equations have been studied extensively as well. We refer to the following initial works in this context;  \cite{Isakov, Avdonin_Seidman}.

  In recent years, study of inverse problems involving fractional powers of local operators has been gaining significant attention. The work in this direction for the fractional Laplacian involving a zeroth order term was initiated by \cite{GSU}. The results in \cite{GSU} were subsequently extended to variable coefficient operators with smooth principal part in  \cite{GLX}. An inverse problem for the fractional Laplacian with both zeroth and first order term was recently considered in \cite{CLR}; see also \cite{BGU} for a related work. We should mention the important feature that unlike the local case, the phenomenon of gauge invariance disappears in the nonlocal framework. Moving on to fractional analogoues of the parabolic operator, an inverse problem for a fractional parabolic operator  of the form $(\partial_t - \Delta)^s +q$   was recently considered in \cite{LLR}. Two related works with slightly different fractional parabolic operators are \cite{RS,LL}.  
  
  In this work, our main focus  is the unique determination of the potential and the drift term from the nonlocal DN map for more general operators of the type $(\partial_t - \operatorname{div}( A(x)\nabla))^s$ where $A$ is assumed to be Lipschitz continuous. Our results therefore  generalize those in \cite{LLR} as well as those in the elliptic case in  \cite{GLX} where, instead, smooth coefficients are considered.
  
  \subsection{Main results}
  We now proceed to give the main results of the article.    Our first main result concerns the unique determination of the potential $q$.
  \begin{theorem}\label{determining q}  
    Let $T>0$ and $\Omega \subset \mb{R}^n, n\ge 1$ be an open bounded set. Consider $q_1,q_2 \in L^{\iy}(Q)$ and any two nonempty open sets in $\Omega_e$ say $W_1$ and $W_2$ such that 
    \begin{align*}
       \left. \Ld_{q_1}(f)\right|_{W_2 \times (-T,T)} =  \left. \Ld_{q_2}(f)\right|_{W_2 \times (-T,T)}, \ \tn{ for all } f \in C_0^{\iy}(W_1 \times (-T,T))
    \end{align*}
    then $q_1 = q_2$ in $Q$.
  \end{theorem}
   We also uniquely recover the coefficients $b,q$ for \eqref{ini-ext prob for b,q} given the nonlocal DN map. The following result below is the parabolic generalization  of Theorem 1.1  in \cite{CLR}. 
   
  \begin{theorem}\label{determining b,q}
  Let $T>0$ and $\Omega \subset \mb{R}^n, n\ge 1$ be an open bounded Lipschitz set. Consider $q_1,q_2 \in L^{\iy}(Q)$ and  $ b_1, b_2 \in L^2\left((-T,T); W^{1-s,\iy}(\Omega)\right)$. We further choose two nonempty open sets from $\Omega_e$ say $W_1$ and $W_2$ such that 
    \begin{align*}
       \left. \Ld_{b_1,q_1}(f)\right|_{W_2 \times (-T,T)} =  \left. \Ld_{b_2,q_2}(f)\right|_{W_2 \times (-T,T)}, \ \tn{ for all } f \in C_0^{\iy}(W_1 \times (-T,T))
    \end{align*}
    then $ b_1 = b_2, \ q_1 = q_2$ in $Q$.
  \end{theorem}
  The proofs of our main results, Theorem \ref{determining q} and Theorem \ref{determining b,q}, crucially rely on a global weak unique continuation property  for $\Hc^{s}$ (see Theorem \ref{GUCP} below), and substantial parts of this article rest on proving this result. In the following result, for the fractional parabolic space $\mb{H}^s(\mb{R}^{n+1})$, we refer to the definition in \eqref{fpSS}.


  %
  \begin{theorem}\label{GUCP}
    Let $T>0$ and ${U}$ be a nontrivial open set in $\mb{R}^n,n\ge 1$. For some $u\in \mb{H}^s(\mb{R}^{n+1})$, if
    \begin{align*}
     u = \mc{H}^s u = 0 \tn{ in } {U} \times (-T,T),
    \end{align*}
    then $u=0$ in $\mb{R}^n \times (-T,T)$.
  \end{theorem}
  
  In an exactly analogous way as  in \cite{GSU}, Theorem \ref{GUCP} is used to prove the Runge approximation properties  in Theorem \ref{Runge for q} and Theorem \ref{Runge for b and q} that is tailored for Theorem \ref{determining q} and Theorem \ref{determining b,q} respectively.  This allows bypassing the method of complex geometric optics (CGO) solutions which is a crucial ingredient in the local case  (see for instance \cite{SYL}). Over here, it is worthwhile to mention that that Runge approximation results were first obtained for $(-\Delta)^s$ in \cite{DSV}; see also \cite{DSV1}.
  
  Now regarding Theorem \ref{GUCP}, we  mention that in the case when $A=\mathbb I_n$, such a result has been established in \cite[Proposition 5.6]{LLR} as a consequence of the following weak unique continuation property for the corresponding extension problem. We refer to Section \ref{Carleman estimate and Unique continuation} for the precise notations.
  
   Before proceeding further, let us declare that  we will denote $1-2s$ by $a$ from now on.  Notice that $a \in (-1,1)$.
  \begin{theorem}[Weak unique continuation property]\label{wkt}
  Let $U_0$ be a weak solution to the following extension problem
  \begin{equation}\label{wkt2}
  \begin{cases}
  \operatorname{div}_X(x_{n+1}^a \nabla U_0) = x_{n+1}^a \partial_t U_0\ \text{in $\mb B_1^+,\ X=(x, x_{n+1}) $}
  \\
  \operatorname{lim}_{x_{n+1} \to 0^+} x_{n+1}^a \partial_{n+1} U_0(x, 0, t)=  U_0(x, 0, t) =0\ \text{in $B_1 \times (0,1)$.}
  \end{cases}
  \end{equation}
  Then $U_0\equiv 0$ in $\mathbb{B}_1^+  \times (0,1)$.
  \end{theorem}
  Such a result has been derived in \cite{LLR} by the following two steps.
  
  \emph{Step 1:}  By means of repeated differentiation and a bootstrap argument, it is first shown that the zero Cauchy data in \eqref{wkt2} implies that $U_0$ vanishes to infinite order in space and time at every point $(x, 0, t) \in B_1 \times (0,1). $

  \emph{Step 2:} Then by means of a Carleman estimate for the constant coefficient  operator in \eqref{wkt2} that the authors establish, it is shown that $U_0 \equiv 0$ in $\mathbb B_1 ^+ \times (0,1)$.

  Finally once Theorem \ref{wkt} is proven in \cite{LLR}, it is applied to the solution $U$ of the extension problem \eqref{exten prob} corresponding to Dirichlet data $u$  (when $A= \mb I_n$) using which one can  assert that $U$ vanishes in $\mb B_1^+ \times (0,1)$. Then  noting  that $U$ solves a uniformly parabolic PDE with real analytic coefficients away from $\{x_{n+1} =0\}$, one can thus  spread the zero set using the classical theory and thus conclude that $U \equiv 0$ in $\mb R^{n+1}_+ \times (0,1)$. Theorem \ref{GUCP} now follows since $U=u$ at $\{x_{n+1} =0\}$.

  It turns out that more recently in \cite{BG1}, it has been shown that solutions to a more general class of equations modelled on \eqref{wkt2} are real analytic in the space variable $x$ which includes the extension variable and therefore the use of the Carleman estimate in \emph{Step 2} above  can be avoided. We also mention that a certain  variant of the weak unique continuation property in Theorem \ref{wkt} is also used to characterize singular points in the fractional heat obstacle problem in \cite{BDGP}.

  Similar to that in \cite{LLR}, in this work we derive Theorem \ref{GUCP} by obtaining an analogous weak unique continuation property for extension problems of the type \eqref{wk1} where the matrix valued function $A$ satisfies the uniform ellipticity condition in  \eqref{ellip} and the Lipschitz growth condition in \eqref{assum}.  This constitutes one of the key novelties of this work. This is done by means of a new Carleman estimate   that we  establish for degenerate operators of the type \eqref{extop}. The  estimate in Lemma \ref{carl} below can be regarded as a generalization of the Carleman estimate  for uniformly parabolic operators with Lipschitz coefficients  as in the fundamental works of Escauriaza, Fernandez and Vessella in \cite{EF_2003, EFV_2006}.   Such a generalization has  required some  very subtle adaptations of the ideas in \cite{EF_2003, EFV_2006}  for which we  refer to the discussion in Section \ref{Carleman estimate and Unique continuation} below. Inspired by ideas in the recent work \cite{ABDG}, we  combine such an estimate with  a  monotonicity in time result as in Lemma \ref{mont} that we derive using which  we show the validity of a conditional doubling property for solutions to the extension problem. This facilitates  the use of blowup arguments  which then reduces the weak unique continuation property for \eqref{wk1} to that of the constant coefficient extension problem as in Theorem \ref{wkt} above. We would also like to mention that Theorem \ref{energy} is another central result in our work where we show that for $u \in \mb{H}^s(\mb{R}^{n+1})$, the solution $U$ to the corresponding extension problem \eqref{exten prob} belongs to the right energy space and moreover the weighted Neumann derivative  can be interpreted  as a limit in an appropriate norm. Therefore weak type methods can be applied and  this is finally extremely crucial for the proof of the  unique continuation result, Theorem \ref{GUCP}  in Section \ref{Carleman estimate and Unique continuation}. As the reader will see, the proof of our main results rely on several non-trivial facts from analysis and PDE that in our context beautifully combine.
  
  In closing, we provide a brief account of  unique continuation results in the nonlocal setting.  For nonlocal elliptic equations of the type $(-\Delta)^s +V$ a strong unique continuation theorem was obtained by Fall and Felli, see Theorem 1.3 in \cite{FF}. Their analysis combines the approach in \cite{GL1}, \cite{GL2} with the Caffarelli-Silvestre extension method in \cite{CS}.  We also mention the interesting work of Ruland \cite{Ru}, \cite{Ru2}, where the Carleman method has been used, together with \cite{CS}, to obtain results similar to those in \cite{FF} but with weaker assumptions on the potential $V$. In the time dependent case,  for global solutions of \begin{equation}\label{k01}(\partial_t - \Delta)^s u =Vu, \end{equation}   a backward space-time strong unique continuation theorem was previously established  in \cite{BG} with appropriate assumptions on $V$. We also refer to \cite{AT} for some interesting results on the regularity of the nodal sets of such solutions.  The result in \cite{BG} represents the nonlocal counterpart of the one first obtained  by Poon in \cite{Po} for the local case $s=1$. More recently, a space like strong unique continuation result for local solutions to \eqref{k01} has been obtained in \cite{ABDG} which constitutes the nonlocal counterpart of the space-like strong unique continuation results in the aforementioned works \cite{EF_2003, EFV_2006}. It is to be noted that both the works \cite{ABDG} and \cite{BG} uses the extension problem for the fractional heat operator that has been developed in \cite{NS} and \cite{ST} independently. When $s = 1/2$ the extension problem  was first introduced in \cite{Jr1} by Jones.

    The article is organized as follows. We outline the background to define the nonlocal operator $\mc{H}^s$ and its domain in section \ref{Preliminaries}. Also we discuss the mapping property of $\mc{H}^s$ there. In section \ref{Some direct problems}, we derive some results on the extension problem for $\mc{H}^s$ which will be followed by the well-posedness of the initial-exterior problems for $ \mc{H}^s + q$ and $ \mc{H}^s + \langle b,\nabla_{x} \rangle + q$. In section \ref{Carleman estimate and Unique continuation}, we present the unique continuation part where we first establish a Carleman estimate for the extended operator and combine it with a blow-up analysis to deduce the weak unique continuation result mentioned in Theorem \ref{GUCP}. In section \ref{Applications in Inverse Problems}, we prove Runge appoximation properties and provide the unique determination results for the inverse problems, Theorem \ref{determining q} and Theorem \ref{determining b,q}.

    \medskip

    \textbf{Acknowledgement:} Both the authors would like to thank Prof. Venkateswaran Krishnan for various valuable  comments and suggestions  related to this work and for constant encouragement.


   \section{Preliminaries} \label{Preliminaries}
   In this section we introduce the relevant notation and gather some auxiliary  results that will be useful in the rest of the paper. Generic points in $\mb{R}^n \times \mb{R}$ will be denoted by $(x_0, t_0), (x,t)$, etc. For an open set $\Omega\subset \mb{R}^n_x\times \mb{R}_t$, by $C_0^{\infty}(\Omega)$ we mean the set of compactly supported smooth functions in $\Omega$.
    We  will assume that the uniformly parabolic operator $\partial_t - \operatorname{div}(A(x)\nabla_x)$ in $\mb{R}^n\times \mb{R}$ has a globally defined fundamental solution $p(x,x',t)$ that satisfies for every $x\in \mb{R}^n$ and $t>0$
\begin{equation}\label{stoc}
P_t 1(x,t) = \int_{\mb{R}^n} p(x,x',t) \d x' = 1.
\end{equation}   
We also assume that the matrix valued function $A$  is uniformly elliptic, i.e.
\begin{equation}\label{ellip}
\Lambda^{-1} \mb{I} \leq A  \leq \Lambda \mb{I},
\end{equation}
for some $\Lambda>1$ and 
satisfies the following Lipschitz boundedness assumption
\begin{equation}\label{assum}
|A(x) - A(y) | \leq K |x-y|.
\end{equation}
We start by introducing the following notion of evolutive semigroup 
    \begin{align}\label{semigroup}
        P^{\mc{H}}_{\tau} u (x,t) = \int_{\mb{R}^n} p(x,y,\tau) u(y,t-\tau) \ \d y, \qd  \tn{ for } u \in \mc{S}(\mb{R}^{n+1})
    \end{align}
    where $ p(x,z,t) $ is the heat kernel associated to the elliptic operator \begin{equation}\label{L} \mc{L} \overset{def}= \operatorname{div}(A(x) \nabla). \end{equation} Note that, $\{P^{\mc{H}}_{\tau}\}_{\tau \ge 0}$ is a strongly continuous contractive semigroup satisfying $ \|P^{\mc{H}}_{\tau}u - u\|_{L^2(\mb{R}^{n+1})} = O(\tau)$.
    \begin{definition}
    For $s \in (0,1)$ and $u \in \mc{S}(\mb{R}^{n+1})$, we define $\mc{H}^s$ based on the Balakrishnan formula \cite[Eq. (9.63) on pp. 285]{Samko}) in the following way
    \begin{align}\label{defn of frac parabolic}
         \mc{H}^s u (x,t) = -\f{s}{\Gamma(1-s)} \int_{0}^{\iy} \left( P^{\mc{H}}_{\tau} u(x,t) - u(x,t) \right) \f{\d\tau}{\tau^{1+s}}.
    \end{align}
    \end{definition}
    Next we denote  by $\{E_{\lambda}\}$ the spectral measures associated to $\mc{L}$. More precisely, we let
    \begin{equation}\label{spect}
    \mc{L}= - \int_{0}^\infty \lambda\ dE_\lambda.\end{equation}
    Invoking  such a  spectral decomposition for the operator $\mc{L}$ and by using Fourier transform in $t$ variable, we alternatively express $\mc{H}^s u$ in Fourier terms. To do so, we first observe the following representation of the heat semigroup $\{P_t\}_{t\ge 0}$ in terms of spectral measures as well as an important identity for gamma functions
    \begin{align}\label{spec meas}
        P_t = \int_{0}^\iy e^{-\ld t} \d E_{\ld}, \tn{ and } \ -\f{s}{\Gamma(1-s)} \int_{0}^\iy  \f{ e^{-(\ld+i\sigma)t} - 1 }{\tau^{1+s}}\d\tau = \left( \ld+i\sigma \right)^s, \ \ld>0, \sigma\in\mb{R}.
    \end{align}
    We refer to Section 2 in \cite{BGMN} for a detailed account on this aspect.
    Now we consider the Fourier transform in time variable to obtain from \eqref{semigroup}
    \begin{align*}
        \mc{F}_t \left( P^{\mc{H}}_{\tau} u \right) (\xi,\sigma) = e^{-i\sigma\tau} P_\tau\left( \mc{F}_t u(\cdot,\sigma) \right)(\xi)
    \end{align*}
    which results into the Fourier analogue of the definition \eqref{defn of frac parabolic} 
    \begin{align*}
        \mc{F}_t \left(\mc{H}^s u \right)(\cdot,\sigma) & = -\f{s}{\Gamma(1-s)} \int_0^\iy \f{1}{\tau^{1+s}} \int_0^\iy \left( e^{-(\ld+i\sigma)\tau} -1 \right) \d E_\ld(\mc{F}_t u(\cdot,\sigma)) \ \d\tau \\
        & = \int_0^\iy \left(\ld+i\sigma\right)^s \d E_\ld(\mc{F}_t u(\cdot,\sigma))
    \end{align*}
    Here we have crucially used the relations in \eqref{spec meas}. Consequently, we can write for $u\in \mc{S}(\mb{R}^{n+1})$ 
    \begin{align*}
        \|\mc{F}_t \left(\mc{H}^s u \right)(\cdot,\sigma)\|_{L^2(\mb{R}^{n})} = \int_0^\iy |\ld+i\sigma|^{2s} \d\|E_\ld (\mc{F}_t u(\cdot,\sigma))\|^2, \qd  \sigma \in\mb{R}.
    \end{align*}
    Keeping this in mind, we now define the appropriate function space which constitutes the domain of $\mc{H}^s$ and associated norm.
    \begin{definition}
    For $s \in (0,1)$, we define the space $\mc{H}^{2s}(\mb{R}^{n+1})$ to be the completion of $\mc{S}(\mb{R}^{n+1})$ with respect to the norm 
    \begin{align} \label{norm}
        \|u\|_{\mc{H}^{2s}(\mb{R}^{n+1})} = \left( \int_\mb{R} \int_0^\iy \left( 1 + |\ld+i\sigma|^{2} \right)^s  \d\|E_\ld (\mc{F}_t u(\cdot,\sigma))\|^2 \d\sigma \right)^{\f{1}{2}}.
    \end{align}
    \end{definition}
    It is worth noting that $\tn{Dom}(\mc{H}) \subseteq \tn{Dom}(\mc{H}^s)$.
    
    More generally,  we introduce the various   function spaces that are  needed in this set-up. Let  $\mc{O}$  be an open set in $\mb{R}^{n+1}$ and $r \in \mb{R}$. We define
 	\begin{align*}
 	  &  \mc{H}^r\left(\mb{R}^{n+1}\right) = \bigg\{\text{Completion of $\mc{S}(\mb R^{n+1})$ w.r.t  the norm}:\\& \hspace{2.3cm} \int_{\mb{R}} \int_0^\infty \left(\left(1+|\lambda+i\sigma|^2\right)^{r/2}  d||E_\lambda(\mc{F}_{t} u(\cdot,\sigma))||^2\ \d\sigma \right)^{1/2} \bigg\}, \\
 	   & \mc{H}^r\left(\mc{O}\right) = \left\{u\vert_{\mc{O}}; u \in \mc{H}^r\left(\mb{R}^{n+1}\right)\right\}, \qd \widetilde{\mc{H}}^r\left(\mc{O}\right) = \tn{closure of } C_0^\iy(\mc{O}) \tn{ in } H^r\left(\mb{R}^{n+1}\right).
 	\end{align*}   
	Also we define
	\begin{equation}\label{subset}
	||u||_{\mc{H}^r\left(\mc{O}\right)} = \text{inf} \{||v||_{\mc{H}^r(\mb{R}^{n+1})}: v\vert_{\mc{O}} =u\}.
	\end{equation}
	Now from resolution of the parabolic version of the Kato square root problem as in \cite{AEN} and interpolation type arguments, we note   that the following equivalence holds  %
   \begin{align}\label{iden}
       \mb{H}^s (\mb{R}^{n+1}) = \mc{H}^{s} (\mb{R}^{n+1}), \qd s \in (0,1),
   \end{align}    
   where $\mb{H}^s(\mathbb R^{n+1})$ is the parabolic fractional Sobolev space defined as
   \begin{equation}\label{fpSS}
   \mb{H}^{s}(\mathbb R^{n+1}) \overset{def}= \{u \in L^2: (|\xi|^2 + i \rho)^{s/2} \mc{F}_{x,t} u(\xi, \rho) \in L^2\}.
   \end{equation}
   
  We refer to \cite[pages 6-7]{LN} for relevant details. Sometimes when the context is clear, the space-time fourier transform $\mc{F}_{x,t}u$ will be denoted by $\hat{u}.$    From now on in view of \eqref{iden}, we will identify both the spaces $\mb{H}^s$ and $\mc{H}^s$ and furthermore for a closed set $E$ in $\mb{R}^{n+1}$ we let
  \begin{equation}\label{nb1}
  \mb{H}^s_E= \{u \in \mb{H}^{s}(\mb{R}^{n+1}): \text{supp}(u) \subset E\}.
  \end{equation}
  It is easily seen that $\mb{H}^s_E$ is a Hilbert space.
  
  We now note that the adjoint operator $\mc{H}_{*}^s$ is defined in terms of the spectral resolution in the following manner
    \begin{align*}
        \mc{F}_t \left(\mc{H}_{*}^s u \right)(\cdot,\sigma) = \int_0^\iy (\ld-i\sigma)^s \d E_\ld(\mc{F}_t u(\cdot,\sigma)), \qd \tn{ for } u \in \mc{S}(\mb{R}^{n+1}).
    \end{align*}

    For $f,g \in \mc{S}(\mb{R}^{n+1})$, we observe from the properties of the spectral family of projection operators $\{E_\ld\}_{\ld>0}$ that
    \begin{align}
       \nn \langle \mc{H}^s f, g \rangle  = \langle \mc{H}^{\f{s}{2}}f, \mc{H}_{*}^{\f{s}{2}} g \rangle = \langle f, \mc{H}_{*}^s g \rangle & = \int_{\mb{R}}\int_{0}^\iy \left(\ld+i\sigma\right)^s \d\langle E_\ld \mc{F}_t u, \overline{\mc{F}_t v} \rangle(\cdot,\sigma) \d\sigma \\
      \label{mapping}  & \preceq \|f\|_{\mb{H}^{s}(\mb{R}^{n+1})} \|g\|_{\mb{H}^{s}(\mb{R}^{n+1})}.
    \end{align}
  As an outcome of the inequality \eqref{mapping}, we have the mapping property  $\mc{H}^s : \mb{H}^s(\mb{R}^{n+1}) \to \mb{H}^{-s}(\mb{R}^{n+1})$ where $\mb{H}^{-s}$ denotes the dual space. \\

 	%
 
 	%

 	
 	\section{Some Direct problems} \label{Some direct problems}
 	In this section, we study some direct problems related to the fractional operator $\mc{H}^s$. We start with the discussion on the extension problem for $\mc{H}^s$ . Then the well-posedness results for \eqref{ini-ext prob} and  \eqref{ini-ext prob for b,q}      will be presented which mainly relies on the Lax-Milgram type arguments. 
 	\subsection{The extension problem for $\mc{H}^s$}\label{ext}
 	 Our objective here is to solve the extension problem for $\mc{H}^s$ with prescribed Dirichlet data $u \in \mb{H}^s(\mb{R}^{n+1})$. More specifically, we consider solution to the following Dirichlet problem in $\mb{R}^{n+2}_{+}$ 
 	\begin{align}\label{exten prob}
 	\begin{cases}
 	   \mc{L}_a U=  z^{a} \dd_t U - \tn{div} \left( z^{a} A(x) \nabla_{x,z} U \right) = 0,& \qd \tn{ in } \mb{R}^{n+1} \times \mb{R}_{+}, \ a =1-2s\\
 	    U(x,t,0) = u(x,t),&  \qd \tn{ on } \mb{R}^{n+1}.
 	\end{cases}
 	\end{align}
 	by introducing a new variable $ z\in \mb{R}_{+}$. As it is well known by now, \eqref{exten prob} represents the parabolic counterpart of the Caffarelli-Silvestre extension problem as in \cite{CS} for $\mc{H}^s$. See \cite{BGMN, BS, NS, ST}. More precisely, it  has been shown in the cited works  that if $u \in \mb{H}^{2s}$, then we have in $L^{2}(\mathbb{R}^{n+1})$,
	\begin{equation}\label{conv1}
	\operatorname{lim}_{z \to 0^+} z^a \partial_z U= \mc{H}^s.
	\end{equation}

	In our setting of the Calderon problem, it turns out that we need to deal with $u \in \mb{H}^s$. Therefore, this requires generalizing the convergence in \eqref{conv1}  with respect to a weaker norm for functions in $\mb{H}^s$  and this  is the one of the main contents  of Theorem \ref{energy} below. Such a result has already been established in the case when $A=\mathbb{I}$ in \cite[Proposition 4.2]{CLR}.  Moreover, we also show that the extended function belongs to the right energy space so that weak type methods as in \cite{BG} can be subsequently applied.   This is finally relevant to the weak unique continuation result Theorem \ref{GUCP} that we prove in Section \ref{Carleman estimate and Unique continuation}.   In this regard, we now introduce the relevant energy space.
	
 	 For an open set $\Sigma \subseteq \mb{R}^{n+1} \times \mb{R}_{+}$, we define the energy space $\mc{L}^{1,2} \left(\Sigma; z^a \d x \d t \d z \right)$ which consists of all those $ U \in L^2\left(\Sigma; z^a \d x \d t \d z \right) $ such that $\nabla_{x} U$ and $\dd_z U \in L^2\left(\Sigma; z^a \d x \d t \d z \right)$, endowed with the norm 
 	\begin{align*}
 	    \|U\|_{\mc{L}^{1,2} \left( \Sigma; z^a \d x \d t \d z \right)} \overset{def}=  \int_{\Sigma \times \mb{R}_{+}} z^a \left( |U|^2 + |\nabla_x U|^2 + |\dd_z U|^2 \right) \d x \d t \d z.
 	\end{align*}
 As mentioned above, we now state the central result of this subsection  which concerns the various  convergence properties of the extended function in \eqref{exten prob} (and its  weighted Neumann derivative)  corresponding to $u \in \mb{H}^s$. Theorem \ref{energy} below can be regarded as  the variable coefficient analogue of Proposition 4.2 in \cite{CLR}. We crucially adapt some ideas from \cite{BGMN} in our proof  of this result. However unlike that in \cite{BGMN},  since the convergence results  are established in different norms, therefore our proof has required some delicate reworking of the ideas in \cite{BGMN}.
 	\begin{theorem}\label{energy}
 	  Let $s \in (0,1)$ and $u \in \mb{H}^s(\mb{R}^{n+1})$. There exists a solution to \eqref{exten prob} satisfying
 	\begin{align*}
 	   & i)\ \lim_{z \to 0+} U(\cdot,\cdot,z) = u  \tn{ in } \mb{H}^{s}(\mb{R}^{n+1}), \\
 	   & ii)\ \lim_{z \to 0+} \f{2^{-a} \Gamma(s)}{\Gamma(1-s)} z^{a} \dd_z U(\cdot,\cdot, z) = \mc{H}^s u  \tn{ in } \mb{H}^{-s}(\mb{R}^{n+1}), \\
 	   & iii)\  \|U\|_{\mc{L}^{1,2} \left( \mb{R}^{n+1} \times (0,M); z^a \d x \d t \d z \right)} \preceq_{M} \|u\|_{\mb{H}^{s} (\mb{R}^{n+1})}. 
 	\end{align*}
 	\end{theorem}
 	\begin{proof}
 	We first note that the solution to \eqref{exten prob} is given by 
 	\begin{align}\label{soln to extn prob}
 	    U(x,t,z) = \int_0^\iy \int_{\mb{R}^n} P^a_z (x,y,\tau) u(y,t-\tau) \ \d y\d\tau
 	\end{align}
 	where \[P^a_z (x,y,t) : = \f{1}{2^{1-a} \Gamma(\f{1-a}{2})} \f{z^{1-a}}{t^{\f{3-a}{2}}} e^{-\f{z^2}{4t}} p(x,y,t).\] By taking the Fourier transform in time variable in \eqref{soln to extn prob}, we have the expression
 	\begin{align}
 	   \nn \mc{F}_t U(x,\sigma,z) & = \f{z^{1-a}}{2^{1-a}\Gamma(\f{1-a}{2})} \int_0^\iy \f{e^{-\f{z^2}{4\tau}}}{\tau^{\f{3-a}{2}}} e^{-i\tau\sigma} P_\tau(\mc{F}_t u(\cdot,\sigma))(x) \ \d\tau. \\
 	  \label{FT of the soln to extn prob}  & = \f{z^{1-a}}{2^{1-a}\Gamma(\f{1-a}{2})} \int_0^\iy \int_0^\iy \f{e^{-\f{z^2}{4\tau}}}{\tau^{\f{3-a}{2}}} e^{-(\ld+i\sigma)\tau} \d\tau \ \d E_{\ld} (\mc{F}_t u)(\cdot, \sigma)(x)
 	\end{align}
 	where we used the spectral representation of $P_\tau$ as in \eqref{spec meas}  in the last line. The relation \eqref{FT of the soln to extn prob} readily implies
 	\begin{align}\label{re1}
 	    \mc{F}_t \left( U(\cdot,\sigma,z) - u(\cdot,\sigma) \right) = \f{z^{2s}}{4^s \Gamma(s)} \int_0^\iy \int_0^\iy \f{e^{-\f{z^2}{4\tau}}}{\tau^{1+s}} \left( e^{-(\ld + i\sigma) \tau } - 1 \right) \d \tau \ \d E_{\ld} (\mc{F}_t u)(\cdot, \sigma)
 	\end{align}
 	To show that $U$ attains the prescribed data $u$ at $z=0$ in the $\mb{H}^s$ sense, we recall the important identity which can be found in \cite[page 369]{GR}
 	\begin{align}\label{Bessel func}
 	    \int_0^\iy t^{\nu - 1} e^{-\left( \f{\beta}{t} + \gamma t \right)} \ \d t = 2 \left( \f{\beta}{\gamma} \right)^{\f{\nu}{2}} K_\nu (2 \rt{\beta\gamma})
 	\end{align}
 	where $ \Re(\beta), \Re(\gamma) > 0 $ and $ K_{\nu} $ is the Bessel function of third kind. We notice that 
 	\begin{align}
 	   \nn \int_0^\iy \f{e^{-(\ld + i\sigma)\f{z^2}{4\theta}}}{\theta^{1+s}} \left[ e^{-\theta} - 1 \right] \d \theta & = 2 \left( \f{z \rt{\ld+i\sigma}}{2} \right)^{-s} K_{s} \left( z \rt{\ld+i\sigma} \right) - \Gamma(s) \left( \f{z \rt{\ld+i\sigma}}{2} \right)^{-2s} \\
 	  \label{reln_Bessel func}  & \hspace{-2cm} = 2^{1+s} \left( z \rt{\ld+i\sigma} \right)^{-2s} \left( \left( z \rt{\ld+i\sigma} \right)^{s} K_{s} \left( z \rt{\ld+i\sigma} \right) - 2^{s-1} \Gamma(s) \right)
 	\end{align}
 	where we took $ \beta = \f{(\ld + i\sigma)z^2}{4} , \gamma = 1, \nu = -s $ in \eqref{Bessel func} and used analytic continuation to the following identity
 	\begin{align}\label{id}
 	    \int_0^\iy e^{-\f{\zeta^2}{\theta}} \f{\d \theta}{\theta^{1+s}} = \int_0^\iy \f{e^{-p}}{\left( \f{\zeta^2}{p} \right)^{1+s}} \f{\zeta^2}{p^2} \d p = \zeta^{-2s}\int_0^\iy p^{s-1} e^{-p} \d p = \Gamma(s) \zeta^{-2s}, \tn{ for } \zeta > 0 . 
 	\end{align}
	Alternatively, extension of the identity \eqref{id} to complex parameters   can also be justified   by a contour type integration in the complex plane. See for instance the proof of Theorem 1.1 in \cite{BGMN}.
 	Using \eqref{re1} we have, 
 	\begin{align}
 	   & \| U(\cdot,\cdot,z) - u(\cdot,\cdot) \|^{2}_{\mb{H}^{s} (\mb{R}^{n+1})} = \int_\mb{R} \int_0^\iy \left( 1 + |\ld + i\sigma|^{\f{s}{2}} \right)^2 \d \|E_\ld (\mc{F}_t \left( U(\cdot,\sigma,z) - u(\cdot,\sigma) \right))\|^2 \d\sigma \nn \\
 	    & \qd \preceq z^{4s} \int_{\mb{R}} \int_0^\iy \left( 1 + |\ld + i\sigma|^{\f{s}{2}} \right)^2 \left\vert \int_0^\iy \f{e^{-\f{z^2}{4\tau}}}{\tau^{1+s}} \left( e^{-(\ld + i\sigma) \tau } - 1 \right) \d \tau \right\vert^2 \d \|E_\ld (\mc{F}_t u)(\cdot,\sigma)\|^2 \d\sigma \nn \\
 	    & \qd \preceq z^{4s} \int_{\mb{R}} \int_0^\iy \left( 1 + |\ld + i\sigma|^{\f{s}{2}} \right)^2 |\ld + i\sigma|^{2s} \left\vert  \int_0^\iy \f{e^{-(\ld + i\sigma)\f{z^2}{4\theta}}}{\theta^{1+s}} \left[ e^{-\theta} - 1 \right] \ \d \theta \right\vert^2 \d \|E_\ld (\mc{F}_t u)(\cdot,\sigma)\|^2 \d\sigma,\notag
	    \end{align}
	    where in the last line, we used a complex change of variable which can be justified  by  a contour type integration.  Now using \eqref{reln_Bessel func} we find from above	    
	    \begin{align}
 	& \| U(\cdot,\cdot,z) - u(\cdot,\cdot) \|^{2}_{\mb{H}^{s} (\mb{R}^{n+1})}\notag\\ &\qd \preceq \int_{\mb{R}} \int_0^\iy \left( 1 + |\ld + i\sigma|^{\f{s}{2}} \right)^2 \left\vert \left( z \rt{\ld+i\sigma} \right)^{s} K_{s} \left( z \rt{\ld+i\sigma} \right) - 2^{s - 1} \Gamma(s) \right\vert^2 \d \|E_\ld (\mc{F}_t u)(\cdot,\sigma)\|^2 \d\sigma \nn \\
 	    & \qd \preceq \underset{|\xi| \le \ep}{\tn{sup}} \ \left|{\xi^s K_s(\xi) - 2^{s - 1}  \Gamma(s)} \right|^2 \int_{\mb{R}} \int_0^\iy \left( 1 + |\ld + i\sigma|^{\f{s}{2}} \right)^2 \d \|E_\ld (\mc{F}_t u)(\cdot,\sigma)\|^2 \d\sigma \nn \\
 	 \label{bdry val of CS extn}  & \qd \qd \qd + \underset{|\xi| > \ep}{\tn{sup}} \ \left|{\xi^s K_s(\xi) - 2^{s - 1}  \Gamma(s)} \right|^2 \int_{\mb{R}} \int_0^\iy \boldsymbol{\chi}_{z |\ld + i\sigma|^{\f{1}{2}} > \ep} \left( 1 + |\ld + i\sigma|^{\f{s}{2}} \right)^2 \d \|E_\ld (\mc{F}_t u)(\cdot,\sigma)\|^2 \d\sigma 
 	\end{align}
 	Taking $ z \to 0+ $ in \eqref{bdry val of CS extn}, we notice that its second term converges to zero. We also use the fact that $\xi^s K_s(\xi)$ is uniformly bounded for large $\xi$ which follows from the fact that for all large $\xi$
	\[
	|K_s(\xi) | \leq C e^{-\xi}.\]
	See for instance \cite[(5.11.8)]{Le}.
	
	After that, we use 
 	\begin{align*}
 	    \lim_{z \to 0} z^s K_s(z) = 2^{s - 1} \Gamma(s)
 	\end{align*}
 	and let $ \ep $ approach to zero in \eqref{bdry val of CS extn} to conclude that the first integral in \eqref{bdry val of CS extn} goes to $0$. This establishes i).
	
	 We now turn our attention to ii), i.e. we show that
 	$ \lim_{z \to 0+} \f{2^{-a} \Gamma(s)}{\Gamma(1-s)} z^{a} \dd_z U(\cdot,\cdot, z) = \mc{H}^s u $ in $ {\mb{H}^{-s}(\mb{R}^{n+1})} $. We first assume that $u \in \mathcal S(\mathbb{R}^{n+1})$. In order to prove ii), we will  make use of  the following identity holds which was observed in \cite[(3.14)]{BGMN}
 	%
	
 	\begin{align}\label{weighted norm der}
 	    \f{2^{-a} \Gamma(s)}{\Gamma(1 - s)} \ z^{a}\dd_{z} \mc{F}_t U (\cdot,\sigma,y) = - \f{1}{\Gamma(1 - s)} \int_0^\iy \left( \ld + i\sigma \right)^s \int_0^\iy \f{e^{-\theta} e^{-(\ld + i\sigma)\f{z^2}{4\theta}}}{\theta^s} \d \theta \ \d E_\ld (\mc{F}_t u) (\cdot,\sigma)
 	\end{align}
 	Also we have
 	\begin{align*}
 	    \mc{F}_t \left( \mc{H}^s u \right)(\cdot, \sigma) = \int_{0}^\iy \left( \ld + i \sigma \right)^s \d E_\ld (\mc{F}_t u)(\cdot,\sigma).
 	\end{align*}
 	We also make use of the following identity which follows from \eqref{Bessel func} by taking $ \beta = \f{(\ld + i\sigma)z^2}{4}, \gamma = 1, \nu = 1-s.$
 	\begin{align}\label{b1}
 	    \int_0^\iy \f{e^{-\theta} e^{-(\ld + i\sigma)\f{z^2}{4\theta}}}{\theta^s} \d \theta = 2 \left( \f{z \rt{\ld + i\sigma}}{2} \right)^{1 - s} K_{1 - s} \left(z \rt{\ld+i\sigma} \right).
 	\end{align}

	 We thus find
 	\begin{align}\label{cr1}
 	   & \f{2^{-a} \Gamma(s)}{\Gamma(1 - s)} \ z^{a}\dd_{z} \mc{F}_t U (\cdot,\sigma,z) - \mc{F}_t \left( \mc{H}^s u \right)(\cdot, \sigma) \\
 	    & = \f{2^s}{\Gamma(1 - s)} \int_0^\iy \left( \ld + i\sigma \right)^s \left( \left( {z \rt{\ld + i\sigma}} \right)^{1 - s} K_{1 - s} \left(z \rt{\ld+i\sigma} \right) - 2^{-s} \Gamma(1 - s) \right) \d E_\ld (\mc{F}_t u) (\cdot,\sigma).\notag
 	\end{align}
 	Now for any test function $ \phi \in {\mb{H}^{s}(\mb{R}^{n+1})} $,  using Cauchy-Schwarz inequality we find
 	\begin{align}\label{bl1}
 	   & \left\langle \f{2^{-a} \Gamma(s)}{\Gamma(1-s)} z^{a} \dd_z U(\cdot,\cdot, z) - \mc{H}^s u, \phi \right\rangle_{\mb{H}^{-s}(\mb{R}^{n+1}), {\mb{H}^{s}(\mb{R}^{n+1})}} \\ 	    & = \int_\mb{R} \int_0^\iy \d \left\langle E_\ld \left( \f{2^{-a} \Gamma(s)}{\Gamma(1 - s)} \ y^{a}\dd_{y} \mc{F}_t U (\cdot,\sigma,z) - \mc{F}_t \left( \mc{H}^s u \right)(\cdot, \sigma) \right), E_\ld \phi \right\rangle \d \sigma \notag\\
 	    & \preceq \left( \int_\mb{R} \int_0^\iy \left( 1 + |\ld + i\sigma| \right)^{-s} \d \left\| E_\ld \left( \f{2^{-a} \Gamma(s)}{\Gamma(1 - s)} \ z^{a}\dd_{z} \mc{F}_t U (\cdot,\sigma,z) - \mc{F}_t \left( \mc{H}^s u \right)(\cdot, \sigma) \right) \right\|^2 \d\sigma \right)^{1/2} \notag \\
 	    & \hspace{1cm}\times\left(  \int_\mb{R} \int_0^\iy \left( 1 + |\ld + i\sigma| \right)^{s} \d \|E_\ld \phi\|^2  \d\sigma\right)^{1/2}\notag \\
 	    & \preceq \|\phi\|_{\mb{H}^{s}(\mb{R}^{n+1})} \left(\int_\mb{R} \int_0^\iy (1+|\ld + i\sigma|)^{-s} \d \left\| E_\ld \left( \f{2^{-a} \Gamma(s)}{\Gamma(1 - s)} \ z^{a}\dd_{z} \mc{F}_t U (\cdot,\sigma,z) - \mc{F}_t \left( \mc{H}^s u \right)(\cdot, \sigma) \right) \right\|^2 \d\sigma \right)^{1/2}.\notag
	    \end{align}
	    Then  using \eqref{cr1}  and also by    using properties of the projection operators  $\{E_\lambda\}$ we infer
 	\begin{align*}
 	   & \left\|  \f{2^{-a} \Gamma(s)}{\Gamma(1-s)} z^{a} \dd_z U(\cdot,\cdot, z) - \mc{H}^s u \right\|_{\mb{H}^{-s}(\mb{R}^{n+1})} \\
 	    & \preceq \underset{|w| \le \ep}{\tn{sup}} \ \left|{w^{1-s} K_{1-s}(w) - 2^{-s}  \Gamma(1-s)} \right| \left(\int_{\mb{R}} \int_0^\iy |\ld + i\sigma|^{s} \d \| E_\ld (\mc{F}_t u) (\cdot,\sigma)\|^2\right)^{1/2}  \\
 	    & \hspace{2cm} + \underset{|w| > \ep}{\tn{sup}} \ \left|{w^{1-s} K_{1-s}(w) - 2^{-s}  \Gamma(1-s)} \right| \left(\int_{\mb{R}} \int_0^\iy \boldsymbol{\chi}_{z |\ld + i\sigma|^{\f{1}{2}} > \ep} |\ld + i\sigma|^{s} \d \|E_\ld (\mc{F}_t u) (\cdot,\sigma)\|^2\right)^{1/2}.
 	\end{align*}	
Similarly as before, we first take $ z \to 0+ $ and then let $\ep \to 0$ to assert that
 	\begin{align}\label{kj}
 	     \left\| \f{2^{-a} \Gamma(s)}{\Gamma(1-s)} z^{a} \dd_z U(\cdot,\cdot, z) - \mc{H}^s u \right\|_{\mb{H}^{-s}(\mb{R}^{n+1})} \to 0, \qd \tn{ as } z \to 0+,
 	\end{align}
	for $u \in \mathcal S(\mathbb{R}^{n+1})$.

	    Now let $u_k \to u$ in $\mb{H}^s$ where $u_k$'s are in $\mathcal S(\mathbb{R}^{n+1})$.  We denote by $U_k$ and $U_l$ the solutions to the extension problem \eqref{exten prob} corresponding  to Dirichlet data $u_k$ and $u_l$ respectively. Then using \eqref{weighted norm der} and by an analogous argument as in \eqref{bl1}  we find as $k\to \infty$
	    
	    \begin{align}\label{cc}
	    &\left\| \f{2^{-a} \Gamma(s)}{\Gamma(1-s)} z^{a} \dd_z (U_k - U)(\cdot,\cdot, z)  \right\|_{\mb{H}^{-s}(\mb{R}^{n+1})}\\
	    &   \preceq \underset{\mathbb R^+}{\tn{sup}} \ \left|w^{1-s} K_{1-s}(w) \right| \left(\int_{\mb{R}} \int_0^\iy |\ld + i\sigma|^{s} \d \| E_\ld (\mc{F}_t (u_k -u) (\cdot,\sigma))\|^2 \right)^{1/2}  \to 0,\notag \\
	    & \text{(since $u_k \to u$ in $\mb{H}^s$)}.\notag
 	  	    	    \end{align}   	    
	    
	Thus  $\{z^a \dd_z U_k\}$  is uniformly Cauchy in $z$ as $z \to 0^+$. This fact combined with \eqref{kj} implies   ii) in a standard way.

 	%
 	
 	%
 	
 	%

	Now we plan to demonstrate the energy estimate 
 	\begin{align}\label{enr}
 	    \|U\|_{\mc{L}^{1,2} \left( \mb{R}^{n+1} \times (0,M); z^a \d x \d t \d y \right)} \preceq_{M} \|u\|_{\mb{H}^{s} (\mb{R}^{n+1})}.
 	\end{align}
 	We closely follow the arguments from \cite{BGMN} to establish \eqref{enr}. We will not be concerned with proving 
 	\begin{align*}
 	    \|U\|_{L^2 (\mb{R}^{n+1} \times (0,M); z^a \d x \d t \d y)} \preceq_{M} \|u\|_{\mb{H}^s (\mb{R}^{n+1})}.
 	\end{align*}\ wer
 	as this is already covered in \cite[(3.15)]{BGMN}.  We first estimate the term
 	\begin{align*}
 	   \left\| z^{\f{a}{2}} \dd_z U \right\|_{L^2 (\mb{R}^{n+1} \times (0,M))} & = \int_0^M \int_\mb{R} z^{a} \| \dd_z \mc{F}_t U (\cdot,\sigma,y) \|_{L^2(\mb{R}^n)}^2 \d \sigma \d z.
 	\end{align*}   
  By using \eqref{weighted norm der} and \eqref{b1}, we find that such a term equals
 	 \begin{align*}
 	   & \int_0^M \int_\mb{R} \int_0^\iy z^{-a} |\ld + i\sigma|^{2s} \left\vert \int_0^\iy \f{e^{-\theta} e^{-(\ld + i\sigma)\f{z^2}{4\theta}}}{\theta^s} \d \theta \right\vert^2 \d\|E_{\ld} \mc{F}_t u(\cdot,\sigma)\|^2 \d\sigma \d z \\
 	  & \simeq \int_0^M \int_\mb{R} \int_0^\iy  z^{-a} |\ld + i\sigma|^{2s} \left\vert \left( {z \rt{\ld + i\sigma}} \right)^{1 - s} K_{1 - s} \left( z\rt{\ld+i\sigma} \right) \right\vert^2 \d\|E_{\ld} \mc{F}_t u(\cdot,\sigma)\|^2 \d\sigma \d z
	  &= I_1+I_2,
 	\end{align*}
	where $I_1$ is the integral over the region where  $|z \rt{\ld + i\sigma}| \le 1$ and    $I_2$ is the integral over the complement. 	We first estimate $I_1$ as follows. 	%
 	\begin{align}
 	  \nn I_1= &  \int_\mb{R} \int_0^\iy \int_0^{|\ld + i\sigma|^{-\f{1}{2}}} z^{-a} |\ld + i\sigma|^{2s} \left\vert \left( {z \rt{\ld + i\sigma}} \right)^{1 - s} K_{1 - s} \left( z \rt{\ld+i\sigma} \right) \right\vert^2 \d\|E_{\ld} \mc{F}_t u(\cdot,\sigma)\|^2 \d\sigma \d z \\
 	 \nn  & \preceq \underset{|z| \le 1}{\tn{sup}} \left\vert z ^{1 - s} K_{1 - s}(z) \right\vert^2 \int_\mb{R} \int_0^\iy \left( \int_0^{|\ld + i\sigma|^{-\f{1}{2}}} z^{2s-1} \d z \right) |\ld + i\sigma|^{2s} \d\|E_{\ld} \mc{F}_t u(\cdot,\sigma)\|^2 \d\sigma \\
 	  \label{energy 1} &  \preceq \int_\mb{R} \int_0^\iy |\ld + i\sigma|^{s} \ \d\|E_{\ld} \mc{F}_t u(\cdot,\sigma)\|^2 \d\sigma \preceq \|u\|_{\mb{H}^s (\mb{R}^{n+1})}.
 	\end{align}
 	Likewise, $I_2$ can be estimated as
 	\begin{align}
 	  \nn I_2 &= \int_\mb{R} \int_0^\iy \int_{|\ld + i\sigma|^{-\f{1}{2}}}^\iy z^{-a} |\ld + i\sigma|^{2s} \left\vert \left( {z \rt{\ld + i\sigma}} \right)^{1 - s} K_{1 - s} \left(z \rt{\ld+i\sigma} \right) \right\vert^2 \d\|E_{\ld} \mc{F}_t u(\cdot,\sigma)\|^2 \d\sigma \d z \\
 	  \nn  & \preceq \int_\mb{R} \int_0^\iy \left(\int_{|\ld + i\sigma|^{-\f{1}{2}}}^\iy z  \left\vert K_{1 - s}\left(z \rt{\ld+i\sigma} \right) \right\vert^2 \ \d z \right) |\ld + i\sigma|^{1+s} \ \d\|E_{\ld} \mc{F}_t u(\cdot,\sigma)\|^2 \d\sigma \\
 	  \nn  & \preceq \int_\mb{R} \int_0^\iy \left(\int_{|\ld + i\sigma|^{-\f{1}{2}}}^\iy e^{-z |\ld + i\sigma|^{\f{1}{2}}}\f{\d z}{|\ld + i\sigma|^{\f{1}{2}}} \right) |\ld + i\sigma|^{1+s} \ \d\|E_{\ld} \mc{F}_t u(\cdot,\sigma)\|^2 \d\sigma \\
 	  \label{energy 2}  & \preceq \int_\mb{R} \int_0^\iy \left( \int_{1}^\iy e^{-m} \ \d m \right) |\ld + i\sigma|^{s} \ \d\|E_{\ld} \mc{F}_t u(\cdot,\sigma)\|^2 \d\sigma \preceq \|u\|_{\mb{H}^s (\mb{R}^{n+1})}.
 	\end{align}
 	where we have used the asymptotic 
 	$ \vert K_{1-s}(z) \vert^2 = \large{O}\left( \f{e^{-|z|}}{|z|} \right)$ for $|z| \ge 1.$
 	See for instance \cite[5.11.10]{Le}. Combining  \eqref{energy 1} and \eqref{energy 2}, we obtain $\left\| z^{\f{a}{2}} \dd_z U \right\|_{L^2 (\mb{R}^{n+1} \times (0,M))} \preceq \|u\|_{\mb{H}^s(\mb{R}^{n+1})}$. Now we estimate the term $\|\nabla_xU\|_{L^2 (\Sigma; z^a \d x \d t \d z)}$. For that, we note that from the resolution of the famous Kato square root problem as in \cite{AHLMT} and \cite{AHLT}, we have for a smooth function $f$ decaying rapidly at infinity in $\mb R^n$  that the following holds
	\begin{equation}\label{kres}
	||\nabla_x f||_{L^{2}(\mb R^n)} \approx ||(-\mc{L})^{1/2} f||_{L^{2}(\mb R^n)}.
	\end{equation}
	Combining this with Plancherel theorem in the $t$ variable we find	 	%
 	\begin{align*}
 	   & \|\nabla_x U\|_{L^2 (\Sigma; z^a \d x \d t \d z)} \approx \|z^{\f{a}{2}} (-\mc{L})^{\f{1}{2}} \mc{F}_t U\|_{L^2 (\mb{R}^{n+1} \times (0,M))}\ \text{(using Plancherel theorem in the $t$ variable)} \\
 	    & = \int_0^M \int_\mb{R} \int_0^\iy z^a \ld \ \d\|E_{\ld} \mc{F}_t U(\cdot,\sigma)\|^2 \d\sigma \d z \\
 	   & = \int_0^M \int_\mb{R} \int_0^\iy z^a \ld z^{4s} \left\vert \int_0^\iy e^{-\f{z^2}{4\tau}} e^{-(\ld+i\sigma)\tau} \f{\d\tau}{\tau^{1+s}}\right\vert^2  \d\|E_{\ld} \mc{F}_t u(\cdot,\sigma)\|^2 \d\sigma \d z \\
 	    & = \int_0^M \int_\mb{R} \int_0^\iy\ld z^{1+2s} |\ld + i\sigma|^{2s} \left\vert \int_0^\iy e^{-\theta} e^{-(\ld + i\sigma) \f{z^2}{4\theta}} \f{\d\theta}{\theta^{1+s}}\right\vert^2 \d\|E_{\ld} \mc{F}_t u(\cdot,\sigma)\|^2 \d\sigma \d z \\
 	    & \approx \int_0^M \int_\mb{R} \int_0^\iy \ld z^{1-2s} \left\vert \left( z \rt{\ld+i\sigma} \right)^{s} K_{s} \left(z \rt{\ld+i\sigma} \right) \right\vert^2 \d\|E_{\ld} \mc{F}_t u(\cdot,\sigma)\|^2 \d\sigma \d z\ \text{(using \eqref{Bessel func} with $\nu=-s$)}
	    \\&=J_1+J_2,
 	\end{align*}
 	where $J_1$ is integration over the region where  $|z \rt{\ld + i\sigma}| \le 1$ and    $J_2$ is the integral over the complement. We find
 	\begin{align}
 	 \nn J_1=  & \int_\mb{R} \int_0^\iy \int_0^{|\ld + i\sigma|^{-\f{1}{2}}} \ld z^{1-2s} \left\vert \left( z \rt{\ld+i\sigma} \right)^{s} K_{s} \left(z \rt{\ld+i\sigma} \right) \right\vert^2 \d\|E_{\ld} \mc{F}_t u(\cdot,\sigma)\|^2 \d\sigma \d z \\
 	  \nn & \preceq \underset{|z| \le 1}{\tn{sup}} \left\vert z^{s} K_{s}(z) \right\vert^2 \int_\mb{R} \int_0^\iy \left( \int_0^{|\ld + i\sigma|^{-\f{1}{2}}} z^{1-2s} \d z \right) \ld \ \d\|E_{\ld} \mc{F}_t u(\cdot,\sigma)\|^2 \d\sigma  \\
 	  \label{energy 3}  & \preceq \int_\mb{R} \int_0^\iy |\ld + i\sigma|^{s-1} \ld \ \d\|E_{\ld} \mc{F}_t u(\cdot,\sigma)\|^2 \d\sigma \preceq \|u\|_{\mb{H}^s (\mb{R}^{n+1})}.
 	\end{align}
 	Likewise, $J_2$ is estimated in the following way.
 	\begin{align}
 	  \nn J_2=& \int_\mb{R} \int_0^\iy \int_{|\ld + i\sigma|^{-\f{1}{2}}}^\iy \ld z^{1-2s} \left\vert \left( z \rt{\ld+i\sigma} \right)^{s} K_{s} \left(z \rt{\ld+i\sigma} \right) \right\vert^2 \d\|E_{\ld} \mc{F}_t u(\cdot,\sigma)\|^2 \d\sigma \d z \\
 	  \nn  & \preceq \int_\mb{R} \int_0^\iy \left( \int_{|\ld + i\sigma|^{-\f{1}{2}}}^\iy  z \left\vert K_{s}\left( z \rt{\ld+i\sigma} \right) \right\vert^2 \d z \right) \ld |\ld + i\sigma|^{s} \d\|E_{\ld} \mc{F}_t u(\cdot,\sigma)\|^2 \d\sigma \d z \\
 	  \nn & \preceq \int_\mb{R} \int_0^\iy \left(\int_{|\ld + i\sigma|^{-\f{1}{2}}}^\iy e^{-z |\ld + i\sigma|^{\f{1}{2}}}\f{\d z}{|\ld + i\sigma|^{\f{1}{2}}} \right) |\ld + i\sigma|^{1+s} \ \d\|E_{\ld} \mc{F}_t u(\cdot,\sigma)\|^2 \d\sigma \\
 	   \label{energy 4} & \preceq \int_\mb{R} \int_0^\iy \left( \int_{1}^\iy e^{-m} \ \d m \right) |\ld + i\sigma|^{s} \ \d\|E_{\ld} \mc{F}_t u(\cdot,\sigma)\|^2 \d\sigma \preceq \|u\|_{\mb{H}^s(\mb{R}^{n+1})}.
 	\end{align}
 	where we have again used the asymptotic 
 	$ \vert K_{s}(z) \vert^2 = \large{O}\left( \f{e^{-|z|}}{|z|} \right)$ for $|z| \ge 1$.
 	From the inequalities \eqref{energy 3} and \eqref{energy 4}, we conclude that $\|\nabla_x U\|_{L^2 (\mb{R}^{n+1} \times (0,M); z^a \d x \d t \d z)} \preceq \|u\|_{\mb{H}^s(\mb{R}^{n+1})}$. 	This finishes the proof of the theorem.

 	\end{proof}

	\subsection{Fundamental solution of the extension problem}\label{fundsol}

   We now  introduce the fundamental solution $ \mc{G}(Y,X,t) $ associated to the extended operator
   \begin{align*}
   	\mc{L}_a & := x_{n+1}^a \dd_t - \tn{div} \left( x_{n+1}^a A(x,t) \nabla \right)
   \end{align*}

   $X=(x, x_{n+1}), Y= (y, y_{n+1})$ will denote  generic points in $\mb{R}^n \times \mb{R}.$  It is to be noted that   $x_{n+1}$ will play the role of extension variable $z$ that was introduced in Subsection \ref{ext}.
   
   For a function $f$, we let
   \begin{equation}\label{normal}
   	\py f= \operatorname{lim}_{x_{n+1} \to 0^+} x_{n+1}^a \partial_{n+1} f(x, x_{n+1}).
   \end{equation}
   This limit is interpreted in $\mb{H}^{-s}$ sense in Theorem \ref{energy} ( where $z$ plays the role of $x_{n+1}$)  but would be eventually interpreted in the strong point wise sense in Section \ref{Carleman estimate and Unique continuation} once we have the regularity result in Lemma \ref{reg}.

   We now recall that it was shown in \cite{Gcm} that given $\phi \in C_0^{\infty}(\mb{R}^{n+1}_+)$ the solution of the Cauchy problem with Neumann condition
   \begin{align}\label{CN}
   	\begin{cases}
   		\mc{L}_a U=0 \hspace{2mm} &\text{in}  \hspace{2mm}\mb{R}^{n+1}_+ \times (0,\infty)\\
   		U(X,0)=\phi(X), \hspace{2mm} & X \in \mb{R}^{n+1}_+,\\
   		\py U(x, 0, t) = 0 \hspace{2mm} & x\in \Rn,\ t \in (0, \infty)
   	\end{cases}
   \end{align}
   is given by the formula 
   \begin{equation}\label{Ga}
   	\mathscr P^{(a)}_t \phi(Y) \overset{def}{=} U(Y,t) = \int_{\mb{R}^{n+1}_+} \phi(X) \mc{G} (Y,X,t) x_{n+1}^a dX,
   \end{equation}
   where 
   \begin{equation}\label{fund}
   	\mc{G}(Y,X,t) = p(y,x, t) \ p_a(x_{n+1},y_{n+1};t), \end{equation}  and where $p(y,x,t)$ is the heat-kernel associated to $\left(\dd_t - \tn{div}(A(x,t)\nabla_{x})\right)$ as in \eqref{stoc} and $p_a$ is the fundamental solution of  the Bessel operator $\dd^2_{x_{n+1}} + \f{a}{x_{n+1}}\dd_{x_{n+1}}$. Such a function $p_a$ is given by the formula
   \begin{equation}\label{pa}
   	p_a(x_{n+1}, y_{n+1}; t) = (2t)^{- \f{1+a}{2}} e^{-\f{x_{n+1}^2 + y_{n+1}^2}{4t}} \left( \f{x_{n+1} y_{n+1}}{2t} \right)^{ \f{1-a}{2}} I_{\f{a-1}{2}}\left( \f{x_{n+1} y_{n+1}}{2t} \right),
   \end{equation}
   where $I_\nu (z)$ the modified Bessel function of the first kind defined by the series
   
   \begin{align}\label{besseries}
   	I_{\nu}(w) = \sum_{k=0}^{\infty}\frac{(w/2)^{\nu+2k} }{\Gamma (k+1) \Gamma(k+1+\nu)}, \hspace{4mm} |w| < \infty,\; |\operatorname{arg} w| < \pi.
   \end{align}
   
   It follows from \eqref{stoc} that
   \begin{equation}\label{stoc1}
   	\int_{\mb{R}^{n+1}_+} x_{n+1}^a \mc{G}(Y,X,t) dX=1,
   \end{equation}
   and also \begin{equation}\label{st}\mathscr P^{(a)}_t \phi(X) \underset{t\to 0^+}{\longrightarrow} \phi(X).\end{equation}
   
   We finally record the following Gaussian bounds for $p(y,x, t)$  as in \cite{Ar}  which will be needed in Section \ref{Carleman estimate and Unique continuation}    %
   \begin{align}\label{gbd}
   	&  \f{1}{N_0 t^{n/2}} e^{-\f{N_0 |x - y|^2}{t}} \le p(y,x;t) \le \f{N_0}{t^{n/2}} e^{-\f{|x - y|^2}{N_0 t}}.\end{align}

 	  	\subsection{The initial-exterior problem for $\mc{H}^s + q(x,t)$} 		
 	In this subsection, we discuss some well-posedness  results for  the forward problem \eqref{ini-ext prob}. More generally, we consider the presence of a non-trivial source term in the PDE, i.e. we look for  existence and uniqueness results of the problem 
 	\begin{align}\label{gen in-ext prob}
    \begin{cases}
      \left( \mc{H}^s + q(x,t) \right)u = F, & \tn{ in } Q:= \Omega \times (-T,T)\\
      u(x,t) = f(x,t), & \tn{ in } Q_e:= \Omega_e \times (-T,T) \\
      u(x,t) = 0, & \tn{ for } t \leq -T,
   \end{cases}
   \end{align}
   where $\Omega_e$ denotes the complement of $\Omega$.

 	We consider the bilinear map $\mc{B}_q(\cdot, \cdot)$ on $ \mb H^s(\mb{R}^{n+1}) \times \mb H^s(\mb{R}^{n+1})$ defined by
 	\begin{align*}
 	    \mc{B}_q \left( f,g \right) = \langle \mc{H}^{\f{s}{2}}f, \mc{H}_*^{\f{s}{2}} g \rangle + \int_{Q} q f \overline g .
 	\end{align*}
 	It follows by an application of Cauchy-Schwartz inequality that
 	\begin{align}\label{continuity}
 	    \lvert \mc{B}_q \left( f,g \right) \rvert \preceq \|f\|_{\mb H^s (\mb{R}^{n+1})} \|g\|_{\mb H^s(\mb{R}^{n+1})}.
 	\end{align}
   Akin to that in \cite{CLR},  it turns out that we need to study a time localized problem.  Therefore to this end, we introduce the notation 
   \begin{align}\label{tct}
       u_T(x,t) = u(x,t) \chi_{[-T,T]}(t)
   \end{align}
 	and note that, $u_T \in \mb{H}^s(\mb{R}^{n+1})$ whenever $u\in \mb{H}^s(\mb{R}^{n+1})$. This follows from the fact that $\chi_{[-T,T]}$ is a multiplier in the Sobolev space $H^{\gamma}(\mathbb R)$ for $|\gamma| \leq \frac{1}{2}$. See for instance \cite[Theorem 11.4 in Chapter 1]{LM}.  Thus we cast all the upcoming analysis for $u_T$ as we can only guarantee the uniqueness  up to $ t = T $.  Also it is to be noted that  from the representation of $\mc{H}^s$ as in \eqref{defn of frac parabolic} it follows that $\mc{H}^s u (x,t) = \mc{H}^s (\chi_{(-\infty, T]} u) (x,t)$ for all $t \leq T$.

	Below we simply denote the distributional pairing $ \langle \cdot, \cdot\rangle_{\chd, \ch} $ by $ \langle \cdot, \cdot \rangle$ where $\chd$ denotes the dual space. 	%
 	\begin{definition}(Weak solutions)\label{wk}
 	    Consider $\Omega$ to be an open bounded set in $\mb{R}^{n}$ and $T>0$. For $F \in \chd$ and $f \in \mc{H}^s\left(Q_e\right)$, we say $u\in \mb{H}^s(\mb{R}^{n+1})$ to be a weak solution of \eqref{gen in-ext prob} if  $ v:= \left(u-f\right)_{T} \in \ch$ and
 	    \begin{align*}
 	        \mc{B}_q \left( u, \phi \right) = \langle F,\phi \rangle, \qd  \fa \phi \in \ch
 	    \end{align*}
 	    or equivalently,
 	    \begin{align*}
 	        \mc{B}_q \left( v, \phi \right) = \langle F - (\mc{H}^s+q) f,\phi \rangle,  \qd  \fa \phi \in \ch.
 	    \end{align*}
 	\end{definition}
 	    
 	Now, we state the well-posedness results of the initial-exterior problem \eqref{ini-ext prob}.     
 	\begin{theorem}\label{well-posedness theorem for q}
 	 Let $\Omega$ be an open bounded set in $\mb{R}^{n}$ and $T>0$. Consider $q \in L^{\iy}(Q), f \in \mc{H}^s(Q_e)$ and $F\in \chd$. Then there exists a countable set of real numbers $\Sigma:= \{\ld_i\}_{1\le i \le \iy}$ such that $ \ld_1 \le \ld_2 \le \ld_3 \le ... \to \iy $ such that given $\ld \in \mb{R} \setminus \Sigma$, there exists a unique  $u_T \in \mb{H}^s(\mb R^{n+1})$  with $(u-f)_T \in \ch$ for which  	 %
 	 \begin{align*}
 	    \left(\mc{H}^s + q(x,t) - \ld \right)u_T = F, \qd \tn{ in } Q.\\
 	     	 \end{align*}
 	 Moreover $u_T$ satisfies $\|u_T\|_{\mb{H}^s(\mb{R}^{n+1})} \preceq \left(\|F\|_{\chd} + \|f\|_{\mc{ H}^s(Q_e)}\right)$. 
 	\end{theorem}
 	\begin{proof}
 	We argue as in Theorem 3.2 in \cite{CLR}. Let $ v:= (u-f)_T  $ and $ \widetilde{F} := F - \left(\mc{H}^s + q \right)f $. 
 	%
 	%
 	 We first show the coercivity of the bilinear map $ \mc{B}_{q}(v,w) + \mu \int_Q vw $, for $\mu$ large enough.  For $ w \in \ch $ and $ \mu \ge \|\tn{min}\{q,0\}\|_{L^\iy(Q)} $, we notice that
 	\begin{align}\label{io1}
 	   \mc{B}_{q}(w,w) & + \mu \int_Q |w|^2(x,t) \d x\d t  = \langle \mc{H}^{\f{s}{2}}w, \mc{H}_*^{\f{s}{2}} w \rangle + \int_Q \left( \mu + q(x,t) \right) |w|^2(x,t) \ \d x\d t \\
 	 \nn & = \int_\mb{R} \int_{0}^\iy \left( \ld + i \sigma \right)^s \d \|E_\ld (\mc{F}_t w)(\cdot,\sigma)\|^2 \d\sigma + + \int_Q \left( \mu + q(x,t) \right) |w|^2(x,t) \ \d x\d t \\
 	  \nn  & = \int_\mb{R} \int_{0}^\iy \left\vert \ld + i \sigma \right\vert^s( \cos (s\theta) + i \operatorname{sin}(s\theta)) \d \|E_\ld (\mc{F}_t w)(\cdot,\sigma)\|^2 \d\sigma + \int_Q \left( \mu + q(x,t) \right) |w|^2(x,t) \ \d x\d t\\
	  \nn  & = \int_\mb{R} \int_{0}^\iy \left\vert \ld + i \sigma \right\vert^s \cos (s\theta) \d \|E_\ld (\mc{F}_t w)(\cdot,\sigma)\|^2 \d\sigma + \int_Q \left( \mu + q(x,t) \right) |w|^2(x,t) \ \d x\d t,	  
	  \end{align}
	  where $\operatorname{tan}(\theta)= \frac{\sigma}{\lambda}$ and where we utilized that $\operatorname{sin}(s\theta)$ is an odd function  in the last identity above.  Since $\lambda \geq 0$, it is seen that $\theta \in (-\pi/2, \pi/2)$ and thus for a fixed $0<s<1$ 
	  \begin{equation}\label{io}
	  \operatorname{cos}(s\theta) \geq \cos(s \pi/2) \overset{def}= c_s>0.\end{equation}  
	  Using \eqref{io} along with \eqref{iden} in  \eqref{io1} we obtain
	  \begin{align}\label{ct1}
	&   \mc{B}_{q}(w,w)  + \mu \int_Q |w|^2(x,t) \d x\d t\\
	    	 \nn & \succeq \int_{\mb{R}^{n+1}} \left\vert |\xi|^2 + i\rho \right\vert^s \left\vert \hat{w}(\xi, \rho) \right\vert^2 \d \xi\d \rho  \succeq \int_{\mb{R}} \left\| \left( -\Delta_x \right)^{\f{s}{2}} \mc{F}_t w (\cdot, \sigma) \right\|^2_{L^2(\mb{R}^n)} \d\sigma \succeq \|w\|_{L^2(\mb{R}^{n+1})}
 	\end{align}
 	where in the last inequality, we  used Hardy-Littlewood-Sobolev inequality in $x$ variable combined with the fact $w$ is compactly supported. Thus from \eqref{continuity} and \eqref{ct1} we conclude that the bilinear form $\mc{B}_{q}(v,w) + \mu \int_Q vw$ is coercive and bounded. Thus by Lax-Milgram theorem, there is a unique solution $ v = G_\mu \widetilde{F} \in \ch $ which satisfies
 	\begin{align*}
 	    \mc{B}_{q}(v,w) + \mu \int_Q vw = \langle \widetilde{F}, w \rangle, \qd \fa w \in \ch.
 	\end{align*}
 	alongwith the bound
 	\begin{align}\label{stability of solution}
 	    \|v\|_{\ch} \preceq \|\widetilde{F}\|_{\chd}
 	\end{align}
 	From \eqref{stability of solution}, we find $ \|u_T\|_{\mb{H}^s(\mb{R}^{n+1})} \preceq \left( \|F\|_{\chd} + \|f\|_{\mc{H}^s(Q_e)} \right)$.
 	In particular, \eqref{stability of solution} implies that the source to solution map i.e $G_{\mu}: \chd \to \ch$ is continuous.  Thus using  \eqref{iden}, by an  application of the compact Sobolev embedding we deduce that
 	\begin{align*}
 	    G_\mu : L^2(Q) \to L^2(Q)
 	\end{align*}
 	is a compact operator and therefore  by  the spectral theorem,  there exists a countable set of eigenvalues of $G_{\mu}$  which are $\f{1}{\ld_i + \mu}$ with $ \ld_i \to \iy$. This is evident from the following observation
 	\begin{align*}
 	    \mc{B}_{q}(v,w) - \ld \int_{Q} vw = \langle \widetilde{F} + (\mu+\ld)v, w\rangle
 	\end{align*}
 	Also it is not hard to see $ \Sigma: = \{\ld_i\}_{1 \le i \le \iy} \subseteq \left( - \|\tn{min}\{q,0\}\|_{L^{\iy}(Q)}, \iy \right)$. Finally, the Fredholm alternative ensures the existence and uniqueness of the problem under consideration.
 	\end{proof} 	
 	\tb{Remark}: In view of Theorem \ref{well-posedness theorem for q}, we could rephrase the eigenvalue condition \eqref{asus} by saying $0 \notin \Sigma$. It follows from the inequality  \eqref{ct1} that for non-negative potentials i.e. when $ q \ge 0$ a.e. in $Q$, we have $\Sigma \subset \mb{R}_{+}$. The same assertion holds for small enough potentials i.e. when $\|q\|_{L^\iy(Q)}$ is small.

	 	Similarly, one can prove the well-posedness results for  the adjoint equation to \eqref{ini-ext prob} which is the future-exterior problem
 	\begin{align}\label{adjoint problem for q}
 	\begin{cases}
      \left( \mc{H}_{*}^s + q(x,t) \right)u^* = 0, & \tn{ in } Q:= \Omega \times (-T,T)\\
      u^*(x,t) = f(x,t), & \tn{ in } Q_e:= \Omega_e \times (-T,T) \\
      u^*(x,t) = 0, & \tn{ for } t \geq T.
    \end{cases}
 	\end{align}
 	The analysis here would be identical to that of initial-exterior problem \eqref{ini-ext prob} and we could have similar well-posedness result here also.  Moreover we observe that if we let	\[
	\tilde u(x,t)= u(x,-t)
	\]
	then 
	\begin{equation}\label{k}
	(\mc{H}_{*}^s u)(x,t)= (\mc{H}^s \tilde u)(x, -t).
	\end{equation}
 	Moreover from \eqref{asus}  and Fredholm alternative it follows that
	\begin{equation}\label{asus1}
	\text{$0$ is not a Dirichlet eigenvalue for the adjoint problem \eqref{adjoint problem for q}}.
	\end{equation}

 	\subsection{The initial-exterior problem for $\mc{H}^s +\langle b(x,t), \nabla_{x} \rangle+ q(x,t)$} We will introduce the notion of weak solutions for the problem \eqref{ini-ext prob for b,q}. For the weak formulation, we define   the corresponding bilinear form as follows%
 	\begin{align*}
 	    \mc{B}_{b,q} (f,g) = \langle \mc{H}^{\f{s}{2}}f, \mc{H}_*^{\f{s}{2}} g \rangle + \int_Q \langle b(x,t), \nabla_{x} f \rangle \ g + \int_{Q} q f \overline g , 
 	\end{align*}
 	where $ b \in L^\iy\left((-T,T); W^{1-s,\iy}(\Omega) \right)$ and $q\in L^\iy(Q)$.  Similar to that in \cite{CLR}, a Kato-Ponce type inequality will be used to obtain the boundedness of the term $\int_Q \langle b(x,t),\nabla_{x} f  \rangle g$ (see  \eqref{bdo} below). We now define the weak formulation of \eqref{gen in-ext prob}. 
 	\begin{definition}
 	    Let $\Omega$ be a  Lipschitz  domain  in $\mb{R}^{n}$, $s > \f{1}{2}$ and $T>0$. For $F \in \chd$ and $f \in \mc{H}^s\left(Q_e\right)$, we say that  $u\in \mb{H}^s(\mb{R}^{n+1})$ is a  weak solution of \eqref{gen in-ext prob} if   $ v:= \left(u-f\right)_{T} \in \ch$ and   
 	    \begin{align*}
 	        \mc{B}_{b,q} \left( u, \phi \right) = \langle F,\phi \rangle, \qd  \fa \phi \in \ch
 	    \end{align*}
 	    or equivalently,
 	    \begin{align*}
 	        \mc{B}_{b,q} \left( v, \phi \right) = \langle F - (\mc{H}^s + \langle b, \nabla_{x} \rangle + q) f,\phi \rangle, \qd  \fa \phi \in \ch.
 	    \end{align*}
 	\end{definition}
 	
 	Now, we state and prove  the well-posedness result for \eqref{ini-ext prob for b,q}. 
 	\begin{theorem}\label{well-posedness theorem for b,q}
 	 Let $\Omega$ be a Lipschitz domain in $\mb{R}^{n}$, $s > \f{1}{2}$ and $T>0$. 
 	 Assume $b\in L^\iy\left((-T,T); W^{1-s,\iy}(\Omega) \right)$,\\ $q \in L^{\iy}(Q), f \in \mc{H}^s(Q_e)$ and $F\in \chd$. Then there exists a countable set of real numbers $\Sigma:= \{\ld_i\}_{1\le i \le \iy}$  with $ \ld_1 \le \ld_2 \le \ld_3 \le ... \to \iy $ such that whenever  $\ld \in \mb{R} \setminus \Sigma$,    there exists a unique solution solution $u_T \in \mb{H}^s(\mb{R}^{n+1})$ satisfying
 	 \begin{align*}
 	 \begin{cases}
 	    \left(\mc{H}^s + \langle b(x,t), \nabla_{x} \rangle+ q(x,t)  - \ld \right)u_T = F, \qd \tn{ in } Q,\\
 	    u_T(x,t) = f(x,t), \qd \tn{ for } (x,t) \in Q_e.
 	 \end{cases}
 	 \end{align*}
 	 such that $\|u_T\|_{\mb{H}^s(\mb{R}^{n+1})} \preceq \left(\|F\|_{\chd} + \|f\|_{\mc{ H}^s(Q_e)}\right)$. 
 	\end{theorem}
 	\begin{proof}
 	It suffices to show the boundedness of $\mc{B}_{b,q}$ and the coercivity of $\mc{B}_{b,q}(w,w) + \mu \int_Q |w|^2$ for large enough $\mu$ and then one can argue similarly as in the proof of Theorem \ref{well-posedness theorem for q}. For boundedness, we need to show 
   \[ 
   \lvert \mc B_{b,q} (u,v) \rvert \preceq \|u\|_{\mb H^s (\mb{R}^{n+1})} \|v\|_{\mb H^s (\mb{R}^{n+1})}.
   \]
   Now as before, we have
   \begin{align*}
       \left\vert \langle \mc{H}^s u ,\mc{H}_*^s v \rangle + \int_Q q uv \right\vert \preceq \|u\|_{\mb H^s (\mb{R}^{n+1})} \|v\|_{\mb H^s (\mb{R}^{n+1})}.
   \end{align*}
 Thus it suffices to show that 
   \begin{align}\label{Kato-Ponce}
    \left\vert \int_{Q} u(x,t) \langle b(x,t), \nabla_{x} v  \rangle\ \d x\d t \right\vert \preceq \ \|u\|_{\mb{H}^s(\mb{R}^{n+1})} \|v\|_{\mb{H}^s(\mb{R}^{n+1})}  
   \end{align}
   In order to establish \eqref{Kato-Ponce}, we argue as in \cite{CLR}. We first choose $B \in L^\iy\left((-T,T); W^{1-s,\iy}(\mb{R}^{n}) \right) $ with $ B = b $ a.e in $ Q $, such that 
   \[ \|B\|_{L^\iy\left((-T,T); W^{1-s,\iy}(\mb{R}^{n+1}) \right)} \le C \|b\|_{L^\iy\left((-T,T); W^{1-s,\iy}(\Omega) \right)} \]
   and notice the following estimate for all $t\in (-T,T)$ 
   \begin{align}\label{bdo}
   	\left\vert\int_\Omega u(x,t) \langle b(x,t), \nabla_{x} v(x,t) \rangle\ \d x \right\vert \\
   	& \preceq \|B(\cdot,t)u(\cdot,t)\|_{ \mb{H}^{1-s}(\mb{R}^{n})} \|\nabla_x v\|_{\mb{H}^{s-1} (\mb{R}^n)}\notag \\
   	& \preceq \|B\|_{L^\iy\left((-T,T); W^{1-s,\iy}(\mb{R}^{n}) \right)} \|u(\cdot,t)\|_{\mb{H}^{1-s}(\mb{R}^{n})} \|v(\cdot,t)\|_{\mb{H}^s(\mb{R}^{n})}\notag \\
   	& \preceq \|u(\cdot,t)\|_{\mb{H}^{s}(\mb{R}^{n})} \|v(\cdot,t)\|_{\mb{H}^s(\mb{R}^{n})}, \notag
   \end{align}
    where we crucially  used the assumption $ s > \f{1}{2} $ and  also employed the Kato-Ponce inequality  in \cite[Theorem 1]{GO_2014} to obtain
   \begin{align*}
       \|B(\cdot,t)u(\cdot,t)\|_{\mb{H}^{1-s}(\mb{R}^{n})} & \approx \left\| \mc{J}^{1-s}(B(\cdot,t)u(\cdot,t)) \right\|_{L^2(\mb{R}^{n})} \\ 
      & \preceq \|B(\cdot,t)\|_{L^\iy (\mb{R}^{n})} \|\mc{J}^{1-s} u(\cdot,t)\|_{L^2(\mb{R}^{n})} + \|\mc{J}^{1-s} B(\cdot,t)\|_{L^\infty(\mb{R}^{n})} \|u\|_{L^2 (\mb{R}^{n})}  \\
      & \preceq \|B\|_{L^\iy\left((-T,T); W^{1-s,\iy}(\mb{R}^{n}) \right)} \|u(\cdot,t)\|_{\mb{H}^s(\mb{R}^{n})} \preceq \|u(\cdot,t)\|_{\mb{H}^s(\mb{R}^{n})}.
   \end{align*}
   Here $ \mc{J}:= \left( \tn{Id} - \Delta_x\right)^{\f{1}{2}}$.
   Now we use Cauchy-Schwarz inequality and  the Plancherel theorem in the $t$ variable to find 
   \begin{align*}
   &	\left\vert \int_{Q} u(x,t) \langle b(x,t), \nabla_{x} v \rangle \ \d x\d t \right\vert  \preceq \int_{\mb{R}} \|u(\cdot,t)\|_{H^{s}(\mb{R}^{n})} \|v(\cdot,t)\|_{H^s(\mb{R}^{n})} \d t \\
   	& \preceq \left(\int_{\mb{R}^{n+1}} \left( 1 + |\xi|^{2s}\right) |\mc{F}_x u|^2(\xi,t) \d\xi\d t\right)^{1/2} \left(\int_{\mb{R}^{n+1}} \left( 1 + |\xi|^{2s}\right) |\mc{F}_x v|^2(\xi,t) \d\xi \d t\right)^{1/2} \\
   	& \preceq \|u\|_{\mb{H}^s(\mb{R}^{n+1})} \|v\|_{\mb{H}^s(\mb{R}^{n+1})}.
   \end{align*}
    Next we head towards proving the coercivity of  $ \mc{B}_{b,q} (w,w) + \mu \int_Q |w|^2$.  In this regard we observe that
   \begin{align*}
       \langle \mc{H}^{\f{s}{2}}w, \mc{H}_*^{\f{s}{2}}w\rangle + \int_{Q} w \langle b,  \nabla_{x} w \rangle + \int_{Q} q |w|^2 \ge c_0 \|f\|^2_{\mb H^s(\mb{R}^{n+1})} - \left\vert \int_{Q} w \langle b, \nabla_{x,t} w \rangle\right\vert -  \|q\|_{L^\iy (Q)} \|w\|^2_{L^2(Q)}. 
   \end{align*}
As in \eqref{bdo}, we observe that
   \begin{align*}
      \left\vert \int_{Q} w \langle b, \nabla_{x} w \rangle\right\vert & \preceq \int_\mb{R} \|w(\cdot,t)\|_{\mb H^{1-s}(\mb{R}^{n})} \|w(\cdot,t)\|_{H^s(\mb{R}^{n})} \d t \\
       & \preceq \left(\int_{\mb{R}^{n+1}} \langle\xi\rangle^{2(1-s)} |\mc{F}_x w|^2(\xi,t) \d\xi\d t\right)^{1/2} \left(\int_{\mb{R}^{n+1}} \left( 1 + |\xi|^{2s}\right) |\mc{F}_x w|^2(\xi,t) \d\xi \d t\right)^{1/2}\\ & \text{(where $\langle \xi \rangle\overset{def}= \sqrt{1+|\xi|^2}$)}  \\
       & \preceq \left(\int_{\mb{R}^{n+1}} \langle\xi\rangle^{2s} |\mc{F}_x w|^2(\xi,t) \d\xi\d t\right)^{\f{1-s}{2s}} \left(\int_{\mb{R}^{n+1}} |\mc{F}_x w|^2(\xi,t) \d\xi\d t\right)^{\f{2s-1}{2s}} \|f\|_{\mb{H}^s(\mb{R}^{n+1})}\\
       &\preceq \|w\|^{\f{2s-1}{s}}_{L^2(\mb{R}^{n+1})} \|w\|^{\f{1}{s}}_{\mb{H}^s(\mb{R}^{n+1})}. 
   \end{align*}
    Now an application of Young's inequality gives
   \begin{align*}
        \left\vert \int_{Q} w \langle b, \nabla_{x} w \rangle\right\vert \le \ep \|w\|_{\mb{H}^s(\mb{R}^{n+1})}^2 + C_\ep \|w\|^2_{L^2(Q)}.
   \end{align*}
Thus by choosing $\epsilon$ small enough, we can conclude
   \begin{align*}
      \mc{B}_{b,q}(w,w)  \ge \f{c_0}{2} \|w\|^2_{\mb H^s(\mb{R}^{n+1})} - \left( C_\ep + \|q\|_{L^\iy(Q)} \right) \|w\|^2_{L^2(Q)}
   \end{align*}
   Finally by  choosing $ \mu \ge \left( C_\ep + \|q\|_{L^\iy(Q)} \right) $, we find that the coercivity of the corresponding bilinear form follows. As previously said, the rest of the proof remains the same as that for Theorem \ref{well-posedness theorem for q}.   %
\end{proof}

 	Likewise, we have the well-posedness result for the adjoint problem defined in the following way
 	\begin{align*}
 	    \begin{cases}
      \mc{H}_{*}^s u^* - \nabla_{x} \cdot (bu^*)+ q(x,t) u^* = 0, & \tn{ in } Q:= \Omega \times (-T,T)\\
      u^*(x,t) = f(x,t), & \tn{ in } Q_e:= \Omega_e \times (-T,T) \\
      u^*(x,t) = 0, & \tn{ for } t \geq T.
    \end{cases}
 	\end{align*}

 	
 	\section{Global unique continuation property} \label{Carleman estimate and Unique continuation}
	\subsection{Carleman estimate}
	In this section, we prove the global weak unique continuation result stated in Theorem \ref{GUCP}.  We follow the strategy  in \cite{ABDG} by first establishing a conditional doubling estimate for solutions  to the extension problem  and then we use a blowup argument to reduce the problem to a weak unique continuation property for the homogeneous extension problem with constant coefficients as in Theorem \ref{wkt}. Therefore, as a first step we prove a Carleman estimate for the extended operator. Keeping in mind  the possibility  of other applications,  we allow the matrix coefficients   to depend on both space and time variables in the Carleman estimate as stated in Lemma \ref{car1} below.

	Similar to that in Subsection \ref{fundsol},  in  this section, we remind the reader of  the following  notations.  $(X,t) = (x, x_{n+1}, t), (Y,s)= (y, y_{n+1}, s)$ will denote generic points in $\mb{R}^{n} \times \mb{R} \times \mb{R}$. For a given $r>0$, we will denote by $\mathbb B_r(Y)$, the Euclidean ball  in $\mathbb{R}^{n+1}$ of radius $r$ centered at $Y$ and $B_r(y) \overset{def}= \{x: (x, 0) \in \mathbb B_r(Y) \}$. Likewise we let $\mathbb B_r^+(Y) \overset{def}= \mathbb B_r(Y) \cap \{ X:x_{n+1}>0\}$. When the center $Y$ of $\mathbb B_r(Y)$ is not explicitly indicated, then we are taking $Y = 0$. Similar agreement for the thick half-balls $\mathbb B_r^+(x_0,0)$.

	For notational ease, $\nabla U$ and  $\operatorname{div} U$ will respectively refer to the quantities  $\nabla_X U$ and $ \operatorname{div}_X U$.  The partial derivative in $t$ will be denoted by $\partial_t U$ and also at times  by $U_t$. The partial derivative $\partial_{x_i} U$  will be denoted by $U_i$ and also by $\partial_i U$. 
	
	We will assume that 	$ A(x,t) := \left( a_{ij}(x,t) \right)_{ij} $ be a $ (n+1) \times (n+1) $ is a positive definite block matrix  valued function  satisfying \eqref{ellip} with 
	\begin{align}\label{matrix}
		A(0,0) = \mathbb I_{n+1}, \ a_{(n+1)i}(x,t) = \delta_{(n+1)i}, \ \ \fa i \in \{1,2,...,n+1\},
	\end{align} 
	such that  the following Lipschitz growth condition holds
	\begin{equation}\label{lip1}
		|A(x,t)- A(y,s)| \leq K (|x-y|+|t-s|).
	\end{equation}
	It follows from \eqref{lip1} that if we let  $ B(X,t) \equiv \{b_{ij}(x,t)\}_{1 \le i,j \le n+1} := A(x,t) - \mathbb I_{n+1}$, then
	\begin{align*}
		b_{ij}(x,t) = O(|x|+t) \ \fa i,j \in \{1,2,...,n+1\} \tn{ and } b_{(n+1)i}(x,t) = 0, \ \fa i \in \{1,2,...,n+1\}.
	\end{align*}

	\hspace{0.5mm}
	
	Corresponding to $A$ as in \eqref{matrix} above, we  consider the following extended backward parabolic operator 	
	\begin{align}\label{extop}
		\widetilde{\mc{H}} & := x_{n+1}^a \dd_t + \tn{div} \left( x_{n+1}^a A(x,t) \nabla \right).
	\end{align}	
	For notational convenience, it will be  easier to work with this backward extension operator in \eqref{extop} above. Similar to that in Subsection \ref{fundsol}, it should be clear to the reader that $x_{n+1}$ plays the role of the extension variable $z$ in Subsection \ref{ext}.


	In the proof of our Carleman estimate, we adapt the approach in the fundamental works  \cite{EF_2003} and \cite{EFV_2006} to our setting of degenerate operators as in \eqref{extop} and this has required some delicate adaptations. It is to be mentioned that although our method is inspired by ideas in \cite{EF_2003, EFV_2006},  nevertheless at a technical level, our proof of the Carleman estimate somewhat differs from that in \cite{EF_2003} even in the case when $a=0$. The proof of such an estimate in \cite{EF_2003} relies on first establishing a generic Rellich type identity  with respect to appropriate Carleman weights in the Gaussian space  ( see \cite[Lemma 1]{EF_2003}). This is then  combined with a clever use  of some  logarithmic inequalities  as stated in Lemma \ref{rl} below which is needed  to absorb certain error terms that arises due to  the perturbation of the variable coefficient principal part.   Differently from that in \cite{EF_2003}, in our proof we instead analyse the positivity property  of the associated conjugate operator directly. Our method however also uses  the ODE satisfied by the Carleman weight given in Lemma \ref{sigma} below in the same spirit as in \cite{EF_2003, EFV_2006}. We are of the opinion that our proof revisits the approach in \cite{EF_2003, EFV_2006} with a somewhat different viewpoint which can possibly be of independent interest.

	Before we state  our main Carleman estimate, we first  gather some  relevant results from \cite{EF_2003}  that are crucially needed in our context. The following  result which is Lemma 4 in \cite{EF_2003} is regarding the existence  of a suitable weight function $\sigma$ which has the    appropriate pseudo-convexity property   needed for the  Carleman estimate.
	
	\begin{lemma} \label{sigma}
		Let
		\begin{equation}\label{theta'}\theta(t) = t^{\f{1}{2}} \left( \log \f{1}{t} \right)^{\f{3}{2}}.\end{equation}  Then the solution to the ordinary differential equation 
		$$\frac{d}{dt} \log \left(\frac{\sigma}{t\sigma'}\right)= \frac{\theta(\lambda t)}{t},~\sigma(0)=0,~\sigma'(0)=1,$$
		where $\lambda >0,$ has the following properties when $0\leq \lambda t\leq 1$:
		\begin{enumerate}
			\item $t e^{-N} \leq \sigma(t) \leq t,$
			\item $e^{-N} \leq \sigma'(t)\leq 1,$
			\item $|\partial_t[\sigma \log \frac{\sigma}{\sigma' t}]|+|\partial_t[\sigma \log \frac{\sigma}{\sigma' }]|\leq 3N$,
			\item $\left|\sigma \partial_t \left(\frac{1}{\sigma'}\partial_t[\log \frac{\sigma}{\sigma'(t)t}]\right)\right| \leq 3N e^{N} \frac{\theta(\gamma t)}{t},$
		\end{enumerate}	
		where $N$ is some universal constant.
	\end{lemma}
	
	We also need the following  real analysis lemma from \cite{EF_2003}.  See lemma 3.3 in \cite{EF_2003}.
	\begin{lemma}\label{rl}
		Given $m>0$, $\exists C_m$ such that for all $y \ge 0$ and $0 < \epsilon < 1$, 
		\begin{equation}\label{ein}
			y^m e^{-y} \le C_m \left [ \epsilon + \left(\log(\frac{1}{\epsilon}) \right )^m e^{-y} \right].
		\end{equation}
	\end{lemma}
	Finally, we also need the following Hardy type inequality in the Gaussian space which  can be found in Lemma 2.2 in \cite{ABDG}.  This can be regarded as the weighted analogue of  Lemma 3 in \cite{EFV_2006}.
	\begin{lemma}[Hardy type inequality]\label{hardy}
		For all $h \in C_0^{\infty}(\overline{\mb{R}^{n+1}_+})$ and $b>0$ the following inequality holds
		\begin{align*}
			& \int_{\mb{R}^{n+1}_+} x_{n+1}^a h^2 \frac{|X|^2}{8b} e^{-|X|^2/4b}dX \leq  2b \int_{\mb{R}^{n+1}_+} x_{n+1}^a|\nabla h|^2 e^{-|X|^2/4b} dX
			\\
			& + \frac{n+1+a}{2} \int_{\mb{R}^{n+1}_+}x_{n+1}^a h^2  e^{-|X|^2/4b}  dX.
		\end{align*}
	\end{lemma}

	We now state and prove our main Carleman estimate which constitutes the generalization of the Carleman estimate  in \cite[Lemma 6]{EFV_2006}  to degenerate operators of the type \eqref{extop}. 
	
	\begin{theorem}\label{carl1}

		Let $ \widetilde{\mc{H}}$ be the   backward in time extension operator  in \eqref{extop} where  $A(x,t)$ is  a matrix valued function satisfying \eqref{matrix} and \eqref{lip1}. Let $ w \in C_0^\iy \left( \overline{\mb{B}^+_4} \times [0,\left. \f{1}{3\ld}\right) \right)$  such that $\py w\equiv 0$ on $\{x_{n+1}=0\}$ where $ \ld = \frac{\alpha}{ \delta^2}$ for some  $ \delta \in (0,1) $. Then the following estimate holds  for all large $\alpha$	 and $\delta$ sufficiently small  %
		\begin{align}\label{car1}
			& \alpha^2 \int_{\mb{R}^{n+1}_+ \times [c, \iy)} x_{n+1}^{a} \sigma^{-2 \alpha}(t) \ w^2 \ G  + \alpha \int_{\mb{R}^{n+1}_+ \times [c, \iy)} x_{n+1}^{a} \sigma^{1-2 \alpha}(t)\  |\nabla w|^2 \ G  \\
			& \preceq  \int_{\mb{R}^{n+1}_+ \times [c, \iy)} \sigma^{ 1-2 \alpha}(t) x_{n+1}^{-a} \ \lvert \widetilde{\mc{H}} w \rvert ^2 \ G  + \alpha^{c' \alpha} \ \tn{sup}_{t \ge c} \int_{\mb{R}^{n+1}_+} x_{n+1}^{a} \left( w^2 + t |\nabla w|^2 \right) \ \d X \notag\\
			& + \sigma^{-2 \alpha}(c) \left\{ - c \int_{t=c} x_{n+1}^a \ |\nabla w(X,c)|^2 \ G(X,c) \ \d X + \alpha \int_{t=c} x_{n+1}^a \ |w(X,c)|^2 \ G(X,c) \ \d X\right\}.\notag
		\end{align}
		Here $\sigma$ is as in Lemma \ref{sigma},  $G(X,t) = \f{1}{{t^\f{n+1+a}{2}}} e^{-\f{|X|^2}{4t}}$ and $0 < c \le \f{1}{5\ld}. $
	\end{theorem}
	\begin{proof}
		Let $\theta$ be as in Lemma \ref{sigma}.   For $t\in[0,\f{1}{3\ld})$, we first make the preliminary observation that
		\begin{align}\label{large}
			\f{\theta(\ld t)}{t} \ge \ld^{1/2} t^{-\f{1}{2}} \left( \log 3 \right)^{\f{3}{2}} \ge \sqrt{3}\ld \left( \log 3 \right)^{\f{3}{2}} \geq \ld .
		\end{align}
		Also, with a slight abuse of notation, we treat the quantity $- \left(\f{t\sigma'}{\sigma}\right)'$ as $\f{\theta(\ld t)}{t}$ since the term $\f{t\sigma'}{\sigma}$ is positively bounded from both sides in view of Lemma \ref{sigma}. The solid integrals below will be taken in $ \mb{R}^n \times [c, \iy) $ where $ 0 < c \le \f{1}{\ld} $ and we refrain from mentioning explicit limits in the rest of our discussion. \\

		Note that we have the following equivalent expression for $\widetilde{\mc{H}}$ 
		\[ x_{n+1}^{-\f{a}{2}} \widetilde{\mc{H}} = x_{n+1}^{\f{a}{2}} \left( \dd_t + \tn{div}(A(x,t)\nabla) + \f{a}{x_{n + 1}} \dd_{n+1}\right).\]
		Consider the conjugation 
		\[ w(X,t) = \sigma^\alpha(t) e^{\f{|X|^2}{8t}} v(X,t). \]
		We then note that
		\begin{align}\label{sto1}
			w_t = e^{\f{|X|^2}{8t}} \left( \sigma^\alpha(t) v_t + \alpha \sigma^{\alpha - 1}(t) \sigma'(t) v - \f{|X|^2}{8t^2}\sigma^\alpha(t) v \right), \ \nabla w = e^{\f{|X|^2}{8t}} \sigma^\alpha(t) \left( \nabla v + \f{X}{4t} v \right).
		\end{align}
		From \eqref{sto1} we find
		\begin{align*}
			& \tn{div} (A(x,t)\nabla w) = \tn{div} \left( \sigma^\alpha(t) e^{\f{|X|^2}{8t}} A(x,t) \left( \nabla v + \f{X}{4t} v \right) \right) \\ 
			& = \sigma^\alpha(t) e^{\f{|X|^2}{8t}} \left[ \tn{div} (A(x,t)\nabla v) + \f{\langle X, A(x,t) \nabla  v \rangle} {2t} + \left( \f{\langle X, A(x,t) X \rangle}{16 t^2} + \f{\tn{div} (A(x,t) \cdot X)} {4t} \right) v \right]
		\end{align*}
		Now we define the vector field 
		\begin{equation}\label{defz} \mc{Z} := 2t \dd_t + X \cdot A(x,t) \nabla \end{equation}
		and combine the preceding observations to deduce
		\begin{align*}
			x_{n+1}^{-\f{a}{2}} \sigma^{-\alpha}(t) e^{-\f{|X|^2}{8t}} \widetilde{\mc{H}}w 
			& = x_{n+1}^{\f{a}{2}}  \left[ \tn{div}\left( A(x,t)\nabla v\right) + \f{1}{2t} \mc{Z}v + \left( \f{\tn{div}(A(x,t) X) + a} {4t} + \f{\alpha \sigma'} {\sigma}\right) v \right.\\
			& \hspace{5cm}\left. + \left( \f{\langle X, A(x,t) X \rangle}{16 t^2} - \f{|X|^2}{8t^2}\right) v  + \f{a}{x_{n+1}} \dd_{n+1} v \right].
		\end{align*}
		Using \eqref{matrix} and \eqref{lip1}, we further note that the following relations hold for $(X,t)$ varying in a compact set containing the origin 
		\begin{align}\label{1}
			&  \tn{div}(A(x,t)X) = n+1 + O(|X|+t),  \qd \qd  \langle X, A(x,t) X \rangle = |X|^2 + |X|^2 O(|X|+t), \end{align}
		\begin{align}\label{2}
			&  \tn{div}(x_{n+1}^a A(x,t)X) = x_{n+1}^a (\tn{div}(A(x,t)X) + a) = x_{n+1}^a \left( n+1+a+O(|X|+t) \right) .
		\end{align}
		Next we consider the expression
		\begin{align}\label{ex1}
			&  \int \sigma^{-2 \alpha}(t) t^{-\mu} x_{n+1}^{-a} e^{-\f{|X|^2}{4t}} \left(\f{t \sigma'}{\sigma}\right)^{-\f{1}{2}}  \lvert \widetilde{\mc{H}} w \rvert ^2 \\ 
			&= \int x_{n+1}^{a} t^{-\mu} \left(\f{t \sigma'}{\sigma}\right)^{-\f{1}{2}} \left[ \tn{div}\left( A(x,t)\nabla v\right) + \f{1}{2t} \mc{Z}v + \f{a}{x_{n+1}} \dd_{n+1} v \right.\notag \\
			&  \left. + \left( \f{\tn{div}(A(x,t) X) + a} {4t} + \f{\alpha \sigma'}{\sigma} \right) v + \left( \f{\langle X, A(x,t) X \rangle}{16 t^2} - \f{|X|^2}{8t^2}\right) v  \right] ^2,\notag
		\end{align}
		where $\mu$ is to be chosen later. Then we estimate the integral \eqref{ex1} from below with an application of the algebraic inequality
		\[ \int P^2 + 2 \int PQ \le \int \left( P + Q \right)^2 \]
		where $P$ and $Q$ are chosen as
		\begin{align*}
			& P = \f{x_{n+1}^{\f{a}{2}} t^{-\f{\mu + 2}{2}}}{ 2} \left(\f{t \sigma'}{\sigma}\right)^{-\f{1}{4}} \mc{Z}v, \\ 
			& Q = x_{n+1}^{\f{a}{2}} t^{-\f{\mu}{2}} \left(\f{t \sigma'}{\sigma}\right)^{-\f{1}{4}} \left[ \left( \f{\tn{div}(A(x,t) X) + a} {4t} + \f{\alpha \sigma'}{\sigma} \right) v + \tn{div}\left( A(x,t)\nabla v\right) + \f{a \dd_{n+1}v}{x_{n+1}} \right. \\
			& \hspace{4cm} \left. + \f{\langle X, A(x,t) X \rangle - 2 |X|^2}{16 t^2} v  \right].
		\end{align*} 
		To establish the Carleman estimate, we calculate all the terms coming from the cross product, i.e. from $\int PQ.$ We write $$ \int PQ: = \sum_{k=1}^4 \mc{I}_k,$$ where
		\begin{align*}
			&  \mc{I}_1 = \int x_{n+1}^{a} t^{-\mu} \left(\f{t \sigma'}{\sigma}\right)^{-\f{1}{2}} \f{1}{2t} \mc{Z}v \left( \f{\tn{div}(A(x,t) X) + a}{4t} + \f{\alpha \sigma'} {\sigma} \right)v, \\
			&  \mc{I}_2 = \int x_{n+1}^a t^{-\mu} \left(\f{t \sigma'}{\sigma}\right)^{-\f{1}{2}} \f{\mc{Z}v} {2t} \ \tn{div} \left( A(x,t)\nabla v \right), \\
			& \mc{I}_3 = \int x_{n+1}^a t^{-\mu} \left(\f{t \sigma'}{\sigma}\right)^{-\f{1}{2}} \f{\mc{Z}v} {2t} \left( \f{\langle X, A(x,t) X \rangle}{16 t^2} - \f{|X|^2}{8t^2}\right) v, \\
			&  \mc{I}_4 = \int x_{n+1}^{a} t^{-\mu} \left(\f{t \sigma'}{\sigma}\right)^{-\f{1}{2}} \f{\mc{Z}v} {2t} \f{a \dd_{n+1}v}{x_{n+1}}.
		\end{align*}
		
		We start with the term $\mc{I}_1$. We have
		\begin{align}
			\nn \mc{I}_1 & = \int x_{n+1}^{a} t^{-\mu} \left(\f{t \sigma'}{\sigma}\right)^{-\f{1}{2}} \f{1}{2t} \mc{Z}v \left( \f{\tn{div}(A(x,t) X) + a}{4t} + \f{\alpha \sigma'} {\sigma} \right)v \\
			\nn & = \int x_{n+1}^{a} t^{-\mu} \left(\f{t \sigma'}{\sigma}\right)^{-\f{1}{2}} \f{\mc{Z}v}{2t} \left( \f{n+1+a + O(|X|+t)}{4t} + \f{\alpha \sigma'} {\sigma}  \right) v \\
			\nn & = \f{n+1+a}{8} \int x_{n+1}^{a} t^{-\mu-2} \left(\f{t \sigma'}{\sigma}\right)^{-\f{1}{2}} \mc{Z}\left( \f{v^2}{2} \right) + \f{\alpha}{2} \int x_{n+1}^{a} t^{-\mu-2} \left(\f{t \sigma'}{\sigma}\right)^{\f{1}{2}} \mc{Z}\left( \f{v^2}{2} \right) \\
			\nn  & \hspace{0.7cm} + \int x_{n+1}^{a} t^{-\mu-2} \left(\f{t \sigma'}{\sigma}\right)^{-\f{1}{2}} \mc{Z}v \ O(|X|+t)v,
		\end{align}
		which after employing the AM-GM inequality  to the term  $\int x_{n+1}^{a} t^{-\mu-2} \left(\f{t \sigma'}{\sigma}\right)^{-\f{1}{2}} \mc{Z}v \ O(|X|+t)v$  can be  bounded from below in the following way
		\begin{align}\label{e0}
			& \mc{I}_1 \geq \f{n+1+a}{8} \int x_{n+1}^{a} t^{-\mu-2} \left(\f{t \sigma'}{\sigma}\right)^{-\f{1}{2}} \mc{Z}\left( \f{v^2}{2} \right)         + \f{\alpha}{2} \int x_{n+1}^{a} t^{-\mu-2} \left(\f{t \sigma'}{\sigma}\right)^{\f{1}{2}} \mc{Z}\left( \f{v^2}{2} \right) 
			\\
			& \hspace{0.7cm} -\ep \int x_{n+1}^{a} t^{-\mu-2} \left(\f{t \sigma'}{\sigma}\right)^{-\f{1}{2}} |\mc{Z}v|^2 + \f{O(1)}{\ep} \int x_{n+1}^{a} t^{-\mu} \left(\f{t \sigma'}{\sigma}\right)^{-\f{1}{2}} \left(\f{|X|^2 v^2}{t^2} + v^2 \right).\notag
		\end{align}
		We remark here that $ \ep>0 $ will be chosen in a way so  that the term 
		\[ -\ep \int x_{n+1}^{a} t^{-\mu-2} \left(\f{t \sigma'}{\sigma}\right)^{-\f{1}{2}} |\mc{Z}v|^2 \]
		gets absorbed in $\f{1}{2} \int P^2$. We would  also like to  mention that the term $ \f{O(1)}{\ep} \int x_{n+1}^{a} t^{-\mu} \left(\f{t \sigma'}{\sigma}\right)^{-\f{1}{2}} \left(\f{|X|^2 v^2}{t^2} + v^2 \right)$ will be eventually estimated favourably  by using the log-inequality stated in Lemma \ref{rl}.	  See \eqref{error 1}-\eqref{error 2} below. 
		Therefore, we first  engage our attention on the terms
		\[ \f{n+1+a}{8} \int x_{n+1}^{a} t^{-\mu-2} \left(\f{t \sigma'}{\sigma}\right)^{-\f{1}{2}} \mc{Z}\left( \f{v^2}{2} \right) \tn{ and } \  \f{\alpha}{2} \int x_{n+1}^{a} t^{-\mu-2} \left(\f{t \sigma'}{\sigma}\right)^{\f{1}{2}} \mc{Z}\left( \f{v^2}{2} \right). \]
		We  choose $\mu$ such that
		\[
		\tn{div}_{{X,t}}(x_{n+1}^a t^{-\mu-2} \mc{Z}(0,0))=0.
		\]
		Note that $\mc{Z}(0,0)= X \cdot \nabla_X + 2t \partial_t$. This implies that 
		\begin{equation}\label{mu}
			\mu=\frac{n-1+a}{2}.\end{equation} 
		With such a choice of $\mu$,  by integrating by  parts and also by using \eqref{1}  we  then observe
		\begin{align}\label{e2}
			&  \f{n+1+a}{8} \int x_{n+1}^{a} t^{-\mu-2} \left(\f{t \sigma'}{\sigma}\right)^{-\f{1}{2}} \mc{Z}\left( \f{v^2}{2} \right) \\
			& = \f{(n+1+a)}{16} \int x_{n+1}^{a} t^{-\mu-1} \left(\f{t \sigma'}{\sigma}\right)^{-\f{3}{2}}\left(\f{t \sigma'}{\sigma}\right)' v^2 + O(1) \int x_{n+1}^{a} t^{-\mu-2} \left(\f{t \sigma'}{\sigma}\right)^{-\f{1}{2}} \left(|X|+t \right) v^2\notag  \\
			&  - \left( \f{n+1+a}{8} \right) \ c^{-\mu-1} \left(\f{c \sigma'(c)}{\sigma(c)}\right)^{-\f{1}{2}} \int_{t=c} x_{n+1}^{a} v^2(X,c) \ \d X.\notag
		\end{align}
		Similarly we find
		\begin{align}\label{e40}
			& \f{\alpha}{2} \int x_{n+1}^{a} t^{-\mu-2} \left(\f{t \sigma'}{\sigma}\right)^{\f{1}{2}} \mc{Z}\left( \f{v^2}{2} \right) \\
			& = -\f{\alpha}{4} \int x_{n+1}^{a} t^{-\mu-1} \left(\f{t \sigma'}{\sigma}\right)^{-\f{1}{2}} \left(\f{t \sigma'}{\sigma}\right)' v^2  - \f{\alpha}{2} c^{-\mu-1} \left(\f{c \sigma'(c)}{\sigma(c)}\right)^{-\f{1}{2}} \int_{t=c} x_{n+1}^{a} v^2(X,c) \ \d X\notag \\
			& + O(1) \ \alpha \int x_{n+1}^{a} t^{-\mu-2} \left(\f{t \sigma'}{\sigma}\right)^{\f{1}{2}} \left(|X|+t \right) v^2.\notag
		\end{align}
		From \eqref{e0}, \eqref{e2} and \eqref{e40} it follows  using Lemma \ref{sigma} that the following inequality holds for all $\alpha$ large
		\begin{align}\label{nm1}
			& \mc{I}_1  \succeq \alpha \int x_{n+1}^{a} t^{-\mu-1} \f{\theta(\ld t)} {t} v^2 - \alpha \int x_{n+1}^{a} t^{-\mu-2} |X| v^2 - \alpha c^{-\mu-1}\int_{t=c} x_{n+1}^{a} v^2(X,c) \ \d X \\
			&-\ep \int x_{n+1}^{a} t^{-\mu-2} \left(\f{t \sigma'}{\sigma}\right)^{-\f{1}{2}} |\mc{Z}v|^2\notag \\	     & \succeq \alpha \int x_{n+1}^{a} \sigma^{-2\alpha}(t) \f{\theta(\ld t)}{t} \ w^2  G - \alpha \int x_{n+1}^{a} t^{-\mu-2} \sigma^{-2\alpha}(t) |X| \ w^2 e^{-\f{|X|^2}{4t}} - \alpha \sigma^{-2\alpha}(c) \int_{t=c} x_{n+1}^{a} w^2  G \notag
			\\
			&-\ep \int x_{n+1}^{a} t^{-\mu-2} \left(\f{t \sigma'}{\sigma}\right)^{-\f{1}{2}} |\mc{Z}v|^2, \notag	 \end{align}
		where all  terms with sub-critical power in $t$ can be absorbed in $\alpha \int x_{n+1}^{a} t^{-\mu-1} \f{\theta(\ld t)} {t} v^2$  by using the largeness of $\f{\theta(\ld t)}{t}$ as observed  in \eqref{large} above.
		
		Next we consider the term $\mc{I}_2$ which  finally contributes the positive  gradient terms in our Carleman estimate.  This is accomplished by a Rellich type argument. By integrating by parts and also by using $\py v=0$, we find 	
		\begin{align}\label{k0}
			& \mc{I}_2  = \int x_{n+1}^a t^{-\mu} \left(\f{t \sigma'}{\sigma}\right)^{-\f{1}{2}} \f{\mc{Z}v} {2t} \ \tn{div} \left( A(x,t)\nabla v \right) = \int x_{n+1}^a t^{-\mu} \left(\f{t \sigma'}{\sigma}\right)^{-\f{1}{2}} \sum_{i,j = 1}^{n+1} \left( a_{ij}(x,t) v_j\right)_i \dd_t v\\
			&   + \f{1}{2} \int x_{n+1}^a t^{-\mu-1} \left(\f{t \sigma'}{\sigma}\right)^{-\f{1}{2}} \langle X, A(x,t) \nabla  v \rangle\  \tn{div}(A(x,t)\nabla v)\notag \\
			& = - \int x_{n+1}^a t^{-\mu} \left(\f{t \sigma'}{\sigma}\right)^{-\f{1}{2}} \sum_{i,j = 1}^{n+1} a_{ij} v_j \dd_t v_i - a \int x_{n+1}^{a-1} t^{-\mu} \left(\f{t \sigma'}{\sigma}\right)^{-\f{1}{2}} v_{n+1}\dd_t v \notag\\
			&    + \f{1}{2} \int x_{n+1}^a t^{-\mu-1} \left(\f{t \sigma'}{\sigma}\right)^{-\f{1}{2}} \sum_{i,j,p,q = 1}^{n+1} X_i a_{ij} v_j \left( a_{pq} v_q \right)_p \notag\\
			& = - \f{1}{2} \int x_{n+1}^a t^{-\mu} \left(\f{t \sigma'}{\sigma}\right)^{-\f{1}{2}} \dd_t \langle \nabla v, A\nabla v\rangle + \f{1}{2} \int x_{n+1}^a t^{-\mu} \left(\f{t \sigma'}{\sigma}\right)^{-\f{1}{2}} \sum_{i,j = 1}^{n+1} v_j  v_i \dd_ta_{ij}\notag \\
			&  - \int x_{n+1}^{a} t^{-\mu} \left(\f{t \sigma'}{\sigma}\right)^{-\f{1}{2}} \f{\mc{Z}v} {2t} \f{a \dd_{n+1}v}{x_{n+1}} - \f{1}{2} \int x_{n+1}^a t^{-\mu-1} \left(\f{t \sigma'}{\sigma}\right)^{-\f{1}{2}} \sum_{i,j,p,q} (X_i a_{ij} v_j)_p a_{pq} v_q \notag\\
			& = - \f{1}{4} \int x_{n+1}^a t^{-\mu} \left(\f{t \sigma'}{\sigma}\right)^{-\f{3}{2}} \left( \f{t \sigma'}{\sigma} \right)' \langle \nabla v, A\nabla v\rangle - \f{\mu}{2} \int x_{n+1}^a t^{-\mu-1} \left(\f{t \sigma'}{\sigma}\right)^{-\f{1}{2}} \langle \nabla v, A\nabla v\rangle\notag  \\
			&   + \f{1}{2} c^{-\mu} \left(\f{c \sigma'(c)}{\sigma(c)}\right)^{-\f{1}{2}} \int_{t=c} x_{n+1}^{a} \langle \nabla v, A \nabla v \rangle (X,c) \ \d X  + O(1) \int x_{n+1}^a t^{-\mu} \left|\nabla v\right|^2 - \mc{I}_4 + \mc{K}, \notag
		\end{align}
		where  
		\begin{align}
			\nn  \mc{K}  = - \f{1}{2} \int x_{n+1}^a t^{-\mu-1} \left(\f{t \sigma'}{\sigma}\right)^{-\f{1}{2}} \sum_{i,j,p,q} (X_i a_{ij} v_j)_p \  a_{pq} v_q. \end{align}	
		Using \eqref{1} and \eqref{2} we then obtain
		\begin{align}
			&   \mc{K}  = - \f{1}{2} \int x_{n+1}^a t^{-\mu-1} \left(\f{t \sigma'}{\sigma}\right)^{-\f{1}{2}} \sum_{i,j,p,q} (X_i a_{ij} v_j)_p \  a_{pq} v_q\\
			& = - \f{1}{2} \int x_{n+1}^a t^{-\mu-1} \left(\f{t \sigma'}{\sigma}\right)^{-\f{1}{2}} \sum_{i,j,p,q} \left( X_i a_{ij} v_{jp} + \delta_{ip} a_{ij} v_j + X_i a_{ij,p} v_j \right) \  a_{pq} v_q \notag\\
			& = - \f{1}{2} \int x_{n+1}^a t^{-\mu-1} \left(\f{t \sigma'}{\sigma}\right)^{-\f{1}{2}} \left( \left| A \nabla v \right|^2 + \sum_{i,p,q} X_i v_{ip} a_{pq} v_q + \sum_{i,j,p,q} X_i b_{ij} v_{jp} a_{pq} v_q  + \sum_{i,j,p,q} X_i a_{ij,p} v_j a_{pq} v_q \right) \notag\\
			& = - \f{1}{2} \int x_{n+1}^a t^{-\mu-1} \left(\f{t \sigma'}{\sigma}\right)^{-\f{1}{2}} \bigg( \left| A \nabla v \right|^2 + \f{1}{2} X \cdot \nabla \langle A \nabla v, \nabla v \rangle 
			+ \f{1}{2} \langle X, B(X,t) \nabla \langle A \nabla v, \nabla v \rangle \rangle \notag \\ & + O(|X|) \langle A \nabla v, \nabla v \rangle   \bigg) .\notag  \end{align}
		
		Now by integrating by parts the integral $  - \f{1}{4} \int x_{n+1}^a t^{-\mu-1}  X \cdot \nabla \langle A \nabla v, \nabla v \rangle$, we obtain from above that the following holds,
		\begin{align}\label{k1}
			& \mc{K} = \f{n-1+a}{4} \int x_{n+1}^a t^{-\mu-1} \left(\f{t \sigma'}{\sigma}\right)^{-\f{1}{2}} \langle A \nabla v, \nabla v \rangle + O(1)  \int x_{n+1}^a t^{-\mu-1} ( |X| + t ) \left|\nabla v\right|^2.
		\end{align}
		
		where we also used that $ a_{ij}(X,t) = \delta_{ij} + b_{ij}(X,t) $. Since $\mu=\frac{n-1+a}{2}$, from \eqref{k0} and \eqref{k1} we obtain
		\begin{align}\label{24}
			& \mc{I}_2 + \mc{I}_4 = - \f{1}{4} \int x_{n+1}^a t^{-\mu} \left(\f{t \sigma'}{\sigma}\right)^{-\f{3}{2}} \left( \f{t \sigma'}{\sigma} \right)' \langle \nabla v, A \nabla v \rangle   + O(1) \int x_{n+1}^a t^{-\mu-1} \left( |X| + t \right) \left|\nabla v\right|^2  \\  
			&  + \f{1}{2} c^{-\mu} \left(\f{c \sigma'(c)}{\sigma(c)}\right)^{-\f{1}{2}} \int_{t=c} x_{n+1}^{a} \langle \nabla v, A \nabla v \rangle (X,c) \ \d X.\notag
		\end{align}
		To express the above relation in terms of $ u $, we first recall 
		\begin{equation}\label{uv} \nabla v = \sigma^{-\alpha}(t) e^{-\f{|X|^2}{8t}} \left( \nabla w - \f{X}{4t} w \right). \end{equation}
		We now consider the term  $- \f{1}{4} \int x_{n+1}^a t^{-\mu} \left(\f{t \sigma'}{\sigma}\right)^{-\f{3}{2}} \left( \f{t \sigma'}{\sigma} \right)' \langle \nabla v, A \nabla v \rangle$. 	Using \eqref{uv}  and also \eqref{1} and \eqref{2} we have
		\begin{align}\label{i1}
			& - \f{1}{4} \int x_{n+1}^a t^{-\mu} \left(\f{t \sigma'}{\sigma}\right)^{-\f{3}{2}} \left( \f{t \sigma'}{\sigma} \right)' \langle \nabla v, A \nabla v \rangle \\
			& = - \f{1}{4} \int x_{n+1}^a t^{-\mu} \left(\f{t \sigma'}{\sigma}\right)^{-\f{3}{2}} \left( \f{t \sigma'}{\sigma} \right)' \sigma^{-2\alpha}(t) \left\langle \nabla w - \f{X}{4t} w, A \left( \nabla w - \f{X}{4t} w \right) \right\rangle e^{-\f{|X|^2}{4t}}\notag \\
			& =  - \f{1}{4} \int x_{n+1}^a t^{-\mu} \left(\f{t \sigma'}{\sigma}\right)^{-\f{3}{2}} \left( \f{t \sigma'}{\sigma} \right)' \sigma^{-2\alpha}(t) \left( \langle \nabla w, A \nabla w \rangle + \f{\langle X, AX \rangle}{16 t^2} w^2 - \f{1}{4t} \ \langle AX \cdot \nabla (w^2) \rangle \right)  \ e^{-\f{|X|^2}{4t}}\notag  \\
			& = - \f{1}{4} \int x_{n+1}^a t^{-\mu} \left(\f{t \sigma'}{\sigma}\right)^{-\f{3}{2}} \left( \f{t \sigma'}{\sigma} \right)' \sigma^{-2\alpha}(t) \left( \langle \nabla w, A \nabla w \rangle - \f{\langle X, AX \rangle}{16t^2} w^2 \right) \ e^{-\f{|X|^2}{4t}}\notag \\
			& \hspace{0.7cm} - \f{1}{16} \int  t^{-\mu-1} \left(\f{t \sigma'}{\sigma}\right)^{-\f{3}{2}} \left(\f{t \sigma'}{\sigma} \right)' \tn{div}\left( x_{n+1}^a AX \right) w^2  \ e^{-\f{|X|^2}{4t}}\notag \\
			& = - \f{1}{4} \int x_{n+1}^a t^{-\mu} \left(\f{t \sigma'}{\sigma}\right)^{-\f{3}{2}} \left( \f{t \sigma'}{\sigma} \right)' \sigma^{-2\alpha}(t) \left( \langle \nabla w, A \nabla w \rangle - \f{\langle X, AX \rangle}{16t^2} w^2 \right) \ e^{-\f{|X|^2}{4t}} \notag\\
			&  - \f{n+1+a}{16} \int x_{n+1}^a t^{-\mu-1} \left(\f{t \sigma'}{\sigma}\right)^{-\f{3}{2}} \left(\f{t \sigma'}{\sigma} \right)' \sigma^{-2\alpha}(t) w^2 e^{-\f{|X|^2}{4t}}\notag \\ 
			\label{changed gradient} & \hspace{0.7cm}  + O(1) \int x_{n+1}^a t^{-\mu-1} \left(\f{t \sigma'}{\sigma} \right)' \sigma^{-2\alpha}(t) \left( |X| + t \right) w^2  e^{-\f{|X|^2}{4t}}.\notag
		\end{align}
		A purely negative term in \eqref{i1} above is 
		\begin{equation}\label{hi2} \mc{I}^{*}_2 = \f{1}{64} \int x_{n+1}^a t^{-\mu-2} \left(\f{t \sigma'}{\sigma}\right)^{-\f{3}{2}} \left( \f{t \sigma'}{\sigma} \right)' \sigma^{-2\alpha}(t) \langle X, AX \rangle \ w^2 \ e^{-\f{|X|^2}{4t}}.  \end{equation}
		This  will be handled eventually after combining $\mc{I}_2$ with $\mc{I}_3$ due to the presence of a similar term in $\mc{I}_3$. See \eqref{I_3} and \eqref{jusi2} below.
		The boundary integral in \eqref{k0}  above,  i.e.  the term \[ \f{1}{2} c^{-\mu} \left(\f{c \sigma'(c)}{\sigma(c)}\right)^{-\f{1}{2}} \int_{t=c} x_{n+1}^{a} \langle \nabla v, A \nabla v \rangle (X,c) \] can be treated in a similar fashion which finally results in
		\begin{align*}
			& \f{1}{2} c^{-\mu} \left(\f{c \sigma'(c)}{\sigma(c)}\right)^{-\f{1}{2}} \int_{t=c} x_{n+1}^{a} \langle \nabla v, A \nabla v \rangle (X,c) \ \d X \\
			& = \f{1}{2} c^{-\mu} \sigma^{-2\alpha}(c) \left(\f{c \sigma'(c)}{\sigma(c)}\right)^{-\f{1}{2}} \int_{t=c} x_{n+1}^{a} \left( \langle \nabla w, A \nabla w \rangle - \f{\langle X, A(X) \rangle}{16c^2} w^2 + \f{n+1+a}{4c} w^2 \right) \ e^{-\f{|X|^2}{4c}} \ \d X \\
			& \hspace{0.7cm} + O(1) \ c^{-\mu-1} \sigma^{-2\alpha}(c) \int_{t=c} x_{n+1}^{a} \left( |X| + c \right) w^2  \ e^{-\f{|X|^2}{4c}}.
		\end{align*}
		Now a purely negative term in the above expression is 
		\begin{equation}\label{hi22} 
			\mc{I}^{**}_2 = {-} \f{1}{32} c^{-\mu-2} \sigma^{-2\alpha}(c) \left(\f{c \sigma'(c)}{\sigma(c)}\right)^{-\f{1}{2}} \int_{t=c} x_{n+1}^{a} {\langle X, A(X,c)X \rangle} \ w^2(X,c) \ e^{-\f{|X|^2}{4c}} \d X
		\end{equation}
		which will be taken care of by a similar term in \eqref{I_3}. See also \eqref{jusi22} below.  Using \eqref{uv} and also by making use of the inequality $ \left( a + b \right)^2 \le 2 (a^2 + b^2) $, we obtain
		\begin{align}\label{e4}
			\int x_{n+1}^a t^{-\mu-1} |X| \left|\nabla v\right|^2 & \le 2 \int x_{n+1}^a t^{-\mu-1} \sigma^{-2\alpha}(t) |X| |\nabla w|^2 e^{-\f{|X|^2}{4t}} \\
			& \hspace{3cm} + \frac{1}{8} \int x_{n+1}^a t^{-\mu-3} \sigma^{-2\alpha}(t) |X|^3 |w|^2 e^{-\f{|X|^2}{4t}} \notag
		\end{align}
		and 
		\begin{align}\label{e5}
			\int x_{n+1}^a t^{-\mu-1} t \left|\nabla v\right|^2 & \le 2 \int x_{n+1}^a t^{-\mu} \sigma^{-2\alpha}(t) |\nabla w|^2 e^{-\f{|X|^2}{4t}} \\
			& \hspace{3cm} + \frac{1}{8} \int x_{n+1}^a t^{-\mu-2} \sigma^{-2\alpha}(t) |X|^2 |w|^2 e^{-\f{|X|^2}{4t}}.\notag
		\end{align}
		%
		%
		%
		%
		%
		%
		%
		Thus  from \eqref{24}-\eqref{e5} and also by using Lemma \ref{sigma} we deduce the following estimate
		\begin{align}\label{nm2}
			& \mc{I}_2 + \mc{I}_4 - \mc{I}^{*}_2 - \mc{I}^{**}_2  \succeq  \int x_{n+1}^{a} \sigma^{1-2\alpha}(t) \frac{\theta(\lambda t)}{t} \left| \nabla w \right|^2 G  + c \sigma^{-2\alpha}(c) \int_{t=c} x_{n+1}^{a} \left| \nabla w \right|^2 G \\ 
			& - \int x_{n+1}^a t^{-\mu-1} \sigma^{-2\alpha}(t) \left( |X| |\nabla w|^2 + \f{|X|^2}{t} w^2 + \f{|X|^3}{t^2}w^2 \right) e^{-\f{|X|^2}{4t}} - \ c^{-\mu-1} \sigma^{-2\alpha}(c) \int_{t=c} x_{n+1}^{a} |X| w^2 e^{-\f{|X|^2}{4c}}\notag \\
			&  - O(1) \int x_{n+1}^a t^{-\mu-2} \sigma^{-2\alpha}(t) \left( |X| + t \right) w^2  e^{-\f{|X|^2}{4t}}.  \notag
		\end{align} The only cross-product term from $ \int P Q $ which remains to be addressed is $\mc{I}_3$. We have
		\begin{align}\label{I_3}
			\mc{I}_3  & = \int x_{n+1}^a t^{-\mu} \left(\f{t \sigma'}{\sigma}\right)^{-\f{1}{2}} \f{\mc{Z}v} {2t} \left( \f{\langle X, A X \rangle}{16 t^2} - \f{|X|^2}{8t^2}\right) v \\
			& = \f{1}{32} \int x_{n+1}^a t^{-\mu-2} \left(\f{t \sigma'}{\sigma}\right)^{-\f{1}{2}} \left( \langle X, A X \rangle - 2|X|^2 \right)\ \dd_t(v^2)\notag \\
			&  + \f{1}{64} \int x_{n+1}^a t^{-\mu-3} \left(\f{t \sigma'}{\sigma}\right)^{-\f{1}{2}} \left( \langle X, A X \rangle - 2|X|^2 \right)\ \langle X, A \nabla (v^2) \rangle\notag \\
			& = \f{n+3+a}{64} \int x_{n+1}^a t^{-\mu-3} \left(\f{t \sigma'}{\sigma}\right)^{-\f{1}{2}} \left( \langle X, A X \rangle - 2|X|^2 \right)\ v^2\ \text{(using $\mu=\frac{n-1+a}{2}$)} \notag \\
			& + \f{1}{64} \int x_{n+1}^a t^{-\mu-2} \left(\f{t \sigma'}{\sigma}\right)^{-\f{3}{2}} \left(\f{t \sigma'}{\sigma}\right)' \left( \langle X, A(x,t) X \rangle - 2|X|^2 \right)v^2 + \f{1}{32}  \int x_{n+1}^a t^{-\mu-2} \left(\f{t \sigma'}{\sigma}\right)^{-\f{1}{2}} \langle X, A_t X \rangle v^2 \notag\\
			&  - \f{1}{64} \int t^{-\mu-3} \left(\f{t \sigma'}{\sigma}\right)^{-\f{1}{2}} \left( \langle X, A X \rangle - 2|X|^2 \right) \tn{div}\left( x_{n+1}^a A(x,t)X\right) v^2\notag  \\
			&  - \f{1}{64} \int x_{n+1}^a t^{-\mu-3} \left(\f{t \sigma'}{\sigma}\right)^{-\f{1}{2}} \left\langle X, A\nabla\left( \langle X, A X \rangle - 2|X|^2 \right) \right\rangle v^2\notag \\
			&  - \f{1}{32} c^{-\mu-2} \left( \f{c \sigma'(c)}{\sigma(c)} \right)^{-\f{1}{2}} \int_{t=c} x_{n+1}^a \left( \langle X, A X \rangle - 2|X|^2 \right) v^2 .\notag
		\end{align}	
		Now using   
		\begin{align}\label{ere1}
			&  \tn{div}\left( x_{n+1}^a A(x,t)X\right) = n+1+a+O(|X|+t)\\
			&\text{and also that} \left(\langle X, A X \rangle - 2 |X|^2\right) O(|X|+t) = O(1) \left(|X|^3 + t|X|^2\right),\notag
		\end{align}
		we find 
		\begin{align}\label{just1}
			& \f{n+3+a}{64} \int x_{n+1}^a t^{-\mu-3} \left(\f{t \sigma'}{\sigma}\right)^{-\f{1}{2}} \left( \langle X, A X \rangle - 2|X|^2 \right)\ v^2\\
			& - \f{1}{64} \int t^{-\mu-3} \left(\f{t \sigma'}{\sigma}\right)^{-\f{1}{2}} \left( \langle X, A X \rangle - 2|X|^2 \right) \tn{div}\left( x_{n+1}^a A(x,t)X\right) v^2\notag \\
			&= \f{1}{32} \int x_{n+1}^a t^{-\mu-3} \left(\f{t \sigma'}{\sigma}\right)^{-\f{1}{2}} \left( \langle X, A X \rangle - 2 |X|^2 \right)\notag\\
			& + O(1) \int x_{n+1}^a t^{-\mu-3} |X|^3 v^2+ O(1) \int x_{n+1}^a t^{-\mu-2} |X|^2 v^2. \notag
		\end{align}
		Similarly by using
		\begin{align*}
			\langle X, A(x,t) X \rangle - 2|X|^2 = - \langle X, A(x,t)X \rangle + O(1) (|X|^3+ t|X|^2)
		\end{align*}
		
		we have
		\begin{align}\label{jusi2}
			& \f{1}{64} \int x_{n+1}^a t^{-\mu-2} \left(\f{t \sigma'}{\sigma}\right)^{-\f{3}{2}} \left(\f{t \sigma'}{\sigma}\right)' \left( \langle X, A(x,t) X \rangle - 2|X|^2 \right)v^2= -\mc{I}_2^*\\
			& + O(1) \int x_{n+1}^a t^{-\mu-2} \left(\f{t \sigma'}{\sigma}\right)' |X|^3 v^2 + O(1) \int x_{n+1}^a t^{-\mu-1} \left(\f{t \sigma'}{\sigma}\right)' v^2,\notag
		\end{align}
		and
		\begin{align}\label{jusi22}
			&  - \f{1}{32} c^{-\mu-2} \left( \f{c \sigma'(c)}{\sigma(c)} \right)^{-\f{1}{2}} \int_{t=c} x_{n+1}^a \left( \langle X, A X \rangle - 2|X|^2 \right) v^2\\
			&= -\mc{I}_2^{**} + O(1) \ c^{-\mu-2} \int_{t=c} x_{n+1}^a |X|^3 v^2 + O(1) \ c^{-\mu-1} \int_{t=c} x_{n+1}^a |X|^2 v^2,   \notag
		\end{align}
		where $\mc{I}_2^*$ and $\mc{I}_2^{**}$ are as in \eqref{hi2} and \eqref{hi22} respectively.
		Thus using \eqref{just1}-\eqref{jusi22} in \eqref{I_3} combined with the fact that
		$  \langle X, A_t X \rangle = O(|X|^2)$,
		we obtain 
		\begin{align}\label{io4}
			\mc{I}_3  & = \f{1}{32} \int x_{n+1}^a t^{-\mu-3} \left(\f{t \sigma'}{\sigma}\right)^{-\f{1}{2}} \left( \langle X, A X \rangle - 2 |X|^2 - \f{1}{2} \langle X, A\nabla (\langle X, A X \rangle - 2|X|^2) \rangle \right) v^2 \\
			&+ O(1) \int x_{n+1}^a t^{-\mu-2} |X|^2 v^2  + O(1) \int x_{n+1}^a t^{-\mu-2} \left(\f{t \sigma'}{\sigma}\right)' |X|^3 v^2 + O(1) \int x_{n+1}^a t^{-\mu-1} \left(\f{t \sigma'}{\sigma}\right)' v^2  \notag \\
			& - \mc{I}^*_2 - \mc{I}^{**}_2 + O(1) \int x_{n+1}^a t^{-\mu-3} |X|^3 v^2 + O(1) \ c^{-\mu-2} \int_{t=c} x_{n+1}^a |X|^3 v^2 + O(1) \ c^{-\mu-1} \int_{t=c} x_{n+1}^a |X|^2 v^2.\notag
		\end{align}
		Now using  the largeness of $\f{\theta(\ld t)}{t}$ and also the fact that 
		\begin{align*}
			\langle X, A X \rangle - 2|X|^2 = -|X|^2 + O(1) \left(|X|^3 + t|X|^2\right), \\
			\langle X, A\nabla (\langle X, A X \rangle - 2|X|^2) \rangle = - 2|X|^2 + O(1) \left(|X|^3 + t|X|^2\right),
		\end{align*} 
		we find
		\begin{align}\label{tyur}
			&	 \f{1}{32} \int x_{n+1}^a t^{-\mu-3} \left(\f{t \sigma'}{\sigma}\right)^{-\f{1}{2}} \left( \langle X, A X \rangle - 2 |X|^2 - \f{1}{2} \langle X, A\nabla (\langle X, A X \rangle - 2|X|^2) \rangle \right) v^2\\
			&=  O(1) \int x_{n+1}^a t^{-\mu-2}\sigma^{-2\alpha} |X|^2  w^2e^{-\f{|X|^2}{4t}}  
			+ O(1) \int x_{n+1}^a t^{-\mu-3}\sigma^{-2\alpha} |X|^3 w^2 e^{-\f{|X|^2}{4t}}.\notag
		\end{align}
		
		Using \eqref{tyur} in \eqref{io4} we thus  deduce  the following estimate

		\begin{align}\label{pen}
			&   \mc{I}_3 + \mc{I}^*_2 + \mc{I}^{**}_2   =    O(1) \int x_{n+1}^a \sigma^{-2\alpha} \frac{\theta(\lambda t)}{t} w^2 G  + O(1) \int x_{n+1}^a t^{-\mu-2}\sigma^{-2\alpha} |X|^2  w^2e^{-\f{|X|^2}{4t}}  \\
			&  + O(1) \int x_{n+1}^a t^{-\mu-3}\sigma^{-2\alpha} |X|^3 w^2 e^{-\f{|X|^2}{4t}}  + O(1) \sigma^{-2\alpha}(c) \ c^{-\mu-2} \int_{t=c} x_{n+1}^a |X|^3 w^2 e^{-\f{|X|^2}{4c}}      \notag \\
			&  + O(1) \sigma^{-2\alpha}(c) \int_{t=c} x_{n+1}^{a} w^2  G.\notag
		\end{align}
		We now estimate all the error terms in \eqref{nm1}, \eqref{nm2} and \eqref{pen} above using the log inequality in Lemma \ref{rl}.

		Taking $\ep = (\ld t)^{2\alpha+\mu+\f{3}{2}}$ and $m = \f{1}{2}$ in \eqref{ein}, we observe
		\begin{align*}
			t^{-\mu -2} |X| e^{-\f{|X|^2}{4t}} = 2t^{-\mu-\f{3}{2}} \left( \f{|X|^2}{4t} \right)^{1/2} e^{-\f{|X|^2}{4t}} & \le C t^{-\mu-\f{3}{2}} \left( (\ld t)^{2\alpha+\mu+\f{3}{2}} + \left(2\alpha+\mu+\f{3}{2}\right)^\f{1}{2} \left(\log \f{1}{\ld t} \right)^\f{1}{2} e^{-\f{|X|^2}{4t}} \right) \\
			& \le C \left( \ld^{2\alpha + N} t^{2\alpha} + \rt \alpha t^{-\mu-\f{3}{2}} \left(\log \f{1}{\ld t} \right)^\f{1}{2} e^{-\f{|X|^2}{4t}}            \right) \\
			& \le C \left( \ld^{2\alpha + N} t^{2\alpha} + \delta \left( \ld t \log \f{1}{\ld t} \right)^\f{1}{2} t^{-\mu-2} e^{-\f{|X|^2}{4t}}  \right)\\
			&\le C \left( \ld^{2\alpha + N} t^{2\alpha} + \delta \ \f{\theta(\ld t)}{t} t^{-\mu-1} e^{-\f{|X|^2}{4t}} \right).	\end{align*} 
		Since $\sigma \sim t$, it follows from the above  inequality that the following estimate holds
		\begin{align}\label{error 1}
			\alpha   \int x_{n+1}^a t^{-\mu-2} \sigma^{-2\alpha} |X| e^{-\f{|X|^2}{4t}} w^2 \le C \left( \ld^{2\alpha+ N} \int x_{n+1}^a  u^2 + \delta \alpha \int x_{n+1}^a \f{\theta(\ld t)}{t} \sigma^{-2\alpha}  w^2 G\right).
		\end{align}
		Again by applying \eqref{ein} with $m = \f{3}{2}$ and $\ep = \left(\ld t\right)^{2\alpha + \mu + \f{3}{2}}$  we obtain
		\begin{align}\label{ty}
			&  t^{-\mu -3} |X|^3 e^{-\f{|X|^2}{4t}} = 8 t^{-\mu - \f{3}{2}} \left( \f{|X|^2}{4t} \right)^{\f{3}{2}} e^{-\f{|X|^2}{4t}} \le C t^{-\mu - \f{3}{2}} \left( (\ld t)^{2\alpha + \mu + \f{3}{2}} + \left(\alpha+\mu+\f{3}{2}\right)^\f{3}{2} \left(\log \f{1}{\ld t} \right)^\f{3}{2} e^{-\f{|X|^2}{4t}} \right) \\
			& \le C \left( \ld^{2\alpha + N} t^{2\alpha} + \delta \alpha \ \f{\theta(\ld t)}{t} t^{-\mu-1} e^{-\f{|X|^2}{4t}} \right) \notag
		\end{align}
		and thus similarly as for \eqref{error 1}, using \eqref{ty} we deduce the following inequality
		\begin{align}\label{error 2}
			\int x_{n+1}^a t^{-\mu-3} \sigma^{-2\alpha} |X|^3  e^{-\f{|X|^2}{4t}} w^2 \le C \left( \ld^{2\alpha+ N} \int x_{n+1}^a  w^2 + \delta \alpha \int x_{n+1}^a \f{\theta(\ld t)}{t} \sigma^{-2\alpha}  w^2 G\right).
		\end{align}
		Also in an essentially similar way, we get
		\begin{align}\label{error 3}
			\int x_{n+1}^a t^{-\mu-1} \sigma^{-2\alpha}(t) e^{-\f{|X|^2}{4t}} |X| \left|\nabla w \right|^2 \le C \left( \ld^{2\alpha+N} \int x_{n+1}^a t |\nabla w|^2 + \delta \int x_{n+1}^a \f{\theta(\ld t)}{t} \sigma^{1-2\alpha}  |\nabla w|^2  G\right).
		\end{align}
		
		Now we note that the other error terms such as   $O(1) \int x_{n+1}^a \sigma^{-2\alpha} \frac{\theta(\lambda t)}{t} w^2 G$ that shows up in \eqref{nm2} and \eqref{pen},   $\delta \alpha \int x_{n+1}^a \f{\theta(\ld t)}{t} \sigma^{-2\alpha}  w^2 G$ that shows up in   \eqref{error 1} and \eqref{error 2}  can be absorbed in the integral $ \alpha \int x_{n+1}^{a} \sigma^{-2\alpha}\f{\theta(\ld t)} {t} w^2G$ in \eqref{nm1} provided $\delta$ is sufficiently small. 
		
		\medskip

		Likewise for small enough $\delta$,  the term $\delta \int x_{n+1}^a \f{\theta(\ld t)}{t} \sigma^{1-2\alpha}  |\nabla w|^2  G$  in \eqref{error 3} can be absorbed in the integral $\int x_{n+1}^a \f{\theta(\ld t)}{t} \sigma^{1-2\alpha}  |\nabla w|^2  G$ which appears in \eqref{nm2}. 	
		
		We finally control the error term  $\sigma^{-2 \alpha}(c) \int_{t=c} x_{n+1}^a \f{|X|^3}{c} \ |w|^2 \ G$ that appears in \eqref{pen} in the following way.
		We have
		\begin{align}\label{ha1}
			& \sigma^{-2 \alpha}(c) \int_{t=c} x_{n+1}^a \f{|X|^3}{c} \ |w|^2 \ G \\
			& = \sigma^{-2 \alpha}(c) \int_{\mb{B}_\rho, \ t=c} x_{n+1}^a \f{|X|^3}{c} \ |w|^2 \ G + \sigma^{-2 \alpha}(c) \int_{\mb{B}^c_\rho, \ t=c} x_{n+1}^a \f{|X|^3}{c} \ |w|^2 \ G\notag \\
			& \le \rho \sigma^{-2 \alpha}(c) \int_{t=c} x_{n+1}^a \f{|X|^2}{c} \ |w|^2 \ G + N^\alpha \ld^{2\alpha+N} \int_{t=c} x_{n+1}^a w^2(X,c),\notag
		\end{align} 
		where we used the estimate $ \f{|X|^3}{c} G(X,c) \sigma(c)^{-2\alpha} \le N^\alpha \ld^{2\alpha+N} $ for $ x \in \mb{B}^c_\rho$ in the last line in the above inequality. Finally the term  $\rho \sigma^{-2 \alpha}(c) \int_{t=c} x_{n+1}^a \f{|X|^2}{c} \ |w|^2 \ G$ is estimated by using the Hardy inequality in Lemma \ref{hardy} as  follows
		
		\begin{equation}\label{ha2}
			\rho \sigma^{-2 \alpha}(c) \int_{t=c} x_{n+1}^a \f{|X|^2}{c} \ |w|^2 \ G	\leq C \rho \sigma^{-2\alpha}(c)  \left( \int x_{n+1}^a 	 w^2 G 	 + c \int x_{n+1}^a |\nabla w|^2 G\right).
		\end{equation}
		Now the term $C \rho \sigma^{-2\alpha}(c) c \int x_{n+1}^a |\nabla w|^2 G$ can be absorbed in the integral $c \sigma^{-2\alpha}(c) \int_{t=c} x_{n+1}^{a} \left| \nabla w \right|^2 G$  in \eqref{nm2} provided $\rho$ is small enough. 	Thus from \eqref{ex1}, \eqref{nm1}, \eqref{nm2}, \eqref{pen} and \eqref{error 1}-\eqref{ha2}, we finally deduce the following estimate
		
		\begin{align}\label{ce}
			& \alpha \int_{\mb{R}^{n+1}_+ \times [c, \iy)} x_{n+1}^{a} \frac{\theta(\lambda t)}{t} \sigma^{-2 \alpha}(t) \ w^2 \ G  +  \int_{\mb{R}^{n+1}_+ \times [c, \iy)} x_{n+1}^{a} \sigma^{1-2 \alpha}(t)\ \frac{\theta(\lambda t)}{t}  |\nabla w|^2 \ G  \\
			& \preceq  \int_{\mb{R}^{n+1}_+ \times [c, \iy)}  \int \sigma^{-2 \alpha}(t) t^{-\mu} x_{n+1}^{-a} e^{-\f{|X|^2}{4t}} \left(\f{t \sigma'}{\sigma}\right)^{-\f{1}{2}}  \lvert \widetilde{\mc{H}} w \rvert ^2  + \alpha^{c' \alpha} \ \tn{sup}_{t \ge c} \int_{\mb{R}^{n+1}_+} x_{n+1}^{a} \left( w^2 + t |\nabla w|^2 \right) \ \d X \notag\\
			& + \sigma^{-2 \alpha}(c) \left\{ - c \int_{t=c} x_{n+1}^a \ |\nabla w(X,c)|^2 \ G(X,c) \ \d X + \alpha \int_{t=c} x_{n+1}^a \ |w(X,c)|^2 \ G(X,c) \ \d X\right\}.\notag
		\end{align}	
		The desired estimate as claimed in \eqref{car1} follows  from \eqref{ce} above  by using \eqref{large}, the fact that $\lambda \sim \alpha$ and also that
		\begin{align*}
			&  \int_{\mb{R}^{n+1}_+ \times [c, \iy)}  \int \sigma^{-2 \alpha}(t) t^{-\mu} x_{n+1}^{-a} e^{-\f{|X|^2}{4t}} \left(\f{t \sigma'}{\sigma}\right)^{-\f{1}{2}}  \lvert \widetilde{\mc{H}} w \rvert ^2 
			\\
			&\approx \int_{\mb{R}^{n+1}_+ \times [c, \iy)} \sigma^{ 1-2 \alpha}(t) x_{n+1}^{-a} \ \lvert \widetilde{\mc{H}} w \rvert ^2 \ G.	  \end{align*}  
		
	\end{proof}
	
	Now by a translation in time,  we find from Lemma  \ref{carl1} that the following estimate holds.
	
	\begin{lemma}\label{carl}
		Let $ \widetilde{\mc{H}}$ be as in \eqref{extop} where  $A(x,t)$  satisfies \eqref{matrix}. Let $ w \in C_0^\iy \left( \overline{\mb{B}^+_4} \times [0,\left. \f{1}{3\ld}\right) \right)$  be such that $\py w\equiv 0$ on $\{x_{n+1}=0\}$ where $ \ld = \frac{\alpha}{ \delta^2}$ for  $ \delta \in (0,1) $   sufficiently small. Then the following estimate holds  for all large $\alpha$	  %
		\begin{align}\label{car10}
			& \alpha^2 \int_{\mb{R}^{n+1}_+ \times [0, \iy)} x_{n+1}^{a} \sigma_c^{-2 \alpha}(t) \ w^2 \ G_c  + \alpha \int_{\mb{R}^{n+1}_+ \times [0, \iy)} x_{n+1}^{a} \sigma_c^{1-2 \alpha}(t)\  |\nabla w|^2 \ G_c  \\
			& \preceq  \int_{\mb{R}^{n+1}_+ \times [0, \iy)} \sigma_c^{ 1-2 \alpha}(t) x_{n+1}^{-a} \ \lvert \widetilde{\mc{H}} w \rvert ^2 \ G_c  + \alpha^{c' \alpha} \ \tn{sup}_{t \ge 0} \int_{\mb{R}^{n+1}_+} x_{n+1}^{a} \left( w^2 + t |\nabla w|^2 \right) \ \d X \notag\\
			& + \sigma^{-2 \alpha}(c) \left\{ - c \int_{t=0} x_{n+1}^a \ |\nabla w(X,0)|^2 \ G(X,c) \ \d X + \alpha \int_{t=0} x_{n+1}^a \ |w(X,0)|^2 \ G(X,c) \ \d X\right\}.\notag
		\end{align}
		Here $\sigma_c(t) = \sigma (c+t)$, $G_c(X,t)= G(X, t+c)$  and $0 < c \le \f{1}{5\ld}. $	
	\end{lemma}
	
	\subsection{Some basic regularity estimates for the extension problem} We now gather some important qualitative properties of the solution to the  extension problem \eqref{exten prob}. We first note that it follows from Theorem \ref{energy} that the following result holds. 
	
	\begin{lemma}\label{ext1}
		Let $U$ be the solution to the extension problem \eqref{exten prob} corresponding to $u \in \mb{H}^s$. Assume that $\mc{H}^s u=0$ in $B_1 \times (-1,0)$ in the sense of Defintion \ref{wk}.  Then $U$ is a weak solution to
		\begin{equation}\label{wk1}
			\begin{cases}
				\operatorname{div}(x_{n+1}^a A(x) \nabla U)= x_{n+1}^a U_t\ \text{in $\mb{B}_1^+ \times (-1,0)$},
				\\
				\py U=0\ \text{at $\{x_{n+1}=0\}$}.
			\end{cases}
		\end{equation}
	\end{lemma}
	
	We refer to Section 4 in \cite{BG} for the precise notion of weak solutions. See also \cite{BGMN}.	
	\begin{proof}
		\emph{Step 1:} Let $\tilde U$ denote the even reflection of $U$ across $\{x_{n+1}=0\}$. See \eqref{refl} below. We claim that
		\begin{equation}\label{ref1}
			\int_{\mb{B}_1 \times (-1, 0)} |x_{n+1}|^a \langle \nabla \tilde U , \nabla \phi \rangle dXdt \int_{\mb{B}_1 \times (-1, 0)} |x_{n+1}|^a \tilde U \phi_t dX dt,
		\end{equation}
		for all $\phi \in C^\infty_0(\mb{B}_1^+ \times (-1, 0))$. 
		We first let 
		\[
		I_\ep= \int_{\mb{B}_1 \times (-1, 0) \cap \{|x_{n+1}| > \ep\}} |x_{n+1}|^a \langle \nabla \tilde U , \nabla \phi \rangle dXdt.
		\]
		Then by using the equation satisfied by $\tilde U$ in $\{|x_{n+1}| > \ep \}$ and divergence theorem, we find 
		\begin{equation}\label{eqk}
			I_\ep= \int_{\mb{B}_1 \times (-1, 0) \cap \{|x_{n+1}| > \ep\}} |x_{n+1}|^a \tilde U \phi_t dX dt  + A_\ep +B_\ep,
		\end{equation} 	
		where 
		\begin{equation}
			\begin{cases}
				A_\ep= - \int_{x_{n+1} =\ep}  x_{n+1}^a \partial_{x_{n+1}} \tilde U(x, \ep, t) \phi(x, \ep, t)dXdt\\  B_\ep=  \int_{x_{n+1} =-\ep}  |x_{n+1}|^a \partial_{x_{n+1}} \tilde U(x, -\ep, t)\phi(x, -\ep, t)dXdt.\end{cases}	\end{equation}
		Using (ii) in Theorem \ref{energy}, we now show that $A_\ep, B_\ep \to 0$ as $\ep \to 0$. We only show it for $A_\ep$ as the arguments for $B_\ep$ is analogous.
		Now $A_\ep$ can be rewritten as
		\[
		A_\ep = - \int_{x_{n+1} =\ep}  x_{n+1}^a \partial_{x_{n+1}} \tilde U(x, \ep, t) \phi(x, 0, t)dXdt	- 	\int_{x_{n+1} =\ep}  x_{n+1}^a \partial_{x_{n+1}} \tilde U(x, \ep, t) (\phi(x, \ep, t)- \phi(x, 0, t))dXdt.	\]
		Using ii) in Theorem \ref{energy}, the fact that $\mc{H}^s u=0$ in $B_1 \times (-1, 0)$ and also that $\phi(\cdot, 0,\cdot)$ is smooth and compactly supported, it follows that as $\ep \to 0$
		\[
		\int_{x_{n+1} =\ep}  x_{n+1}^a \partial_{x_{n+1}} \tilde U(x, \ep, t) \phi(x, 0, t)dXdt \to 0.\]
		Now by using Fundamental theorem of calculus in $x_{n+1}$, we can write
		\[
		(\phi(x, \ep, t)- \phi(x, 0, t)= \ep \psi(x, \ep, t).
		\]
		where $\psi$ is smooth and compactly supported. Thus  using inequality \eqref{cc}, Plancherel theorem and Cauchy-Schwartz inequality, we find that the term 	$\int_{x_{n+1} =\ep}  x_{n+1}^a \partial_{x_{n+1}} \tilde U(x, \ep, t) (\phi(x, \ep, t)- \phi(x, 0, t))dXdt$  can be estimated in the following way
		\begin{align}
			&	\left|\int_{x_{n+1} =\ep}  x_{n+1}^a \partial_{x_{n+1}} \tilde U(x, \ep, t) (\phi(x, \ep, t)- \phi(x, 0, t))dXdt   \right|
			\\
			& \leq  C\ep ||u||_{\mb{H}^s(\mb R^{n+1})} \times \text{sup}_{0 < b<1}||\psi(\cdot, b, \cdot)||_{\mb{H}^s(\mb R^{n+1})} \to 0,
		\end{align}
		as $\ep \to 0$. Thus we find also by using iii) in Theorem \ref{energy} that  $A_\ep$ and likewise $B_\ep \to 0$ as $\ep \to 0$ which establishes the claim in \eqref{ref1}.
		
		\medskip
		
		\emph{ Step 2(Conclusion):} Now given that \eqref{ref1} holds, by a density argument as in \cite[Corollary 1.7]{Ki}  and also by using iii) in Theorem \ref{energy},  it is seen that \eqref{ref1}	  holds for all $\phi$ such that $\nabla \phi, \phi_t \in L^{2}(\mb{B}_1, x_{n+1}^a dX dt)$. 
		
	\end{proof}

	We now state the relevant regularity result for the extension problem which is Lemma 2.2 in \cite{BG1}. We refer to \cite[Chapter 4]{Li} for the relevant notion of parabolic H\"older spaces $H^{k+\alpha}$ that appears below.
	\begin{lemma}\label{reg}
		Let $U$ be a weak solution to \eqref{wk1} in $\mb{B}_1^+ \times (-1, 0]$ where $A$ satisfies \eqref{ellip} and \eqref{assum} or equivalently \eqref{matrix}. Then the extended function $\tilde U$  which is defined as
		\begin{equation}\label{refl}
			\begin{cases}
				\tilde U(x, x_{n+1})= U(x, x_{n+1})\ \text{for $x_{n+1}>0$}
				\\
				\tilde U(x, x_{n+1})= U(x, -x_{n+1})\ \text{for $x_{n+1}<0$}
			\end{cases}
		\end{equation}
		solves
		\begin{equation}\label{po1}
			\operatorname{div}(|x_{n+1}|^a A(x) \nabla \tilde U) - |x_{n+1}|^a \partial_t \tilde U=0
		\end{equation}
		in $\mathbb B_1 \times (-1, 0]$, and moreover $\tilde U \in H^{1+\alpha}(\mathbb B_{1/2} \times (-1/4, 0])$  for all $\alpha>0$.  Moreover, the $H^{1+\alpha}$ norm in $\mb{B}_{1/2} \times (-1/4, 0])$  can be estimated by
		$C \int_{\mathbb{B}_1 \times (-1, 0]} |x_{n+1}|^a \tilde U^2 dX dt$ where $C$ depends on the dimension, the ellipticity  and the Lipschitz character of $A$.
	\end{lemma}	
	
	Moreover by arguing as in the proof of Lemma 5.5 in \cite{BG}, we have the following result regarding the integrability of the second derivatives.
	
	\begin{lemma}\label{reg2}
		Let $U$ be as in Lemma \ref{reg} above. Then we have that the following estimate holds,
		\begin{equation}\label{po2}
			\int_{\mb{B}_{1/2}^+ \times (-1/4,0])} x_{n+1}^a ( |\nabla U|^2 + |\nabla_{x} \nabla U|^2  + U_t^2) + x_{n+1}^{-a} |\nabla (x_{n+1}^a U_{x_{n+1}})|^2 \leq C \int_{\mathbb{B}_1 \times (-1, 0]} |x_{n+1}|^a \tilde U^2 dX dt,\end{equation}
		where $C$ has a similar dependence as in Lemma \ref{reg} above.
	\end{lemma}
	
	As previously said, for notational purposes it will be convenient to work with the following backward version of problem \eqref{wk} in the cylinder  $ \mathbb B^+_4 \times (0,16]$
	\begin{equation}\label{exprob}
		\begin{cases}
			x_{n+1}^a \partial_t U + \operatorname{div}(x_{n+1}^a A(x)\nabla U)=0\ \ \ \ \ \ \ \ \ \ \ \ \text{in} \ \mb{B}_4^+ \times [0, 16),
			\\	
			U(x,0, t)= u(x,t)
			\\
			\py U(x, 0,t)= 0\ \ \ \ \ \ \ \ \ \text{in}\ B_4 \times [0,16).
		\end{cases}
	\end{equation}
	We note that the former can be transformed into the latter by changing $t \to -t$.
	
	We now introduce an assumption that will remain in force for the rest of the section up to the proof of Theorem \ref{GUCP}. When we work with a solution $U$ of the problem \eqref{wk1} in $\mb{B}_{4}^+ \times (-16, 0]$, we will always assume that
	\begin{equation}\label{ass}  
		\int_{\mathbb B_1^+} x_{n+1}^a U(X,0)^2 dX >0.
	\end{equation}
	
	As a consequence of such hypothesis the number 
	\begin{equation}\label{theta}
		\theta \overset{def}{=}\frac{\int_{\mb{B}_4^+ \times (-16, 0]} x_{n+1}^a U(X,t)^2 dXdt }{\int_{\mathbb B_1^+} x_{n+1}^aU(X,0)^2 dX}
	\end{equation}
	will be well-defined. In the remainder of this work the symbol $\theta$ will always mean the number defined by \eqref{theta}. 	
	
	We now state and prove the relevant monotonicity in time result  which is analogous to Lemma 3.1 in \cite{ABDG}.
	\begin{lemma}\label{mont} 
		Let $U$ be a solution of \eqref{exprob}. Then there exists a constant $N = N(n,a,A)>2$ such that $N\operatorname{log}(N\theta) \geq 1$, and for which the following inequality holds for $0\leq t \leq 1/{N\operatorname{log}(N\theta)}$
		\begin{align*}
			N\int_{\mathbb B_2^+} x_{n+1}^aU(X,t)^2 dX \geq \int_{\mathbb B_1^+} x_{n+1}^a U(X,0)^2 dX.
		\end{align*}	
	\end{lemma}
	\begin{proof}
		Let  $f= \phi U,$ where $\phi \in C_0^{\infty}(\mathbb B_2)$ is a spherically symmetric cutoff such that $0\le \phi\le 1$ and $\phi \equiv1$ on $\mathbb B_{3/2}.$ Since $U$ solves \eqref{exprob} and $\phi$ is independent of $t$ and symmetric in $x_{n+1}$, it is easily seen that the function $f$ solves the problem
		\begin{equation}\label{feq}
			\begin{cases}
				x_{n+1}^a f_t + \operatorname{div}(x_{n+1}^a \nabla f) = 2 x_{n+1}^a \langle\nabla U,\nabla \phi\rangle  + \operatorname{div}(x_{n+1}^a \nabla \phi) U\ \ \ \ \ \ \ \text{in} \ \mathbb B_4^+ \times (-16, 0],
				\\	
				f(x,0,t)= u(x,t)\phi(x,0)
				\\
				\py f(x,0, t)= 0\ \ \ \ \ \ \ \ \ \ \ \ \ \ \ \ \ \ \ \ \ \ \ \ \ \ \ \ \ \ \ \ \ \text{in}\ B_4 \times [0,16).
			\end{cases}
		\end{equation}
		Again since $\phi$ is symmetric in $x_{n+1}$, we have $\partial_{n+1} \phi \equiv 0$ on the thin set $\{x_{n+1}=0\}$. This fact and the smoothness of $\phi$ imply that $\frac{\phi_y}{y}$ be bounded up to $\{y=0\}$. Therefore we observe that the following is true 
		\begin{equation}\label{obs1}
			\begin{cases}
				\operatorname{supp} (\nabla \phi) \cap \{x_{n+1}>0\}  \subset \mathbb B_2^+ \setminus \mathbb B_{3/2}^+
				\\
				|\operatorname{div}(x_{n+1}^a \nabla \phi)| \leq C x_{n+1}^a\ \mathbf 1_{\mathbb B_2^+ \setminus \mathbb B_{3/2}^+},
			\end{cases}
		\end{equation}
		where for a set $E$ we have denoted by $\mathbf 1_E$ its indicator function.
		
		We now fix a point $Y \in \mb{R}^{n+1}_+$ and introduce the quantity
		\begin{align*}
			H(t) = \int_{\mb{R}^{n+1}_+} x_{n+1}^a f(X,t)^2 \mc{G}(Y,X,t) dX,
		\end{align*}
		where $\mc{G}$ is as in \eqref{fund}. We note that for $t>0$, $\mc{G} = \mc{G}(Y, \cdot)$ solves
		\begin{equation}\label{eq}
			\operatorname{div}(x_{n+1}^a \nabla \mc{G}) = x_{n+1}^a \partial_t \mc{G}.
		\end{equation}
		Before proceeding further, we remark that in the ensuing computations below,  the formal differentiation under the integral sign and the integration by parts    can be  justified  by an approximation argument by first considering the integrals in the region $\{x_{n+1} >\ep\}$ and then by letting $\ep \to 0$ using the regularity estimates in Lemma \ref{reg} and Lemma \ref{reg2}.
		Thus in view of this, By differentiating $H'$, we observe using \eqref{eq} that the following holds
		\begin{align}\label{hprime}
			H'(t) & = 2 \int_{\mb{R}^{n+1}_+} x_{n+1}^a f f_t \mc{G} + \int x_{n+1}^a f^2 \dd_t\mc{G} \\
			& = 2 \int_{\mb{R}^{n+1}_+} x_{n+1}^a f f_t \mc{G} + \int_{\mb{R}^{n+1}} f^2 \tn{div}\left(x_{n+1}^a A(x,t)\nabla \mc{G} \right)\notag \\
			& = 2 \int f \mc{G} \left( x_{n+1}^a f_t + \tn{div} \left(x_{n+1}^a A(x,t) \cdot \nabla f \right) \right) + 2 \int x_{n+1}^a  \mc{G} \langle \nabla f , A(x,t) \nabla f \rangle.\notag
		\end{align}
		For $Y \in \mathbb B_1^+$, we now claim that the following estimate holds
		\begin{equation}\label{ai1}
			I_1 :=  2 \int f \mc{G} \left( x_{n+1}^a f_t + \tn{div} \left(x_{n+1}^a A(x,t) \cdot \nabla f \right) \right) \geq - N e^{-1/Nt} \int_{\mathbb B_4^+ \times (-16, 0]} x_{n+1}^a U^2 dXdt,
		\end{equation}   
		for some universal $N$.
		We argue as in \cite{ABDG}. In order to establish \eqref{ai1}, we need the following asymptotics of $I_{\frac{a-1}{2}}$ which asserts that there exists  $C(a), c(a) >0$  such that 
		\begin{equation}\label{bessel}
			I_{\frac{a-1}{2}}(z) \leq C(a) z^{\frac{a-1}{2}}  \hspace{6mm} \text{if} \hspace{2mm} 0 < z \le c(a),\ \ \ \ \ I_{\frac{a-1}{2}}(z) \leq C(a) z^{-1/2} e^z \hspace{2mm}\ \  \text{if} \hspace{2mm} z \ge c(a).
		\end{equation}
		See for instance \cite[formulas (5.7.1) and (5.11.8)]{Le}.  We then write the integral on the left hand side  in \eqref{ai1} as   $I_1^1 + I_1^2,$ where $I_1^1$ is integral on the set  $ \mathcal A=\{ X\in \mb{R}^{n+1}_+ \mid x_{n+1}y_{n+1} > 2 t c(a)\}$ and  $I_1^2$ is the integral on the complement $\mathcal A_e$ of $\mathcal A$. We want to bound $I_1$ by appropriately bounding $\mc{G}$ from above in each of the sets $\mathcal A$ and $\mathcal A_e$. In this respect it is important to note that  in view of \eqref{feq} and \eqref{obs1}, the integral in the definition of $I_1$ is actually performed in $X \in \mathbb B^+_2 \setminus  \mathbb B^+_{3/2}$ and on such set we have for every $Y \in \mathbb B^+_1$
		\begin{equation}\label{X}
			\frac{1}{2} \le |X-Y|\le 3.
		\end{equation}
		Our objective is to prove that when $Y \in \mathbb B^+_1$, $X \in \mathbb B^+_2 \setminus  \mathbb B^+_{3/2}$ and $0<t\leq 1$, the following bound holds for some universal $M>0$ 
		\begin{equation}\label{g4}
			\mc{G}(Y, X, t) \leq e^{-\frac{1}{M t}}.
		\end{equation}	   
		To prove that \eqref{g4}  holds when $X \in \mathcal A\cap(\mathbb B^+_2 \setminus  \mathbb B^+_{3/2})$ we argue as follows. Since for $X\in \mathcal A$ we have $\frac{x_{n+1}y_{n+1}}{2t} >  c(a)$, by the second inequality in \eqref{bessel} we have
		\begin{equation}\label{gud}
			I_{\frac{a-1}{2}}\left(\frac{x_{n+1} y_{n+1}}{2t}\right) \le C(a) \left(\frac{x_{n+1} y_{n+1}}{2t}\right)^{-1/2} e^{\frac{x_{n+1} y_{n+1}}{2t}}.
		\end{equation}
		Consider first the case $-1<a \leq 0$. Since for $X\in \mathbb B^+_2$ and $Y\in \mathbb B^+_1$ we trivially have $\frac{x_{n+1}y_{n+1}}{2t} \leq \frac{4}t$, in such case we have $\left(\frac{x_{n+1}y_{n+1}}{2t}\right)^{-a/2} \leq  2^{-a/2}t^{a/2}$. Using this estimate and \eqref{gud} in \eqref{pa},  we obtain
		\begin{align*}
			p_a(y_{n+1},x_{n+1},t) 
			\leq C^\star(a) t^{-1/2} e^{-\frac{( y_{n+1} - x_{n+1})^2}{4t}}.
		\end{align*}
		Combining this bound with \eqref{gbd} we infer that for $Y \in \mathbb B^+_1$ and $X \in \mathcal  A$ 
		\begin{equation}\label{g20}
			\mc{G}(Y, X, t) \leq  C\ t^{-\frac{n+1}{2}} e^{- \frac{|x-y|^2}{N_0t} -\frac{( y_{n+1} - x_{n+1})^2}{4t} }.
		\end{equation}
		
		On the other hand if $a>0$, then we have $\left(\frac{x_{n+1} y_{n+1}}{2t}\right)^{-a/2} \leq	c(a)^{-a/2}$ for $X \in \mathcal A$. Using this estimate and \eqref{gud} in \eqref{pa} we find
		\begin{align*}
			p_a(y_{n+1},x_{n+1},t)  \le C^{\star\star}(a) t^{-\frac{a+1}2} e^{-\frac{(y_{n+1} - x_{n+1})^2}{4t}}.
		\end{align*}
		Combining this bound with \eqref{gbd} we infer that for $Y \in \mathbb B^+_1$ and $X\in  \mathcal A$ 
		\begin{equation}\label{g2}
			\mc{G}(Y, X, t) \leq  C\ t^{-\frac{n+1+a}{2}} e^{- \frac{|x-y|^2}{N_0t} -\frac{( y_{n+1} - x_{n+1})^2}{4t} }.
		\end{equation}
		From \eqref{g20} and \eqref{g2} and \eqref{X} we conclude that when $Y \in \mathbb B^+_1$, $X \in  \mathcal A\cap(\mathbb B^+_2 \setminus  \mathbb B^-_{3/2})$ and $0<t\le 1$, the following bound holds for some universal $C>0$ and for $l= \max\{\frac{n+1}{2},\frac{n+1+a}{2}\}$
		\[
		\mc{G}(Y, X, t) \leq  C\ t^{-l} e^{-\frac{1}{Ct}}.
		\] 
		From this inequality above, \eqref{g4} immediately follows when $X\in \mathcal A\cap(\mathbb B^+_2 \setminus  \mathbb B^+_{3/2})$. If instead $X\in \mathcal A_e \cap(\mathbb B^+_2 \setminus  \mathbb B^+_{3/2})$, keeping in mind that on the set  $\mathcal A_e$ we have $\frac{x_{n+1}y_{n+1}}{2t} \le c(a)$, by the first inequality in \eqref{bessel} we obtain that for all $a \in (-1, 1)$
		\[
		I_{\frac{a-1}{2}}\left(\frac{x_{n+1}y_{n+1}}{2t}\right) \leq C(a) \left(\frac{x_{n+1} y_{n+1}}{2t}\right)^{\frac{a-1}{2}}. 
		\]
		Using this in \eqref{pa} we find
		\[
		p_a(y_{n+1},x_{n+1},t) \le C(a) (2t)^{-\frac{a+1}{2}} e^{-\frac{y_{n+1}^2+x_{n+1}^2}{4t}} \le C^\star(a) t^{-\frac{a+1}{2}} e^{-\frac{(y_{n+1}-x_{n+1})^2}{8t}} .
		\]
		Combining this bound with \eqref{gbd} we again conclude that for $Y\in \mathbb B^+_1$, $0<t\le 1$ and $X\in \mathcal A_e   \cap(\mathbb B^+_2 \setminus  \mathbb B^+_{3/2})$\begin{equation*}
			\mc{G}(Y,X,t) \leq Ct^{-\frac{n+1+a}{2}} e^{-\frac{1}{Ct}}.
		\end{equation*}
		Thus we find that \eqref{g4} holds. Now using \eqref{g4} in the definition of $I_1$ and also by using \eqref{feq} and \eqref{obs1} we finally obtain 
		\[
		|I_1| \le C e^{-\frac{1}{Mt}} \int_{\mathbb B^+_2} x_{n+1}^a \left(|\nabla U| + |U|\right) |U|.
		\]
		We can now appeal to  the  $L^{\infty}$ bounds for $U, \nabla U, U_t$  as in Lemma \ref{reg} to finally conclude that for every $Y\in \mathbb B^+_1$ and $0<t\le 1$ the inequality \eqref{ai1} holds.
		
		Using \eqref{ai1} in \eqref{hprime}, we obtain
		\begin{equation}\label{h1}
			H'(t) \geq  - N e^{-1/Nt} \int_{\mathbb B_4^+ \times (-16, 0]} x_{n+1}^a U^2 dXdt.\end{equation}	   
		
		Now from the approximation to identity property  \eqref{st}    it follows that
		\begin{equation}\label{st1}
			\operatorname{lim}_{t \to 0^+} H(t)= U(Y, 0)^2.
		\end{equation}
		Using \eqref{st1} in \eqref{h1} we obtain
		\begin{equation}\label{st2}
			H(t) \geq U(Y, 0)^2 - N e^{-1/Nt} \int_{\mathbb B_4^+ \times (-16, 0]} x_{n+1}^a U^2 dXdt.
		\end{equation}
		Now by integrating \eqref{st2} with respect to $Y$ in $\mathbb B_1^+	$, exchanging the order of integration and using \eqref{stoc1} we obtain
		\begin{equation}\label{st4}
			\int_{\mathbb B_2^+} x_{n+1}^aU(X, t)^2  dX  \geq \int_{\mathbb B_1^+} x_{n+1}^a U(X, 0)^2 dX - N e^{-1/Nt} \int_{\mathbb B_4^+ \times (-16, 0]} x_{n+1}^a U^2 dXdt.
		\end{equation}
		Note that in \eqref{st4} above, we have renamed the variable $Y$ as $X$. Now from the $L^{\infty}$ bound on $U$ as in Lemma \ref{reg} that the following estimate holds
		\begin{equation}\label{st5}
			\int_{\mathbb B_1^+} x_{n+1}^a U(X, 0)^2 dX \leq C  \int_{\mathbb B_4^+ \times (-16, 0]} x_{n+1}^a U^2 dXdt.
		\end{equation}
		Note that \eqref{st5} in particular implies that $\theta$ as defined in \eqref{theta} is bounded from below away from zero.  Now if we let 
		\begin{equation}\label{t005}
			t \leq \frac{1}{N \operatorname{log}(2N \theta)},\end{equation}
		then we find  from the definition of $\theta$ in \eqref{theta} that
		\begin{equation}\label{s2}
			N e^{-1/Nt} \int_{\mathbb B_4^+ \times (-16, 0]} x_{n+1}^a U^2 dXdt < \frac{1}{2} \int_{\mathbb B_1^+} x_{n+1}^a U(X, 0)^2 dX.\end{equation}
		Using \eqref{s2} in \eqref{st4} we have
		\begin{equation}\label{s4}
			2 \int_{\mathbb B_2^+} x_{n+1}^aU(X, t)^2  dX  \geq \int_{\mathbb B_1^+} x_{n+1}^a U(X, 0)^2 dX,
		\end{equation}
		for all $t$ satisfying \eqref{t005}.
		Thus by letting $2N$ as our new $N$, we find that the conclusion of the lemma follows.
	\end{proof}
	Now given  the Carleman estimate in Lemma \ref{carl} and the monotonicity result  in Lemma \ref{mont}, using Lemma \ref{reg} and the integrability of the second derivatives as in Lemma \ref{reg2},  one can now repeat the arguments  as  in \cite[pages 11- 13]{EFV_2006} ( see also \cite{ABDG}) to  assert that the following conditional doubling inequality holds under the assumption \eqref{ass}.

	\begin{theorem}\label{db1}
		Let $U$ be a solution of \eqref{exprob} in $\mathbb B_4^+ \times [0, 16).$ There exists $N>2$, depending on $n$, $a$  for which $N\log(N\theta) \ge 1$ and such that:
		\begin{itemize}
			\item[(i)] For $r \leq 1/2,$ we have 
			$$\int_{\mathbb B_{2r}^+}x_{n+1}^aU(X,0)^2 dX \leq (N \theta)^N\int_{\mathbb B_{r}}x_{n+1}^aU(X,0)^2 dX.$$
			Moreover for  $r \leq 1/\sqrt{N \operatorname{log}(N \theta)}$  the following two inequalities hold:
			\item[(ii)] $$\int_{\mathbb B_{2r}^+ \times [0, 4r^2)}x_{n+1}^aU(X,t)^2 dXdt \leq \operatorname{exp}(N \operatorname{log}(N \theta) \operatorname{log}(N \operatorname{log}(N \theta)))r^2 \int_{\mathbb B_{r}^+}x_{n+1}^aU^2(X,0)dX.$$
			\item[(iii)]$$\int_{\mathbb B_{2r}^+ \times [0, 4r^2)} x_{n+1}^aU(X,t)^2 dXdt \leq \operatorname{exp}(N \operatorname{log}(N \theta) \operatorname{log}(N \operatorname{log}(N \theta)))\int_{\mathbb B_r^+ \times [0, r^2)}x_{n+1}^aU(X,t)^2dXdt.$$ 
		\end{itemize}
	\end{theorem} 	
	\subsection{Proof of Theorem \ref{GUCP}}
	With Theorem \ref{db1} in hand, we now proceed with the proof of Theorem \ref{GUCP} by means of blowup argument  inspired by   that in \cite{ABDG} and \cite{BG}.
	
	\begin{proof}[Proof of Theorem \ref{GUCP}]
		Without loss of generality, we assume that $A(0)=\mathbb I$ and also that  $u \in \mb{H}^s$ solves $\mc{H}^s u=0$ in $B_4 \times (-16, 0]$ and vanishes  in $B_4 \times (-16, 0]$. It suffices to show that    for the solution $U$ to the extension problem \eqref{exprob} ( by changing $t \to -t$), we must have 
		\begin{equation}\label{crucial}
			U(X,0) \equiv 0, \ \ \ \ \ \ \ \text{for every}\ X\in \mathbb B_1^+.
		\end{equation}
		
		Once \eqref{crucial} is proven, 
		we then note that, away from the thin set $\{x_{n+1}=0\}$, $U$ solves a uniformly parabolic PDE with Lipschitz coefficients and vanishes identically in the half-ball $\mathbb B_1^+$.  We can thus appeal to \cite[Theorem 1]{AV} to assert that $U$ vanishes to infinite order both in space and time  at every $(X,0)$ for $X \in \mathbb B_1^+$. At this point, we can  use the strong unique continuation  result  in   \cite[Theorem 1]{EF_2003}  to finally conclude that $U(X,0) \equiv 0$ for $X\in \mathbb R^{n+1}_+$. Letting $x_{n+1}=0$, this implies $u(x,0) = U(x,0,0) \equiv 0$ for $x\in \mathbb R^n$.  Similarly, we can show that $u(\cdot, t) \equiv 0$ for all $t \in (-16, 0)$ and thus  Theorem \ref{GUCP} would follow.

		Therefore we are left with  establishing  the claim  in \eqref{crucial}. We argue by contradiction and assume that \eqref{crucial} is not true. Consequently, \eqref{ass} does hold and therefore we can use the results in Theorem \ref{db1}.  In particular from  (i) in Theorem \ref{db1} it follows that $\int_{\mathbb B_r^+}x_{n+1}^a U(X, 0)^2 dX > 0$ for all $0<r \leq \frac{1}{2}$. From this fact and the continuity of $U$ up to the thin set $\{x_{n+1}=0\}$ we deduce that
		\begin{equation}\label{nzero}
			\int_{\mathbb B_r^+ \times [0, r^2)} x_{n+1}^aU^2  dX dt>0, \end{equation} for all $0< r \leq 1/2$. Moreover, the  inequality (iii) in Theorem \ref{db1} holds, i.e. there exist $r_0$ and $C$ depending on $\theta$ in \eqref{theta} such that for all $r\leq r_0$ one has
		\begin{equation}\label{jb}
			\int_{\mathbb B_r^+ \times [0, r^2)} x_{n+1}^aU^2 dX dt \leq C \int_{\mathbb B_{r/2}^+ \times [0, r^2/4)} x_{n+1}^aU^2  dXdt.
		\end{equation}
		From this doubling estimate we can derive in a standard  manner  the following inequality for all $r \leq \frac{r_0}{2}$
		\begin{equation*}
			\int_{\mathbb B_r^+ \times [0, r^2)} x_{n+1}^aU^2  dX dt  \geq \frac{r^{M}}{C}   \int_{\mathbb B_{r_0}^+ \times [0, r_0^2)} x_{n+1}^aU^2 dX dt,
		\end{equation*}
		where $M= \operatorname{log}_2 C$. Letting $C_0=\frac{1}{C}\int_{\mathbb B_{r_0}^+ \times [0, r_0^2)} U^2 y^a dX dt$, and noting that $C_0>0$ in view of \eqref{nzero}, we can rewrite the latter inequality as
		\begin{equation}\label{fvan1}
			\int_{\mathbb B_r^+ \times [0, r^2)} U^2 y^a dX dt  \geq C_0 r^{M}.
		\end{equation}
		
		Let now $r_j\searrow 0$ be a sequence such that $r_j \leq r_0$ for every $j\in \mathbb N$, and define
		\[
		U_j(X,t) = \frac{U(r_jX, r_j^2 t)}{\bigg(\frac{1}{r_j^{n+3+a}}\int_{\mathbb B_{r_j}^+ \times [0, r_j^2)} x_{n+1}^aU^2  dX dt\bigg)^{1/2}}. 
		\]
		Note that thanks to \eqref{nzero} the functions $U_j$'s are well defined. Furthermore,   by a change of variable, we note 
		\begin{equation}\label{bound}
			\int_{\mathbb B_1^+ \times [0,1)} x_{n+1}^aU_j^2  dX dt=1.\end{equation}
		Again by a change of variable  and by using  the doubling  inequality \eqref{jb}, we have for all $j$
		\begin{equation}\label{nondeg}
			\int_{\mathbb B_{1/2}^+ \times [0, 1/4)} x_{n+1}^aU_j^2 dX dt \geq C^{-1}.
		\end{equation}
		Moreover $U_j$ solves the following problem in $\mathbb B_1^+ \times [0, 1)$
		\begin{equation}\label{expb1}
			\begin{cases}
				\operatorname{div}( x_{n+1}^a  A(r_j x)\nabla U_j) + x_{n+1}^a \partial_t U_j=0,
				\\
				\py U_j (x,0,t) = 0.
			\end{cases}
		\end{equation}
		From \eqref{bound} and the regularity estimates in Lemma \ref{reg}  and Lemma \ref{reg2}  we infer that, possibly passing to a subsequence which we continue to indicate with $U_j$, we have $U_j \to U_0$ in $H^{1+\alpha}(\mathbb B_{3/4}^+ \times [0, 9/16))$ up to $\{x_{n+1}=0\}$. We infer in a standard way by a weak type argument that  the blowup limit $U_0$ solves in $\mathbb B_{3/4}^+ \times [0, 9/16)$ 
		\begin{equation}\label{expb2}
			\begin{cases}
				\operatorname{div}( y^a \nabla U_0) + y^a \partial_t U_0=0,
				\\
				\py U_0 (x,0, t) = 0.
			\end{cases}
		\end{equation}
		Since $U(x, 0, t)$ vanishes identically in $B_4 \times [0, 16)$, it follows on account of uniform convergence of $U_j$'s to $U_0$  that $U_0(x, 0, t) \equiv 0$ in $ B_{1/2} \times [0, 1/4).$
		On the other hand, from the uniform convergence  of $U_j$'s in $\mathbb B_{1/2}^+ \times [0, 1/4)$ and  the non-degeneracy estimate \eqref{nondeg} we also have
		\begin{equation}\label{nondeg1}
			\int_{\mathbb B_{1/2}^+ \times [0, 1/4)} x_{n+1}^aU_0^2  dX dt \geq C^{-1},
		\end{equation}
		and thus $U_0\not\equiv 0$ in $ \mathbb B_{1/2}^+ \times [0,1/4)$.  This violates the weak unique continuation property in Theorem \ref{wkt}.  Therefore \eqref{crucial} must be true.
		Now  in view of  our discussion after \eqref{crucial}, we find that the conclusion of the theorem thus follows.
		
	\end{proof}

 	
 	\section{Applications to Calder\'on inverse problems} \label{Applications in Inverse Problems}
    In this section, we obtain the  unique recovery result of the potential $q$ as in the initial-exterior problem \eqref{ini-ext prob} from the nonlocal DN map \eqref{DN map}. We rigorously define the DN map introduced in \eqref{DN map} and then derive an Alessandrini type identity in this context. This will be followed by the Runge approximation result which will be  a byproduct of  unique continuation result Theorem \ref{GUCP}. This is  similar to that in \cite{GSU} and \cite{CLR}. As previously mentioned in the introduction, such a Runge type approximation argument allows to bypass the method of CGO solutions and this aspect is quite specific to nonlocal problems. In this section, we closely follow the approach in \cite{CLR}.
    We start by defining the abstract trace space as follows
    \begin{align*}
        \mb{X} := \mc{H}^s \left(\mb{R}^{n} \times [-T,T]\right) \setminus \ch. \end{align*}
 The norm in $\mb{X}$ is defined in an analogous way as \eqref{subset}.
  Before moving on to the definition of the DN map, we would like to  stress  the fact that the solution $ u $ in \eqref{ini-ext prob} corresponding to $ f \in \mb{H}^s(\mb{R}^{n+1}) $ depends only on $ f \vert_{Q_e}$  where $Q_e= \Omega_e \times (-T, T)$ which can be seen as a consequence of uniqueness and the weak formulation in Definition \ref{wk}. To emphasize the dependence of the solution on the data, we declare $u_f$ and $ u^*_f $ to be the solutions of \eqref{ini-ext prob} and \eqref{adjoint problem for q} respectively for the exterior value $f$.  The following proposition below constitutes the rigorous definition of the DN map analogous to Proposition 3.5 in \cite{CLR}. This crucially relies on the well-posedness which is accounted by \eqref{asus}.
    
    \begin{proposition}[The DN map for $\mc{H}^s + q$]
    Let $ s \in (0,1), \ T>0 $ and $\Omega$ be a bounded open set in $ \mb{R}^n, n \ge 1 $ and let $Q=\Omega \times (-T, T)$. Further assume $ q \in L^\iy (Q) $ satisfies the eigenvalue condition \eqref{asus}. For $f, g \in \mc{H}^s(\mb{R}^{n} \times [-T,T])$, we define
    \begin{align*}
        \langle \Ld_q [f], [g] \rangle_{\mb{X}^* \times \mb{X}} = \mc{B}_q (u_f,g).  
    \end{align*}
    Then $ \Ld_q: \mb{X} \to \mb{X}^* $ is a bounded operator. 
    \end{proposition}
    \begin{proof} 
       First we need to justify well-definedness of $\Ld_q$. For that, we consider $f' \in [f]$ and $g' \in [g]$ or, in other words $ f' = f + \phi $ and $ g' = g + \psi $ for some $ \phi, \psi \in \ch$ and wish to show $\langle \Ld_q [f], [g] \rangle_{\mb{X}^* \times \mb{X}} = \langle \Ld_q [f'], [g'] \rangle_{\mb{X}^* \times \mb{X}}$. In this regard, we notice  
       \begin{align*}
           \mc{B}_q (u_{f+\phi},{g+\psi}) = \mc{B}_q (u_{f+\phi},g) = \mc{B}_q (u_f,g)
       \end{align*}
       where the above implications follow directly from the weak formulation of \eqref{ini-ext prob} and the fact that $u_{f}$ depends only on $f\vert_{Q_e}$. Moreover  from \eqref{continuity} we find
       \begin{align*}
           \lvert \langle \Ld_q [f], [g] \rangle_{\mb{X}^* \times \mb{X}} \rvert \le C \  \|f+\phi\|_{\mc{H}^s(\mb{R}^{n} \times [-T,T])} \ \|g+\psi\|_{\mc{H}^s(\mb{R}^{n} \times [-T,T])}
       \end{align*}
       where the constant $C>0$ is independent of the choices $\phi,\psi \in \ch$. This implies $ \Ld_q [f] \in \mb{X}^*$ with $\|\Ld_q\|_{\mb{X} \to \mb{X}^*} \le C$. 
    \end{proof}

    Following the natural pairing, we define the DN map for the adjoint problem \eqref{adjoint problem for q} as 
    \begin{align}\label{adjoint DN map for q}
       \langle [f], \Ld^*_q [g] \rangle_{\mb{X} \times \mb{X}^*} = \langle \Ld_q [f], [g]\rangle_{\mb{X}^* \times \mb{X}} .
    \end{align}
    We also note that if $u^*_g \in H^s(\mb{R}^{n+1})$ solves \eqref{adjoint problem for q} corresponding to the exterior data $g \in \mc{H}^s(Q_e)$ then we have 
    \begin{align}\label{ch of adjoint DN map for q}
        \langle [f], \Ld^*_q [g] \rangle_{\mb{X} \times \mb{X}^*} = \mc{B}_{q} (f,u^*_g)
    \end{align}
    which follows from the variational formulation of the problems \eqref{ini-ext prob} and \eqref{adjoint problem for q}.

    We now state and prove an Alessandrini type identity in our context which plays an essential role in proving the uniqueness result.

    \begin{lemma}[Integral identity for $\mc{H}^s + q$] \label{int identity for q}
    Let $ s \in (0,1), T>0 $ and $\Omega$ be a bounded  open set in $ \mb{R}^n$. Furthermore, let  $ q_1, q_2 \in L^\iy (Q) $ be such that the eigenvalue condition \eqref{asus} holds. Then for $ f, g \in \mc{H}^s(\mb{R}^{n} \times [-T,T]) $, we have
    \begin{align*}
        \langle \left( \Ld_{q_1} - \Ld_{q_2} \right)[f], [g] \rangle_{\mb{X}^* \times \mb{X}} = \int_{Q} (q_1 - q_2) u_f u^*_g.  
    \end{align*}
    where $u_f$ solves the problem \eqref{ini-ext prob} for $q=q_1$ associated to the exterior data $f$ and $u^*_g$ is a solution to \eqref{adjoint problem for q} when $q=q_2$ corresponding to  exterior data  $g$. 
    \end{lemma}
    \begin{proof}
     From the adjoint property  \eqref{adjoint DN map for q}, its characterization in \eqref{ch of adjoint DN map for q} and  also by using the  fact that $u_f \in [f], u^*_g \in [g]$, we find
     \begin{align*}
         \langle \left( \Ld_{q_1} - \Ld_{q_2} \right)[f], [g] \rangle_{\mb{X}^* \times \mb{X}} & = \langle \Ld_{q_1} [f], [g] \rangle_{\mb{X}^* \times \mb{X}} - \langle  [f], \Ld_q^* [g] \rangle_{\mb{X} \times \mb{X}^*} \\
         & = \mc{B}_{q_1} (u_f, u^*_g) - \mc{B}_{q_2} (u_f, u^*_g) \\
         & = \int_Q (q_1-q_2)  u_f u^*_g.
     \end{align*}
    \end{proof}
    
    The final step in the uniqueness proof is  the following density result where we crucially use the weak unique continuation result Theorem \ref{GUCP}. We remark here that the unique determination result \eqref{determining q} only requires the density with respect to $L^2$ norm. But we  eventually need the approximation result  in $\ch$ for recovering both
 	the first and zeroth order perturbations in Theorem \ref{determining b,q}. With an intent to omit a similar discussion for the $(b,q)$ case, we present a general approximation result (in $\ch$ topology) for the case when $b=0$.
    
    \begin{theorem}[Runge approximation for $\mc{H}^s + q$] \label{Runge for q}
     Let $ s \in (0,1), T>0 $ and $\Omega$ be a bounded  open set in $ \mb{R}^n $. Consider ${W}$ to be a bounded open set in $ \mb{R}^{n} $, such that $ \overline{\Omega} \cap \overline{{W}} = \emptyset $. Then the set 
     \begin{align*}
         \mc{D}_{q} ({W}) = \{ u_f - f ; \qd f \in C_0^{\iy} \left({W} \times (-T,T)\right)\}
     \end{align*}
     is dense in $ \ch $ where $Q=\Omega \times (-T, T)$ and $u_f$ is the solution to \eqref{ini-ext prob} corresponding to $f$.
    \end{theorem}
    \begin{proof}
    The proof  is similar to that in \cite{GSU} given the validity of Theorem \ref{GUCP}.  Invoking  the Hahn-Banach theorem, it suffices to show that  there is no non-trivial $ F\in \chd$ which satisfies 
    \begin{align}\label{runge}
        \langle F, u_f - f \rangle = 0, \qd \fa f \in C_0^{\iy} \left({W} \times (-T,T)\right).
    \end{align}
    In order to establish \eqref{runge}, we first construct $\phi\in \ch$ solving the adjoint problem
    \begin{align*} 
 	\begin{cases}
 	   \left( \mc{H}^s_* + q(x,t) \right) \phi = F, & \tn{ in } Q \\
 	   \phi(x,t) = 0, & \tn{ in } Q_e, \tn{ and for}\ t \ge T.
 	\end{cases}
 	\end{align*}
 	Then the weak formulation \eqref{adjoint problem for q}  together with  \eqref{runge} implies
 	\begin{align}\label{approximation}
 	    0 = \langle F, u_f - f \rangle= \mc{B}_q \left(  u_f - f, \phi \right).
 	\end{align}
	Now using the weak formulation for $u_f$, we find that
	\begin{equation}\label{ao1}
	\mc{B}_q\left( u_f, \phi \right)=0,
	\end{equation}
	since $\phi \equiv 0$ in $Q_e$.
	Therefore it follows that
	\begin{equation}\label{ao2}
	\mc{B}_q \left(f, \phi\right)=0
	\end{equation}
	for all $f \in C_{0}^{\infty}(W \times (-T, T))$.
 	Since $f$ is supported in $W \times (-T, T)$, thus from \eqref{ao2} we deduce that 
 	\begin{align}\label{applying GUCP}
 	    \mc{H}^s_* \phi = 0, \ \phi = 0, \qd \tn{ in } W \times (-T,T).
 	\end{align}
 	Now in view of the change of variable in \eqref{k}, we can  invoke Theorem \ref{GUCP} to conclude that  $\phi = 0$ in $ \mb{R}^{n+1} $ which then implies that $ F=0$. This finishes the proof of the Theorem.
    \end{proof}
    
    With the Runge type approximation result as in Theorem \ref{Runge for q} in hand, we now proceed with the proof of Theorem \ref{determining q}.
    
   \begin{proof}[Proof of Theorem \ref{determining q}] Let us fix some $\phi \in C_0^\iy(Q)$ and also let $W_1$ and $W_2$ be as in Theorem \ref{determining q}. We also let $\psi \in C_0^{\infty}(Q)$ such that $\psi \equiv 1$ on $\text{supp}(\phi)$. By virtue of Theorem \ref{Runge for q},  for $k=1,2$ and $j \in \mb{N}$,  there exists exterior values (for both the forward and adjoint problems) $f_{j,k} \in C^\iy_0(W_k \times (-T,T))$ for which
    \begin{align*}
       & \left( \mc{H}^s + q_1\right)u_{j,1} = \left( \mc{H}_{*}^s + q_2\right)u^{*}_{j,2} = 0, \qd \tn{ in } Q, \\
       & u_{j,1} (x,t) = f_{j,1}(x,t), \tn{ and } u^{*}_{j,2} (x,t) = f_{j,2}(x,t), \tn{ in } Q_e, \\
       & u_{j,1}\vert_{t\leq-T} = 0 , \tn{ and } u^*_{j,2}\vert_{t \geq T} = 0,\\
       &u_{j,1} - f_{j,1} = \phi + r_{j,1},\  u^{*}_{j,2} - f_{j,2} = \psi + r_{j,2},    \end{align*}
  such that \begin{equation}\label{zero} \|r_{j,k}\|_{\ch} \to 0\ \text{as $j \to \iy$ for $k = 1,2$}.\end{equation} Plugging these solutions into the Alessandrini type identity in  Lemma \ref{int identity for q}  and by using \\ $\left. \Lambda_{q_1}([f_{j,1}])\right|_{W_2 \times (-T, T)} = \left. \Lambda_{q_2} ([f_{j,1}]) \right|_{W_2 \times (-T,T)}$ and also that $f_{j,2} \in C_0^{\infty}(W_2 \times (-T, T))$, we obtain 
    \begin{align}\label{applying int-ineq for q}
        \int_Q (q_1-q_2) u_{j,1} u^{*}_{j,2} \ \d{x} \d{t} = 0, \qd \tn{ for } j\in\mb{N}. 
    \end{align}
    Now since $f_{j,k} \equiv 0$ in $Q$, by  letting $j \to \iy$ and by  using \eqref{zero}, we find that   \eqref{applying int-ineq for q} reduces to 
    \begin{align*}
        \int_Q (q_1-q_2) \phi \ \d{x} \d{t} = 0.
    \end{align*}
    Since this is valid for any $\phi \in C^\iy_0(Q)$, we deduce that  $ q_1 = q_2 $ in $Q$.  This finishes the proof of the theorem. \end{proof}

   Now we prove Theorem \ref{determining b,q}. The rigorous definition of the DN map and derivation of related integral identity along with Runge approximation result are exactly similar to the ones discussed for the $q$ case. For this reason, we choose to skip the details and merely mention the statements in this setting. The only part different from the previous discussion is the determination of $b$ and $q$ simultaneously. To accomplish that, we follow the strategy in \cite{CLR} by determining  $q$ first and then use it to recover the drift term $b$.    Throughout we assume that \eqref{kju2} holds.
    
    \begin{proposition}[The DN map for $\mc{H}^s + \langle b, \nabla_{x} \rangle + q$]
    Let $ s \in \left( \f{1}{2}, 1 \right), T>0 $ and $\Omega$ be a bounded Lipschitz open set in $ \mb{R}^n$. Further assume that $ b \in L^\iy \left((-T,T);W^{1-s,\iy}(\Omega)\right), \ q \in L^\iy (Q) $. For $f, g \in \mb{H}^s(\mb{R}^{n} \times [-T,T])$, we define
    \begin{align*}
        \langle \Ld_{b,q} [f], [g] \rangle_{\mb{X}^* \times \mb{X}} = \mc{B}_{b,q} (u_f,g).  
    \end{align*}
    Then $ \Ld_{b,q}: \mb{X} \to \mb{X}^* $ is well defined and is  a bounded operator. 
    \end{proposition}
    \begin{lemma}[Integral identity for $\mc{H}^s + \langle b, \nabla_{x}  \rangle + q$] \label{int identity for b and q}
    Let $ s \in \left(\f{1}{2}, 1\right), T>0 $ and $\Omega$ be a bounded Lipschitz open set in $ \mb{R}^n$. For $ q_1, q_2 \in L^\iy (Q)$,  $ b_1, b_2 \in L^\iy \left((-T,T);W^{1-s,\iy}(\Omega)\right)$  and $ f, g \in \mb H^s(\mb{R}^{n} \times [-T,T]) $, we have
    \begin{align*}
        \langle \left( \Ld_{b_1,q_1} - \Ld_{b_2,q_2} \right)[f], [g] \rangle_{\mb{X}^* \times \mb{X}} = \int_{Q} \langle (b_1 - b_2), \nabla_{x}u_{f} \rangle \ u^*_g + \int_{Q} (q_1 - q_2) u_f u^*_g.  
    \end{align*}
    \end{lemma}
    \begin{theorem}[Runge approximation for $\mc{H}^s + \langle b,\nabla_{x} \rangle + q$]\label{Runge for b and q}
     Let $ s \in \left(\f{1}{2}, 1\right), T>0 $ and $\Omega$ be a bounded Lipschitz open set in $ \mb{R}^n$ and  $W$  be a bounded open set in $ \mb{R}^{n} $, such that $ \overline{\Omega} \cap \overline{{W}} = \emptyset.$ Then the set 
     \begin{align*}
         \mc{D}_{b,q} ({W}) = \{ u_f - f ; \ f \in C_0^{\iy} \left({W} \times (-T,T)\right)\}
     \end{align*}
     is dense in $ \ch $.
    \end{theorem}
    \begin{proof}[Proof of Theorem \ref{determining b,q}]
   Let $\phi, \psi \in C_0^{\infty}(Q)$ be such that $\psi \equiv 1$ in $\text{supp}(\phi)$. Thanks to Theorem \ref{Runge for b and q}, for $k=1,2$ and $j \in \mathbb N$,  we can choose the exterior values $f_{j,k} \in C^\iy_0(W_k \times (-T,T))$ for the forward and adjoint problems in a way so that 
    \begin{align*}
        & \left( \mc{H}^s +\langle b_1, \nabla_{x} \rangle + q_1\right)u_{j,1} = \left( \mc{H}_{*}^s + \langle b_2, \nabla_{x} \rangle + q_2\right)u^*_{j,2} = 0, \qd \tn{ in } Q, \\
       & u_{j,1} (x,t) = f_{j,1}(x,t), \tn{ and } u^*_{j,2} (x,t) = f_{j,2}(x,t), \tn{ in } Q_e, \\
       & u_{j,1}\vert_{t\leq-T} = 0 \tn{ and } u^*_{j,2}\vert_{t\geq T} = 0,
    \end{align*}
    where 
    \begin{align}\label{uy}
     &u_{j,1} - f_{j,1} = \psi + r_{j,1}\ \text{and}\  u^*_{j,2} - f_{j,2} = \phi + r_{j,2},\end{align} with $r^j_k \in \ch$ and \begin{equation}\label{r00}\text{$ \|r_{j,k}\|_{\ch} \to 0 $ as $j \to \iy$ for $k = 1,2$.}\end{equation} Proceeding as before, using that $\left. \Lambda_{b_1,q_1}[f_{ji}]\right|_{W_2 \times (-T,T)} = \left. \Lambda_{b_2,q_2}[f_{j1}]\right|_{W_2 \times (-T,T)}$, we find from the integral identity Lemma \ref{int identity for b and q} that the following holds
    \begin{align}\label{int-ineq for determining b,q}
        \int_{Q} \langle(b_1 - b_2), \nabla_{x} u_{j,1} \rangle \ u^*_{j,2} \ \d{x} \d{t} +  \int_{Q}  (q_1 - q_2) u_{j,1} u^*_{j,2} \ \d{x} \d{t} = 0.
    \end{align}
     Now we analyze the term $ \int_{Q}\langle (b_1 - b_2),\nabla_{x} u_{j,1} \rangle \ u^*_{j,2} \ \d{x} \d{t} $. Using \eqref{uy} and the fact that $\nabla \psi \equiv 0$ on the support of $\phi$, we find
    \begin{align}\label{applying int-ineq for b and q}
      \ \int_{Q} \langle (b_1 - b_2), \nabla_{x} u_{j,1} \rangle \ u^*_{j,2} & = \int_{Q} \langle (b_1 - b_2),\nabla_{x} r_{j,1} \rangle \ r_{j,2} + \int_{Q} \langle (b_1 - b_2),\nabla_{x} r_{j,1} \rangle\ \phi.
    \end{align}
    We show that both the integrals in \eqref{applying int-ineq for b and q} converges to zero as $j \to \iy$. Arguing as in \eqref{bdo}, we observe 
    \begin{align*}
        \left\vert \int_{Q} \langle (b_1 - b_2),\nabla_{x} r_{j,1} \rangle \ r_{j,2} \right\vert 
        & \preceq \|r_{j,1}\|_{\ch} \|r_{j,2}\|_{\ch}
    \end{align*}
    and similarly 
    \begin{align*}
       \left\vert  \int_{Q} \langle (b_1 - b_2),\nabla_{x} r_{j,1} \rangle\ \phi\right\vert \preceq \|r_{j,2}\|_{\ch} \|\phi\|_{\ch}. 
    \end{align*}
    Therefore using \eqref{r00} we can conclude that \begin{equation}\label{lim1} \lim\limits_{j \to \iy} \int_{Q} \langle (b_1 - b_2), \nabla_{x} u_{j,1} \rangle \ u^*_{j,2} = 0. \end{equation} Using \eqref{lim1} in \eqref{int-ineq for determining b,q} we deduce 
     \begin{equation}\label{lim2} \lim\limits_{j \to \iy} \int_{Q} (q_1 - q_2) u_{j,1} u^*_{j,2} \ \d{x} \d{t} = 0.\end{equation}
       By following the proof of Theorem \ref{determining q}, we get that \eqref{lim2}  reduces to  
    \begin{align*}
        \int_{Q} (q_1 - q_2) \phi \ \d{x} \d{t} = 0.
    \end{align*}
    Since $\phi \in C_0^\iy(Q)$ is arbitrary, we thus have  $q_1 = q_2$ in $Q$. \\
    
    We now uniquely  determine the drift term.  We first choose some $\phi \in C_0^{\infty}(Q)$ and  then for $i=1,2$ and $k \in \{1,2,...,n\}$,  we  choose  $ \phi_{i,k} \in C_0^\infty (Q)$ in such a way that \begin{equation}\label{t0}\text{$\phi_{2,k} = \phi$ and $\phi_{1,k} = x_k$ on $\tn{supp}(\phi_{2,k})$.}\end{equation} Thanks to Theorem \ref{Runge for b and q}, we can find  $f^k_{j,l} \in C_0^\iy(W_l \times (-T,T))$ for $ k \in\{1,2,..,n\}, j \in \mb{N}$, $l = 1,2$ and solutions $\{u^k_{j, 1}, u^{*,k}_{j,2} \}$  to the forward and adjoint problems  associated to the exterior data $f^k_{j,l}$  such that
    \begin{align*}
        u^k_{j,1} - f^k_{j,1} = \phi_{1,k} + r^k_{j,1}, \qd \tn{ and } u^{*,k}_{j,2} - f^k_{j,2} = \phi_{2,k} + r^k_{j,2},
    \end{align*}
     where $\{r^{k}_{jl}\}$  satisfies the limiting condition in \eqref{r00} as $j \to \infty$.
   Now given $k \in \{1,2,...,n\}$,  using the identity in  Lemma \ref{int identity for b and q} along with the fact that $ q_1 = q_2 $ in $Q$, we have
    \begin{align*}
        0 = \int_Q \langle (b_1 - b_2), \nabla_{x} u^k_{1,j} \rangle \ u^{*,k}_{2,j} & = \int_Q \langle (b_1 - b_2),\nabla_{x} \phi_{1,k} \rangle \ \phi_{2,k} + \int_Q \langle (b_1 - b_2), \nabla_{x} \phi_{1,k} \rangle\ r^k_{j,2} \\
        & + \int_Q \langle (b_1 - b_2), \nabla_{x} r^k_{j,1} \rangle \ \phi_{2,k}+  \int_Q \langle (b_1 - b_2), \nabla_{x} r^k_{j,1}\rangle \ r^k_{j,2}.
    \end{align*}
    Using $\lim\limits_{j \to \iy} \|r^k_{j,l}\|_{\ch} = 0$ for $ l= 1,2$ and $k\in\{1,2,..,n\}$, we notice as before that
    \begin{align*}
        \lim_{j\to\iy} \int_Q \langle (b_1 - b_2), \nabla_{x} \phi_{1,k} \rangle \ r^k_{j,2} = \lim_{j\to\iy} \int_Q \langle (b_1 - b_2), \nabla_{x} r^k_{j,1} \rangle \ \phi_{2,k} = \lim_{j\to\iy} \int_Q \langle(b_1 - b_2),\nabla_{x} r^k_{j,1} \rangle\ r^k_{j,2}= 0
    \end{align*}
   which then implies
    \begin{equation}\label{t01}\int_Q \langle (b_1 - b_2), \nabla_{x} \phi_{1,k}  \rangle\ \phi_{2,k} \ \d x\d t  = 0.\end{equation}
    Now using \eqref{t0} in \eqref{t01}  we find
    \begin{align*}
        \int_Q (b_1 - b_2)_k(x,t) \phi(x,t) \ \d x\d t = 0, \ \tn{ for } k\in\{1,2,..,n\}.
    \end{align*}
Since $\phi \in C_0^\iy(Q)$ is arbitrary, we can thus  infer that  $b_1 = b_2$ in $Q$.  \end{proof}

\end{document}